\UseRawInputEncoding

\documentclass[11pt]{amsart}
    \usepackage{amsmath,amssymb,amsthm,amsfonts,amsrefs}
    \usepackage{mathtools}
    \usepackage{geometry}
    \usepackage{graphicx}
    \usepackage[all]{xy}
    \usepackage[colorlinks, linkcolor=blue, anchorcolor=blue,
            citecolor=blue]{hyperref}

    \usepackage{fancyhdr}
    \usepackage{mathrsfs}
    \newtheorem{definition}{Definition}[section]
    \newtheorem{theorem}[definition]{Theorem}
    \newtheorem{proposition}[definition]{Proposition}
    \newtheorem{lemma}[definition]{Lemma}
    \newtheorem{corollary}[definition]{Corollary}
    \newtheorem{example}[definition]{Example}
    \newtheorem{remark}[definition]{Remark}

\begin{document}

\title{Estimating Reeb Chords using Microlocal Sheaf Theory}
\date{}
\author{Wenyuan Li}
\address{Department of Mathematics, Northwestern University.}
\email{wenyuanli2023@u.northwestern.edu}
\maketitle

\begin{abstract}
    We show that for a closed Legendrian submanifold in a 1-jet bundle, if there is a sheaf with compact support, perfect stalk and singular support on that Legendrian, then (1)~the number of Reeb chords has a lower bound by half of the sum of Betti numbers of the Legendrian; (2)~the number of Reeb chords between the original Legendrian and its Hamiltonian pushoff has a lower bound in terms of Betti numbers when the oscillation norm of the Hamiltonian is small comparing with the length of Reeb chords. In the proof we develop a duality exact triangle and use the persistence structure (which comes from the action filtration) of microlocal sheaves.
\end{abstract}


\section{Introduction}

\subsection{Motivation and Background}
    A contact manifold $(Y, \xi)$ is a $(2n+1)$-manifold $Y$ together with a maximal nonintegrable hyperplane distribution $\xi$. Assume that there exists a 1-form $\alpha \in \Omega^1(Y)$ called a contact form such that $\xi = \ker\alpha$ (this is equivalent to saying that $\xi$ is coorientable). We define the Reeb vector field $R_\alpha$ to be the vector field satisfying
    $$\iota(R_\alpha)\alpha = 1, \,\,\, \iota(R_\alpha)d\alpha = 0.$$
    In a contact manifold $(Y, \ker\alpha)$, we consider Legendrian submanifolds $\Lambda \subset Y$ that are $n$-manifolds such that $T\Lambda \subset \xi|_\Lambda$. Reeb chords on $\Lambda$ are Reeb trajectories that both start and end on $\Lambda$.

    Estimating the number of Reeb chords has been a basic question on Legendrian submanifolds since Arnold's time \cite{Arnold}. When the contact manifold is $(Y, \xi) = (P \times \mathbb{R}_t, \ker(dt - \theta_P))$ where $(P, d\theta_P)$ is an exact symplectic manifold, one can pick the contact form $\alpha = dt - \theta_P$, and then the Reeb vector field is $\partial /\partial t$. For $\Lambda$ a closed Legendrian, consider the Lagrangian projection
    $$\pi_\text{Lag}: \Lambda \hookrightarrow P \times \mathbb{R} \rightarrow P.$$
    The Reeb chords between Legendrian submanifolds correspond bijectively to intersection points of their Lagrangian projections.

    For the number of self Reeb chords, when $n$ is even, there is a topological lower bound coming from $[\pi_\text{Lag}(\Lambda)] \cdot [\pi_\text{Lag}(\Lambda)] = \chi(\Lambda)/2$. Some flexibility results tell us that this is sometimes the best bound one can expect \cite{Fewdoublepoints}. However, under some extra assumptions, there are rigid behaviours beyond this purely algebraic topological bound.

    Using pseudo-holomorphic curves, a number of celebrated theorems on the number of self Reeb chords have been found \cite{Mohnkechord,Cieliechord,Rittertqft}. In particular, for Legendrians $\Lambda \subset P \times \mathbb{R}$, using Legendrian contact homology, works by Ekholm-Etnyre-Sullivan, Ekholm-Etnyre-Sabloff and Dimitroglou Rizell-Golovko \cite{EESorientation,EESduality,RizGolestimating} showed that, under some assumptions, the number of self Reeb chords is bounded from below by half of the sum of Betti numbers.

    Other than estimating self Reeb chords, estimating the number of Reeb chords between $\Lambda$ and some Hamiltonian pushoff $\varphi_H^1(\Lambda)$ has also been an important question. When the contact Hamiltonian comes from a symplectic Hamiltonian on $P$, this question reduces to the Arnold conjecture for (immersed) Lagrangian submanifolds $\pi_\text{Lag}(\Lambda)$ \cite{Arnold}.

    Many Legendrians can be displaced from themselves so that there are no Reeb chords between $\Lambda$ and $\varphi_H^1(\Lambda)$. However, when the norm of the Hamiltonian is sufficiently small, one can get estimates on the number of Reeb chords between $\Lambda$ and $\varphi_H^1(\Lambda)$ using pseudo-holomorphic curves \cite{Akaho,Chekanov,LiuCuplength,RizSenergy}. In particular, a recent result by Dimitroglou Rizell-Sullivan \cite{RizSpersist}, using the persistence of Legendrian contact homology, showed that for Legendrians $\Lambda \subset P \times \mathbb{R}$ satisfying certain assumptions, there is a lower bound of the number of Reeb chords in terms of Betti numbers, when the oscillation norm of the Hamiltonian is small comparing to the length of Reeb chords.

    On the other hand, in recent years microlocal sheaf theory has also shown to be a powerful tool in symplectic and contact geometry \cite{NadZas,Nad,Tamarkin1,Shendeconormal,STZ,STWZ,Gui,GuiGEthm,Guithree,Chiu,Arborealloose,CasalsZas}. In symplectic geometry, microlocal sheaf theory has already been used to show estimations on number of intersection points of Lagrangians (in particular, to solve non-displaceability problems) \cite{Tamarkin1,GKS,GS,Ike,AsanoIkeimmersion}.

    In contact geometry, conjecturally microlocal sheaves should be equivalent to certain representations of the Chekanov-Eliashberg dg algebra defined by pseudo-holomorphic curves, and for $\mathbb{R}^3_\text{std}$ it is known that a category of augmentations of the Chekanov-Eliashberg dg algebra is indeed a microlocal sheaf category consisting of microlocal rank 1 (i.e.~simple) objects \cite{AugSheaf} (in higher dimensions, some results can also be obtained \cite{CasalsMurphydga,AugSheafknot,AugSheafsurface}). Therefore, one may expect that we can use sheaf theory to study the number of Reeb chords.

    However, even though augmentations are sheaves, the isomorphisms are typically not explicit, and therefore it is nontrivial to identify homomorphisms of sheaves with Reeb chords geometrically. The main purpose of this paper is to give some understanding on the correspondence and estimate the number of Reeb chords using microlocal sheaf theory.

\subsection{Results and Methods}
    We will show the following theorems on Reeb chord estimations, using microlocal sheaf theory. In order to apply microlocal sheaf theory, we consider only contact manifolds $J^1(M) = T^*M \times \mathbb{R}$ where $\dim M = n$, which are contactomorphic to
    $$T^{*,\infty}_{\tau > 0}(M \times \mathbb{R}) = \{(x, t, \xi, \tau) \mid |\xi|^2 + |\tau|^2 = 1, \tau > 0\}.$$
    The contact form we choose will be $\alpha = dt - (\xi/\tau)dx$, and thus the Reeb vector field is $R_\alpha = \partial/\partial t$. Recall that the support of a complex of sheaves in $M \times \mathbb{R}$ is
    $$\mathrm{supp}(\mathscr{F}) = \overline{\bigcup_{j \in \mathbb{Z}}\{x \in M \times \mathbb{R} \mid (H^j\mathscr{F})_x \neq 0\}}.$$

\begin{remark}
    Throughout the paper, $Sh^b_\Lambda(M \times \mathbb{R})$ will represent the dg category of sheaves on $M \times \mathbb{R}$ with singular support in $\Lambda$ over $\Bbbk$ with perfect stalks, localized along acyclic objects (this is different from the notations in some literature \cite{GS,Guisurvey,Ike,AsanoIke}).
\end{remark}

    For self Reeb chords of a Legendrian $\Lambda \subset T^{*,\infty}_{\tau > 0}(M \times \mathbb{R})$, we have the following results analogous to Ekholm-Etnyre-Sullivan \cite{EESorientation}, Ekholm-Etnyre-Sabloff \cite{EESduality} and Dimitroglou Rizell-Golovko \cite{RizGolestimating}, where they showed the same inequality under the existence of a finite dimensional representation of the Chekanov-Eliashberg dg algebra, or Sabloff-Traynor \cite{Genfamily}, where they used generating families.

    A Legendrian submanifold $\Lambda \subset T^{*,\infty}_{\tau > 0}(M \times \mathbb{R})$ is called \emph{chord generic}, if the Lagrangian projection $\pi_\text{Lag}(\Lambda)$ is immersed with only transverse double points. Let $\mathcal{Q}(\Lambda)$ be the set of Reeb chords on $\Lambda$. Assume that the Maslov class $\mu(\Lambda) = 0$. Then there is a grading on Reeb chords of $\Lambda$ (where the degree is given by the Conley-Zehnder index; see Section \ref{grade}). Let $\mathcal{Q}_i(\Lambda)$ be the set of degree $i$ Reeb chords on $\Lambda$.

\begin{theorem}\label{puresheaf}
    Let $M$ be orientable, $\Lambda \subset T^{*,\infty}_{\tau > 0}(M \times \mathbb{R})$ be a closed chord generic Legendrian submanifold and $\Bbbk$ be a field (and $\Lambda$ is spin when $\Bbbk\neq \mathbb{Z}/2\mathbb{Z}$). If there exists a $\Bbbk$-coefficient microlocal rank $r$ sheaf $\mathscr{F} \in Sh^b_\Lambda(M \times \mathbb{R})$ with perfect stalks such that $\mathrm{supp}(\mathscr{F})$ is compact, then
    $$|\mathcal{Q}_i(\Lambda)| + |\mathcal{Q}_{n-i}(\Lambda)| \geq b_i(\Lambda; \Bbbk).$$
    In particular, the number of Reeb chords
    $$|\mathcal{Q}(\Lambda)| \geq \frac{1}{2}\sum_{i=0}^nb_i(\Lambda; \Bbbk).$$
    Here $b_i(\Lambda; \Bbbk) = \dim_\Bbbk H^i(\Lambda; \Bbbk)$.
\end{theorem}

\begin{theorem}\label{mixedsheaf}
    Let $M$ be orientable, $\Lambda \subset T^{*,\infty}_{\tau > 0}(M \times \mathbb{R})$ be a closed chord generic Legendrian submanifold and $\Bbbk$ be a field (and $\Lambda$ is spin when $\Bbbk\neq \mathbb{Z}/2\mathbb{Z}$). If there exists a $\Bbbk$-coefficient sheaf $\mathscr{F} \in Sh^b_\Lambda(M \times \mathbb{R})$ with perfect stalks such that $\mathrm{supp}(\mathscr{F})$ is compact, then
    $$|\mathcal{Q}(\Lambda)| \geq \frac{1}{2}\sum_{i=0}^nb_i(\Lambda; \Bbbk).$$
    Here $b_i(\Lambda; \Bbbk) = \dim_\Bbbk H^i(\Lambda; \Bbbk)$.
\end{theorem}

\begin{remark}
    The condition that $\mathrm{supp}(\mathscr{F})$ is compact should be thought of as an analogue of the linear at infinity condition on generating families \cite{Genfamily}. See Appendix \ref{appendix}. If we drop this condition, then there will be counterexamples to te bounds. Consider the positive conormal $\nu^{*,\infty}_{M,\tau > 0}(M \times \mathbb{R}) \subset T^{*,\infty}_{\tau > 0}(M \times \mathbb{R})$ (which is just the zero section $M \subset J^1(M)$). Then $\Bbbk_{M \times [0,+\infty)}$ has the prescribed singular support. However that Legendrian has no Reeb chords.
\end{remark}

\begin{remark}
    When there is a sheaf $\mathscr{F} \in Sh^b_\Lambda(M \times \mathbb{R})$ with perfect stalks, then one can show that \cite{Gui} necessarily the Maslov class $\mu(\Lambda) = 0$. However this condition is not necessary to get estimates on the number of Reeb chords. In general, one can consider the triangulated orbit category $Sh^b_\Lambda(M \times \mathbb{R})_{/[1]}$ consisting of sheaves of $1$-cyclic complexes (see \cite{Orbitcat} and \cite[Section 3]{Gui}). When there is a sheaf $\mathscr{F} \in Sh^b_\Lambda(M \times \mathbb{R})_{/[1]}$, then we can still prove that
    $$|\mathcal{Q}(\Lambda)| \geq \frac{1}{2}\sum_{i=0}^nb_i(\Lambda; \Bbbk),$$
    but we do not work out the details here.
\end{remark}

\begin{remark}\label{hori}
    In \cite{EESorientation,EESduality,RizGolestimating}, they imposed the condition that the Legendrian $\Lambda$ is horizontally displaceable, meaning that there exists a Hamiltonian isotopy $\varphi_H^s\,(s \in I)$ such that there are no Reeb chords between $\Lambda$ and $\varphi_H^1(\Lambda)$. In Section \ref{horizontal}, we show that if $\Lambda$ is horizontally displaceable, then any $\mathscr{F} \in Sh^b_\Lambda(M \times \mathbb{R})$ necessarily has compact support.

    However, there are Legendrians that are not horizontally displaceable but admit sheaves with compact supports, for example, the Legendrian in \cite[Remark 1.9]{Genfamily}. Moreover, we show that there also exist Legendrians that are not horizontally displaceable, admit no generating families linear at infinity, but admit sheaves with compact support. See Appendix \ref{appendix}. This means that our theorems work in a slightly more general setting.
\end{remark}

\begin{remark}
    We know that $r$ dimensional representations of the Chekanov-Eliashberg dg algebra should be equivalent to microlocal rank $r$ sheaves (see \cite{RepSheaf}). Therefore, Theorem \ref{puresheaf} is just an analogue of \cite{EESorientation,EESduality,RizGolestimating}. However, Theorem \ref{mixedsheaf} has no direct analogue in the literature to our knowledge.
\end{remark}

    For Reeb chords between a Legendrian $\Lambda$ and its Hamiltonian pushoff $\varphi_H^1(\Lambda)$, we have the following results, analogous to Dimitroglou Rizell and Sullivan \cite{RizSpersist}. Define the \emph{oscillation norm} of the Hamiltonian to be
    $$\|H\|_\text{osc}^\Lambda = \int_0^1\left(\max_{x\in \varphi_H^s(\Lambda)}H_s - \min_{x\in \varphi_H^s(\Lambda)}H_s\right)\,ds.$$
    Denote by $l(\gamma)$ the length of a Reeb chord $\gamma$. Assume that the Maslov class $\mu(\Lambda) = 0$, which ensures the existence of a grading on chords of $\Lambda$ (see Section \ref{grade}), and let
    $$c_i(\Lambda) = \min\{l(\gamma) \mid \gamma \text{ is a Reeb chord,}\,\deg(\gamma) = i \text{ or } n-i\}.$$
    Order them so that $c_{j_0}(\Lambda) \geq c_{j_1}(\Lambda) \geq \cdots \geq c_{j_n}(\Lambda)$.

\begin{theorem}\label{displaceable}
    Let $M$ be orientable, $\Lambda \subset T^{*,\infty}_{\tau > 0}(M \times \mathbb{R})$ be a closed Legendrian submanifold of dimension $n$, and $\Bbbk$ be a field ($\Lambda$ is spin if $\mathrm{\Bbbk} \neq \mathbb{Z}/2\mathbb{Z}$). Suppose there exists a $\Bbbk$-coefficient pure sheaf $\mathscr{F} \in Sh_{\Lambda}^b(M \times \mathbb{R})$ with perfect stalks such that $\mathrm{supp}(\mathscr{F})$ is compact. Let $H_s$ be any Hamiltonian in $T^{*,\infty}_{\tau > 0}(M \times \mathbb{R})$ such that for some $0 \leq k \leq n$,
    $$\|H\|_\text{osc}^\Lambda < c_{j_k}(\Lambda)$$
    and $\varphi_H^1(\Lambda)$ is transverse to the Reeb flow applied to $\Lambda$. Then the number of Reeb chords between $\Lambda$ and $\varphi_H^1(\Lambda)$ is
    $$\mathcal{Q}(\Lambda, \varphi_H^1(\Lambda)) \geq \sum_{i=0}^k b_{j_i}(\Lambda; \Bbbk).$$
    Here $b_j(\Lambda; \Bbbk) = \dim H^j(\Lambda; \Bbbk)$.
\end{theorem}
\begin{remark}
    It is shown \cite{RizSpersist} that this bound is sharp for Legendrian unknotted spheres with a single Reeb chord.
\end{remark}
\begin{remark}
    Dimitroglou Rizell-Sullivan considered \cite{RizSpersist} Legendrians that only admit augmentations over a subalgebra of the Chekanov-Eliashberg dg algebra $\mathcal{A}^l(\Lambda) \subset \mathcal{A}(\Lambda)$. On the sheaf side, Asano-Ike \cite{AsanoIkeimmersion} proved the above inequality for any Legendrian (including loose Legendrians) when $\|H\|_{osc}^\Lambda \leq \min\{l(\gamma) \mid \gamma \text{ is a Reeb chord}\}$ using a sheaf $\mathscr{F}\in Sh^b_{\Lambda_q \cup \Lambda_r}(M \times \mathbb{R} \times (0,l))_{/[1]}$ (see Definition \ref{translation}) which always exists. Hence we expect that by using Asano-Ike's technique \cite{AsanoIkeimmersion}, one will get analogous results.
\end{remark}

    We are also able to recover the nonsqueezing result of Legendrians admitting sheaves into a stablized/loose Legendrian \cite{RizSpersist} as a byproduct. For the definition of a stablized or loose Legendrian submanifold, see \cite{Loose} or \cite[Chapter 7]{CE}.

\begin{definition}[Dimitroglou Rizell-Sullivan \cite{RizSpersist}]\label{squeezedef}
    Let $U \subset P \times \mathbb{R}$ be an open subset with $H_n(U; \mathbb{Z}/2\mathbb{Z}) \neq 0$. Then a Legendrian submanifold $\Lambda \subset P \times \mathbb{R}$ can be squeezed into $U$ if there is a Legendrian isotopy $\Lambda_t$ with $\Lambda_0 = \Lambda$ and
    $$\Lambda_1 \subset U, \,\,\, [\Lambda_1] \neq 0 \in H_n(U; \mathbb{Z}/2\mathbb{Z}).$$
\end{definition}

\begin{theorem}\label{squeeze}
    Let $\Lambda_\text{loose} \subset T^{*,\infty}_{\tau > 0}(\mathbb{R}^{n+1})$ be a closed stablized/loose Legendrian, and $\Lambda \subset T^{*,\infty}_{\tau > 0}(\mathbb{R}^{n+1})$ be a Legendrian so that there exists $\mathscr{F} \in Sh^b_\Lambda(\mathbb{R}^{n+1})$ with perfect stalks such that $\mathrm{supp}(\mathscr{F})$ is compact and the microstalk has odd Euler characteristic. Then $\Lambda$ cannot be squeezed into a contact tubular neighbourhood of $\Lambda_\text{loose}$.
\end{theorem}

    There are two main difficulties to prove these results using microlocal sheaf theory. Firstly, we need to understand how to directly see the Reeb chords from the homomorphism of sheaves $Hom(\mathscr{F, F})$. Secondly, we need to have the algebraic results, for example a duality and exact triangle, that gives bounds on the rank of $Hom(\mathscr{F, F})$.

\subsubsection{Relating Reeb chords to sheaves}
    What we do to solve the first problem is to add in an extra $\mathbb{R}$ factor corresponding to the $\mathbb{R}$-filtration on Reeb chords and extend the sheaf from $M \times \mathbb{R}$ to $M \times \mathbb{R}^2$ to see the all the Reeb chords explicitly in the extra $\mathbb{R}$ factor. The author learned the idea from the lecture notes of Shende, but this goes back to Tamarkin \cite[Chapter 3]{Tamarkin1}, followed by Guillermou-Schapira \cite[Section 3 \& 4]{GS}, Guillermou \cite[Section 13 \& 16]{Gui}, Ike \cite{Ike}, Asano-Ike \cite{AsanoIkeimmersion}, and very recently Kuo \cite[Section 3]{Kuo}.

\begin{definition}\label{translation}
    Let $q \colon M \times \mathbb{R}^2 \rightarrow M \times \mathbb{R}$ be $q(x, t, u) = (x, t)$ and $r\colon M \times \mathbb{R}^2 \rightarrow M \times \mathbb{R}$ be $r(x, t, u) = (x, t-u)$. For a Legendrian submanifold $\Lambda \subset J^1(M) \cong T^{*,\infty}_{\tau > 0}(M \times \mathbb{R})$, let
    \[\begin{split}
    \Lambda_q &= \{(x, \xi, t, \tau, u, 0) \mid (x, \xi, t, \tau) \in \Lambda\}, \\
    \Lambda_r &= \{(x, \xi, t + u, \tau, u, -\tau) \mid (x, \xi, t, \tau \in \Lambda)\}.
    \end{split}\]
    For a sheaf $\mathscr{F} \in Sh^b(M \times \mathbb{R})$, let
    $$\mathscr{F}_q = q^{-1}\mathscr{F}, \,\,\, \mathscr{F}_r = r^{-1}\mathscr{F}.$$
\end{definition}

    In the definition $\Lambda_q$ (resp.~$\mathscr{F}_q$) is the movie of $\Lambda$ (resp.~$\mathscr{F}$) under the identity contact flow, while $\Lambda_r$ (resp.~$\mathscr{F}_r$) is the movie of $\Lambda$ (resp.~$\mathscr{F}$) under the vertical translation $T_l\,(l \in \mathbb{R})$ defined by the Reeb flow. As we isotope the Legendrian $\Lambda$ via the Reeb flow to $T_l(\Lambda)$, the lengths of Reeb chords from $\Lambda$ to $T_l(\Lambda)$ coming from self chords of $\Lambda$ will decrease. The time when the length of some chord shrinks to zero will be detected by the microlocal behaviour of $\mathscr{H}om(\mathscr{F}_q, \mathscr{F}_r)$.

\begin{definition}\label{defhom}
    Let the projection to the last factor $M \times \mathbb{R}^2 \rightarrow \mathbb{R}, \,(x, t, u) \mapsto u$ be $u$. For $\Lambda \subset T^{*,\infty}_{\tau > 0}(M \times \mathbb{R})$ and $\mathscr{F, G} \in Sh^b_\Lambda(M \times \mathbb{R})$, let
    \[\begin{split}
    Hom_-(\mathscr{F, G}) = \Gamma(u^{-1}([0,+\infty)), \mathscr{H}om(\mathscr{F}_q, \mathscr{G}_r)),\\
    Hom_+(\mathscr{F, G}) = \Gamma(u^{-1}((0,+\infty)), \mathscr{H}om(\mathscr{F}_q, \mathscr{G}_r)).
    \end{split}\]
\end{definition}
\begin{remark}
    For readers who are familiar with generating families of Legendrians \cite{ViterboGen,TrayGen,Genfamily}, they may notice that this definition is similar to the generating family homology and cohomology, where the extra $\mathbb{R}$-factor encodes the value of the difference function. Our definition is partly inspired by that.
\end{remark}
\begin{remark}\label{tamarkin}
    For those who are familiar with the language of Tamarkin categories \cite{Tamarkin1,GS,Ike,AsanoIkeimmersion}, they may realize that this can equivalently be phrased in terms of the internal Hom in Tamarkin category, where there are also two copies of $\mathbb{R}$-factors and the internal Hom is defined using the addition map, which is equivalent to translating one of the sheaves by the Reeb flow. Our definition is also a reformulation of that.
\end{remark}

    In Section \ref{reebchordsection}, we will provide a systematic way to relate Reeb chords to the positive homomorphism of sheaves $Hom_+(\mathscr{F, F})$. The idea is similar to relating singular (co)homology with Morse critical points. In particular, the following Morse inequality holds.

\begin{theorem}\label{reebinequality}
    For $\Lambda \subset T^{*,\infty}_{\tau > 0}(M \times \mathbb{R})$ a chord generic Legendrian and $\mathscr{F} \in Sh^b_\Lambda(M \times \mathbb{R})$ a microlocal rank $r$ sheaf such that $\mathrm{supp}(\mathscr{F})$ is compact. Then for any $k \in \mathbb{Z}$,
    $$r^2\sum_{j\leq k}(-1)^{k-j}|\mathcal{Q}_j(\Lambda)| \geq \sum_{j\leq k}(-1)^{k-j}\dim H^jHom_+(\mathscr{F, F}).$$
    In particular, for any $j\in \mathbb{Z}$, $r^2|\mathcal{Q}_j(\Lambda)| \geq \dim H^jHom_+(\mathscr{F, F})$.
\end{theorem}

\subsubsection{Duality exact triangle}
    We prove a duality exact triangle in microlocal sheaf theory, which is parallel to the duality exact triangle in Legendrian contact homologies \cite{Sabduality,EESduality}, in order to deduce Theorem \ref{puresheaf} and \ref{mixedsheaf}. While microlocal sheaves are equivalent to representations of Chekanov-Eliashberg dg algebras, our proof is purely sheaf theoretic.

\begin{theorem}[Sabloff Duality]\label{duality}
    Let $M$ be orientable. For $\Lambda \subset T^{*,\infty}_{\tau > 0}(M \times \mathbb{R})$ and $\mathscr{F, G} \in Sh^b_\Lambda(M \times \mathbb{R})$ with perfect stalks such that $\mathrm{supp}(\mathscr{F}), \mathrm{supp}(\mathscr{G})$ are compact,
    $$Hom_+(\mathscr{F, G}) \simeq D'Hom_-(\mathscr{G, F})[-n-1],$$
    where $D'\mathscr{F} = \mathscr{H}om(\mathscr{F}, \Bbbk_{M \times \mathbb{R}})$.
\end{theorem}

\begin{theorem}[Sabloff-Sato Exact Triangle]\label{exactseq}
    For $\Lambda \subset T^{*,\infty}_{\tau > 0}(M \times \mathbb{R})$ and $\mathscr{F} \in Sh^b_\Lambda(M \times \mathbb{R})$ a microlocal rank $r$ sheaf with perfect stalks such that $\mathrm{supp}(\mathscr{F})$ is compact, we have an exact triangle
    $$Hom_-(\mathscr{F, F}) \rightarrow Hom_+(\mathscr{F, F}) \rightarrow C^*(\Lambda; \Bbbk^{r^2}) \xrightarrow{+1}.$$
\end{theorem}

\begin{remark}
    As is shown in the name, this exact triangle is coming from Sato's exact triangle which is well known in microlocal sheaf theory. See \cite[Equation 2.17]{Gui} or \cite[Equation 1.3.5]{Guisurvey}. One can find results of a similar flavour in \cite[Section 4.3 \& 4.4]{Ike}.
\end{remark}
\begin{remark}\label{relativecy}
    Theorem \ref{exactseq} also holds for different sheaves $\mathscr{F}$ and $\mathscr{G}$ (though the third term may be replaced by cochains on $\Lambda$ twisted by a local system). In fact, we conjecture that the duality and exact sequence fit into a commutative diagram. Namely, let $Sh^b_{\Lambda,+}(M \times \mathbb{R})_0$ (resp.~$Sh^b_{\Lambda,-}(M \times \mathbb{R})_0$) be the subcategory consisting only of sheaves with perfect stalks and compact supports with morphisms being $Hom_+(-, -)$ (resp.~$Hom_-(-, -)$). Then
    \[\xymatrix{
    Sh^b_{\Lambda,+}(M \times \mathbb{R})_0[n] \ar[r]^{\hspace{10pt}m_\Lambda[n]} \ar[d] & m_\Lambda^*Loc^b_\Lambda(\Lambda)[n] \ar[r] \ar[d]^{PD} & Sh^b_{\Lambda,-}(M \times \mathbb{R})_0[n+1] \ar[d] \\
    D'Sh^b_{\Lambda,-}(M \times \mathbb{R})_0[-1] \ar[r] & D'(m_\Lambda^*Loc^b_\Lambda(\Lambda)) \ar[r]^{D'm_\Lambda\hspace{10pt}} & D'Sh^b_{\Lambda,+}(M \times \mathbb{R})_0,
    }\]
    which should suggest that $m_\Lambda\colon Sh^b_{\Lambda,+}(M \times \mathbb{R})_0 \rightarrow Loc^b_\Lambda(\Lambda)$ is a relative right Calabi-Yau functor \cite{relativeCY}.
\end{remark}

    We also show that our definition of $Hom_+(-,-)$ coincides with the ordinary $Hom(-,-)$. Since the augmentation category $\mathcal{A}ug_+$ is equivalent to the microlocal sheaf category with morphism space $Hom(-,-)$ \cite{AugSheaf}, this tells us that $Hom_+(-,-)$ is indeed the correct analogue of morphisms in $\mathcal{A}ug_+$.

\begin{theorem}\label{computehom}
    For a Legendrian $\Lambda \subset T^{*,\infty}_{\tau > 0}(M \times \mathbb{R})$ and sheaves $\mathscr{F, G} \in Sh^b_\Lambda(M \times \mathbb{R})$ with perfect stalks such that $\mathrm{supp}(\mathscr{F}), \mathrm{supp}(\mathscr{G})$ are compact,
    $$Hom_-(\mathscr{F, G}) \simeq \Gamma(D'\mathscr{F} \otimes \mathscr{G}), \,\,\, Hom_+(\mathscr{F, G}) \simeq Hom(\mathscr{F, G}).$$
\end{theorem}

\subsubsection{Persistence structure}
    For more careful analysis on the differentials of the chain complexes so as to prove Theorem \ref{displaceable} and \ref{squeeze}, we will consider the extra $\mathbb{R}$-factor corresponding to the action filtration of Reeb chords. Indeed, we should not only consider numerical invariants, but construct a persistence module $\mathscr{H}om_{(-\infty,+\infty)}$ and study the persistence structure, as in \cite{PolShe,UsherZpersist,KSpersist,Zhang}, and in particular following Dimitroglou Rizell-Sullivan \cite{RizSpersist} in Floer theory and Asano-Ike \cite{AsanoIke} in sheaf theory.

\begin{definition}\label{persistmod}
    For sheaves $\mathscr{F, G} \in Sh^b(M \times \mathbb{R})$, let
    $$\mathscr{H}om_{(-\infty,+\infty)}(\mathscr{F, G}) = u_*\mathscr{H}om(\mathscr{F}_q, \mathscr{G}_r).$$
\end{definition}

    It turns out that when $\mathscr{F}$ and $\mathscr{G}$ are constructible sheaves with perfect stalks and compact supports, the sheaf $\mathscr{H}om_{(-\infty,+\infty)}(\mathscr{F, G})$ on $\mathbb{R}$ has a canonical decomposition
    $$\mathscr{H}om_{(-\infty,+\infty)}(\mathscr{F, G}) \simeq \bigoplus_{\alpha \in I}\Bbbk^{n_\alpha}_{(a_\alpha, b_\alpha]}[d_\alpha],$$
    and can be viewed as a persistence module on $\mathbb{R}$. In addition, the endpoints of the intervals $(a_\alpha, b_\alpha]$ are exactly lengths of Reeb chords.

    The difference of a family of persistence modules is measured by the interleaving distance $d$. In the setting of sheaf theory the relation between persistence distance and Hamiltonian has been studied by Asano-Ike in \cite{AsanoIke}. Here, we apply their result and get the following critical estimate. Since the Reeb flow does not affect the number of Reeb chords, we will consider a distance $\overline{d}$ invariant under the Reeb flow.

\begin{theorem}\label{persistcontinue}
    Let $\Lambda \subset T^{*,\infty}_{\tau > 0}(M \times \mathbb{R})$ be a closed Legendrian, $H$ be a Hamiltonian on $T^{*,\infty}_{\tau > 0}(M \times \mathbb{R})$ and $\Phi_H^s\,(s \in I)$ be the equivalence functor induced by the Hamiltonian. Then for $\mathscr{F, G} \in Sh^b_\Lambda(M \times \mathbb{R})$ with perfect stalks and compact supports,
    $$\overline{d}(\mathscr{H}om_{(-\infty,+\infty)}(\mathscr{F, G}), \mathscr{H}om_{(-\infty,+\infty)}(\mathscr{F}, \Phi_H^1(\mathscr{G}))) \leq \|H\|_\text{osc}^\Lambda.$$
\end{theorem}

    Combining all these ingredients, we are able to get the results on Reeb chord estimations stated at the beginning of this section.

\begin{remark}
    At the end of the introduction, we briefly explain the relation between this paper and other results in microlocal sheaf theory. As explained in Remark \ref{tamarkin}, our construction is essentially equivalent to the approach using the Tamarkin category \cite{Tamarkin1,GS,Ike,AsanoIke,AsanoIkeimmersion}. Our main contribution in this paper may be the duality and exact triangle. While the Sato-Sabloff exact triangle Theorem \ref{exactseq} may be extracted from \cite[Section 11.3]{Guisurvey} and \cite[Section 4.3]{Ike}, it may be hard to directly find a clear statement of Theorem \ref{exactseq} there. Moreover, the result in Theorem \ref{duality} also seems to be new.

    Since the paper appeared on arXiv, a number of the main results have been improved. For instance, Sato-Sabloff exact triangle has been generalized to Legendrians $\Lambda \subset T^{*,\infty}N$ and compactly supported sheaves that do not necessarily have perfect stalks \cite{KuoLiSpherical}, and Sabloff duality has also been generalized to Legendrians $\Lambda \subset T^{*,\infty}N$ and compactly supported sheaves with perfect stalks \cite{KuoLiSpherical}. Furthermore, these results are strengthened to a strong smooth relative Calabi-Yau structure on $m_\Lambda^l \colon Loc_\Lambda(\Lambda) \to Sh_{\Lambda,+}(N)_0$, which induces a proper relative Calabi-Yau structure on the subcategories with perfect stalks $m_\Lambda\colon Sh^b_{\Lambda,+}(N)_0 \rightarrow Loc^b_\Lambda(\Lambda)$ \cite{KuoLiCY}, confirming the conjecture in Remark \ref{relativecy}.
\end{remark}

\subsection{Organization of the Paper}
    Section \ref{prelimcontact} reviews basic contact geometry, genericity conditions and gradings of Reeb chords. Section \ref{prelimsheaf} reviews basic sheaf theory, singular supports, microlocal Morse theory, microlocalization and how the sheaf category changes with respect to certain operations. In Section \ref{dualityexact} we define $Hom_\pm(-, -)$ and prove Theorem \ref{duality}, \ref{exactseq} and \ref{computehom}. In Section \ref{persist} we review basic concepts in persistence modules, Asano-Ike's results and use that to prove Theorem \ref{persistcontinue}. In Section \ref{reebchordsection} we relate Reeb chords with homomorphisms of sheaves. In particular we prove Theorem \ref{reebinequality}, and finish the proof of Theorem \ref{puresheaf}, \ref{mixedsheaf} and \ref{displaceable}. Finally in Section \ref{loosesection} we prove Theorem \ref{squeeze}.

\subsection*{Acknowledgements}
    I would like to thank my advisors Emmy Murphy and Eric Zaslow for plenty of helpful discussions and comments, in particular Emmy Murphy for suggesting the topic on the estimation of self Reeb chords and explaining to me the results in generating families and Eric Zaslow for discussion on relative Calabi-Yau functors in Remark \ref{relativecy}. I am also grateful to Vivek Shende for his online lecture notes on microlocal sheaf theory. Finally I would thank Yuichi Ike, Joshua Sabloff and the anonymous referee for helpful comments and suggestions.

\section{Preliminaries in Contact Topology}\label{prelimcontact}

\subsection{Jet Bundles and Cotangent Bundles}
    In this section we explain the contact form and Reeb vector field that we are going to work with, and in particular the contactomorphism $J^1(M) \xrightarrow{\sim} T^{*,\infty}_{\tau > 0}(M \times \mathbb{R})$. We also explain the contact Hamiltonians and their vector fields with respect to the specific contact form.

    The 1-jet bundle $J^1(M) = T^*M \times \mathbb{R}$. Consider local coordinates $(x_0, \xi_0, t_0) \in T^*M \times \mathbb{R}$, where $x_0$ is the coordinate on $M$, $\xi_0$ is the coordinate on the fiber of $T^*M$ and $t_0$ is the coordinate on $\mathbb{R}$. The contact structure given by $\ker(dt_0 - \xi_0dx_0)$. We choose the contact form to be $\alpha_0 = dt_0 - \xi_0dx_0$. Now consider
    \[\begin{array}{ccc}
    T^*_{\tau > 0}(M \times \mathbb{R}) & \rightarrow & J^1(M), \\
    (x, \xi, t, \tau) & \mapsto & (x, \xi/\tau, t).
    \end{array}\]
    After taking the quotient of $T^*_{\tau > 0}(M \times \mathbb{R})$ by the dilation $(x, \xi, t, \tau) \mapsto (x, a\xi, t, a\tau)$ by $a \in \mathbb{R}_{>0}$, we get a diffeomorphism
    $$T^{*,\infty}_{\tau > 0}(M \times \mathbb{R}) \xrightarrow{\sim} J^1(M)$$
    where $T^{*,\infty}_{\tau > 0}(M \times \mathbb{R}) = \{(x, \xi, t, \tau)\mid |\xi|^2 + |\tau|^2 = 1, \tau > 0\} \cong T^*_{\tau > 0}(M \times \mathbb{R})/\mathbb{R}_{>0}$. (If you consider the standard Liouville flow on $T^*(M \times \mathbb{R})$ and think of contact manifolds in the way that each contact form corresponds to a specific choice of a hypersurface transverse to the Liouville vector field, maybe it's better think of $T^{*,\infty}_{\tau > 0}(M \times \mathbb{R})$ as $\{(x, \xi, t, \tau)\mid \tau \equiv 1\}$.) There is a natural contact structure on $T^{*,\infty}_{\tau > 0}(M \times \mathbb{R})$ given by restriction of the symplectic structure on $T^{*}(M \times \mathbb{R})$
    $$\ker(\tau dt - \xi dx).$$
    Then one can check that $T^{*,\infty}_{\tau > 0}(M \times \mathbb{R})$ and $J^1(M)$ are contactomorphic through that map defined above.

    Under the contactomorphism, the contact form $\alpha_0 = dt_0 - \xi_0dx_0$ is mapped to
    $$\alpha = dt - (\xi/\tau)dx,$$
    and the Reeb vector field $R_{\alpha_0} = \partial/\partial t_0$ is mapped to
    $$R_\alpha = \frac{\partial}{\partial t}.$$
    This contact form and Reeb vector field are the ones we will be dealing with in the paper.

\begin{remark}
    In the cotangent bundle $T^{*,\infty}(M \times \mathbb{R})$, the Reeb vector field that people are more familiar with may be the vector field producing the geodesic flow. The Reeb vector field we work with here is different because the contact form $\alpha = dt - (\xi/\tau)dx$ is different from the canonical one $\tau dt - \xi dx$. Indeed the contactomorphism we write down does not preserve the canonical contact forms on both sides.
\end{remark}

    Now we consider the correspondence between contact Hamiltonians and contact vector fields determined by this contact form $\alpha = dt - (\xi/\tau)dx$. Given $H \in C^\infty(T^{*,\infty}_{\tau > 0}(M \times \mathbb{R}))$, the corresponding contact vector field $X_H$ is defined by \cite{Geiges}
    $$H = \alpha(X_H), \,\,\, \iota(X_H)d\alpha = dH(R_\alpha) \alpha - dH.$$
    We claim that this contact Hamiltonian can be lifted to a homogeneous symplectic Hamiltonian on $T^*_{\tau > 0}(M \times \mathbb{R})$ in the following way. Let
    $$\widehat H(x, \xi, t, \tau) = \tau H(x, \xi/\tau, t).$$
    Its corresponding symplectic Hamiltonian vector field is defined by
    $$\iota(X_{\widehat H})\omega = -d\widehat H,$$
    where $\omega = d(\tau dt - \xi dx) = d(\tau \alpha)$. By elementary calculation, one will find that the projection $X_{\widehat H}$ onto the hyperplane $\tau = 1$ is $X_H$. Therefore, we will just study the homogeneous Hamiltonian $\widehat H$ (since in microlocal sheaf theory this will be more natural). In particular, one can define the movie of a subset $\widehat \Lambda \subset T^*_{\tau > 0}(M \times \mathbb{R})$ under the Hamiltonian isotopy $\varphi_{\widehat H}^s\,(s\in I)$ as
    $$\widehat \Lambda_H = \{(x, \xi, t, \tau, s, \sigma) \mid (x, \xi, t, \tau) = \varphi_{\hat H}^s(x_0, \xi_0, t_0, \tau_0), \sigma = -\widehat H \circ \varphi_{\widehat H}^s(x_0, \xi_0/\tau_0, t_0)\}.$$
    This is an exact conical Lagrangian submanifold in $T^*_{\tau > 0}(M \times \mathbb{R} \times I)$.

\subsection{Genericity Assumptions}
    In this section we introduce the notions of chord generic Legendrian submanifolds and admissible Legendrian isotopies. They are generic under $C^1$-topology in the space of embeddings/isotopies.

\begin{definition}
    Let $\Lambda \subset J^1(M)$ be a Legendrian submanifold. $\Lambda$ is called chord generic if the Lagrangian projection
    $$\pi_\text{Lag}: \Lambda \rightarrow T^*M$$
    is a Lagrangian immersion with only transverse double points.
\end{definition}
\begin{lemma}[Ekholm-Etnyre-Sullivan, \cite{EESR2n+1}*{Lemma 3.5}]
    Let $\Lambda$ be a Legendrian submanifold. Then for any $\epsilon > 0$ there is a chord generic Legendrian submanifold $\Lambda_\epsilon$ that is $\epsilon$-close to $\Lambda$ in the $C^1$-topology.
\end{lemma}
\begin{remark}
    In fact, being $\epsilon$-close in the $C^1$-topology implies that $\Lambda$ is Hamiltonian isotopic to $\Lambda_\epsilon$ by the Legendrian neighbourhood theorem. In addition the $C^0$-norm of the Hamiltonian isotopy can also be smaller than $\epsilon$.
\end{remark}

    By Legendrian isotopy extension theorem, any Legendrian isotopy can be realized as an ambient Hamiltonian isotopy.

\begin{definition}
    Let $n \geq 2$, $\Lambda \subset J^1(M)$ be a Legendrian submanifold and $H \in C^\infty(J^1(M))$ a contact Hamiltonian. Then the Legendrian isotopy $\Lambda_s = \varphi_H^s(\Lambda) \,(s \in I)$ is admissible if there are $s_1,\dots, s_k \in I$ such that
    \begin{enumerate}
    \item for $s \neq s_1, \dots, s_k$, $\Lambda_s$ is a chord generic Legendrian;
    \item for $s \in (s_i-\epsilon, s_i+\epsilon)$ where $\epsilon>0$ is sufficiently small, $\Lambda_s$ is still chord generic away from some contact ball $U \in J^1(M)$, and in the contact ball $U \simeq \mathbb{R}^{2n+1}$,
    $$\Lambda_t \cap U \simeq \left(\{(x,0,0) \mid x \in \mathbb{R}\} \times L_1\right) \cup \left(\{(x,3x^2+s,x^3+sx)\mid x \in \mathbb{R}\} \times L_2\right)$$
    such that $L_1 \pitchfork L_2$ are transverse Lagrangian subspaces in $\mathbb{R}^{2n-2}$.
    \end{enumerate}
\end{definition}
\begin{lemma}[Ekholm-Etnyre-Sullivan \cite{EESR2n+1}*{Lemma 3.6}]\label{admissible}
    Let $\Lambda_s\,(s\in I)$ be a Legendrian isotopy consisting of chord generic Legendrians connecting $\Lambda_1$ and $\Lambda_1$. Then for any $\epsilon > 0$ there exists an admissible Legendrian isotopy connecting $\Lambda_0$ and $\Lambda_1$ that is $\epsilon$-close to $\Lambda_s\,(s\in I)$ in the $C^1$-topology.
\end{lemma}

\begin{remark}
    Ekholm-Etnyre-Sullivan's definition for admissible Legendrian isotopies requires more conditions, but for our purpose the definition above is already enough.
\end{remark}

\subsection{Grading of Reeb chords}\label{grade}
    In this section we discuss the grading of Reeb chords and Maslov potential.

    Recall that the symplectic structure on $T^*M$ will give a contractible choice of almost complex structures on the tangent bundle $T(T^*M)$, which canonically turns $T(T^*M)$ into a complex vector bundle. On $T^*M$ there is a canonical Lagrangian fibration given by the cotangent fibers. A framing on this Lagrangian fibration together with the almost complex structure $J$ determines a canonical trivialization of the complex vector bundle $T(T^*M)$.

\begin{definition}
    Let $\Lambda \rightarrow J^1(M)$ be a Legendrian immersion, and consider the Lagrangian projection onto $T^*M$. For any $\gamma \colon S^1 \hookrightarrow \Lambda \rightarrow T^*M$, consider the canonically trivialized complex vector bundle $\gamma^*T(T^*M)$ and the Lagrangian subbundle $\gamma^*T\Lambda$. Then the Maslov index of $\gamma$ is
    $$m(\gamma)\colon \mathbb{Z} \xrightarrow{\sim} \pi_1(S^1) \rightarrow \pi_1(U(n)/O(n)) \xrightarrow{\sim} \mathbb{Z}.$$
    Equivalently, we can regard $m(\gamma)$ as in $\mathbb{Z}$. The Maslov class of $\Lambda$ is the homomorphism
    $$\mu(\Lambda)\colon \pi_1(\Lambda) \rightarrow \mathbb{Z}, \, \gamma \mapsto m(\gamma).$$
    In fact $\mu(\Lambda) \in H^1(\Lambda)$.
\end{definition}

    Now we define the Maslov potential for a Legendrian submanifold $\Lambda$ with $\mu(\Lambda) = 0$. Currently Maslov potential is only defined combinatorially for Legendrian knots, since in higher dimensions it is hard (in fact, impossible) to classify the singularities of the front projection. Therefore here we only define the Maslov potential on a strand.

\begin{definition}
    Let $\Lambda \subset J^1(M)$ be a Legendrian submanifold such that the front projection $\pi_\text{front}\colon \Lambda \rightarrow M \times \mathbb{R}$ is a smooth hypersurface on an open dense subset. For a curve $\gamma\colon I \rightarrow \Lambda$, a Maslov potential is a step function
    $$d\colon \gamma(I) \rightarrow \mathbb{Z}$$
    such that for any $a, b \in \gamma(I)$, $d(a) - d(b)$ equals the number of down cusps minus the number of up cusps, and the value at a cusp equals those of points in a small neighbourhood of $\gamma(I)$ with greater $t$ coordinates. Here a cusp is going up (down) if $\gamma^*dt > 0$ ($\gamma^*dt < 0$).
\end{definition}
\begin{remark}
    It is not clear at all that the Maslov potential can be globally well-defined. However, when $\mu(\Lambda) = 0$ there is indeed a well-defined Maslov potential
    $$d\colon \Lambda \rightarrow \mathbb{Z}$$
    such that its restriction to any curve will be a Maslov potential on that strand. For a possible choice of the Maslov potential, see \cite{Gui}.
\end{remark}

    The following definition is coming from the formula obtained by Ekholm-Etnyre-Sullivan \cite[Section 3.5]{EESNoniso}. It may not be a good definition from a geometric viewpoint. However it is the most convenient one for us.

\begin{definition}
    Let $\Lambda \subset J^1(M)$ be a chord generic Legendrian submanifold, $\gamma$ be a Reeb chord on $\Lambda$ starting from $a$ and ending at $b$, and $d$ be a Maslov potential on any strand on $\Lambda$ connecting $a$ and $b$. Let $h_a, h_b$ the functions $\mathbb{R}^n \rightarrow \mathbb{R}$ be functions such that in small contact balls $U_a, U_b$ around $a$ and $b$,
    $$\Lambda \cap U_j = \{(x, dh_j(x), h_j(x)) \mid x \in \mathbb{R}\}.$$
    Let $h_{ab}(x) = h_b(x) - h_a(x)$. Then we define the degree $\deg(\gamma)$ by the following equation
    $$n - \deg(\gamma) = d(a) - d(b) + \mathrm{ind}(D^2h_{ab}) - 1.$$
\end{definition}
\begin{lemma}[Ekholm-Etnyre-Sullivan, \cite{EESNoniso}*{Lemma 3.4}]
    Let $\Lambda \subset J^1(M)$ be a chord generic Legendrian submanifold with $\mu(\Lambda) = 0$, $\gamma$ be a Reeb chord on $\Lambda$ starting from $a$ and ending at $b$. Then $\deg(\gamma)$ is independent of the strand on $\Lambda$ and the Maslov potential $d$.
\end{lemma}

    Basically, the degree $\deg(\gamma)$ is well-defined because it is equal to a shifted Conley-Zehnder index of $\gamma$. We won't discuss Conley-Zehnder indices here. Interested readers may refer to \cite[Section 2.3]{EESNoniso} or \cite[Section 2.2]{EESR2n+1}.

\section{Preliminaries in Sheaf Theory}\label{prelimsheaf}

\subsection{Singular Supports}
    We briefly review results in microlocal sheaf theory that we are going to use in this paper. For the theory of category of sheaves with unbounded cohomologies, we will refer to \cite{Unbound}.

\begin{definition}
    Let $\underline{Sh}(M)$ be the unbounded dg category of sheaves over $\Bbbk$, that consists of complexes of sheaves over $\Bbbk$. Then we let $Sh(M)$ be the dg localization of $\underline{Sh}(M)$ along all acyclic objects.
\end{definition}

\begin{example}
    We denote by $\Bbbk_M$ the constant sheaf on $M$. For a locally closed subset $i_V: V \hookrightarrow M$, abusing notations, we will write
    $$\Bbbk_V = i_{V!}\Bbbk_V \in Sh(M).$$
    In particular, $\Bbbk_V \in Sh(M)$ will have stalk $\Bbbk$ for $x \in V$ and stalk $0$ for $x \notin V$. Note that when $V \hookrightarrow M$ is a closed subset, $\Bbbk_V = i_{V*}\Bbbk_V$.
\end{example}

    We now define the notion of singular supports. For the theory of singular supports for sheaves with unbounded cohomologies, one may refer to \cite{MicrolocalInfty} or \cite[Section 2]{JinTreu}.

\begin{definition}
    Let $\mathscr{F} \in Sh(M)$. Then its singular support $SS(\mathscr{F})$ is the closure of the set of points $(x, \xi) \in T^*M$ such that there exists a smooth function $\varphi \in C^1(M)$, $\varphi(x) = 0, d\varphi(x) = \xi$ and
    $$\Gamma_{\varphi\geq 0}(\mathscr{F})_x \coloneqq \Gamma_{\varphi^{-1}([0,+\infty))}(\mathscr{F})_x \neq 0.$$
    The singular support at infinity is $SS^\infty(\mathscr{F}) = SS(\mathscr{F}) \cap T^{*,\infty}M$.

    For $\widehat \Lambda \subset T^*M$ any conical subset (resp.~$\Lambda \subset T^{*,\infty}M$ any subset), let $Sh_{\widehat\Lambda}(M) \subset Sh(M)$ (resp.~$Sh_\Lambda(M) \subset Sh(M)$) be the subcategory of sheaves such that $SS(\mathscr{F}) \subset \widehat\Lambda$ (resp.~$SS^\infty(\mathscr{F}) \subset \Lambda$).
\end{definition}

\begin{example}
    Let $\mathscr{F} = \Bbbk_{\mathbb{R}^n \times [0,+\infty)}$. Then $SS(\mathscr{F}) = \mathbb{R}^n \times \{(x, \xi) \mid x \geq 0, \xi = 0 \text{ or } x = 0, \xi \geq 0\}$, $SS^\infty(\mathscr{F}) = \nu^{*,\infty}_{\mathbb{R}^n \times \mathbb{R}_{>0},-}\mathbb{R}^{n+1} = \{(x_1, \dots, x_n, 0, 0,\dots, 0, 1)\}$, which is the inward conormal bundle of $\mathbb{R}^n \times \mathbb{R}_{>0}$.

    Let $\mathscr{F} = \Bbbk_{\mathbb{R}^n \times (0,+\infty)}$. Then $SS(\mathscr{F}) = \mathbb{R}^n \times \{(x, \xi) \mid x \geq 0, \xi = 0 \text{ or } x = 0,\, \xi \leq 0\}$, $SS^\infty(\mathscr{F}) = \nu^{*,\infty}_{\mathbb{R}^n \times \mathbb{R}_{>0},+}\mathbb{R}^{n+1} = \{(x_1, \dots, x_n, 0, 0,\dots, 0, -1)\}$, which is the outward conormal bundle of $\mathbb{R}^n \times \mathbb{R}_{>0}$.
\end{example}

    Kashiwara-Schapira proved that the singular support is always a closed coisotropic conical subset in $T^*M$. When the singular support of a sheaf is a subanalytic Lagrangian subset and has perfect stalk, it is called a constructible sheaf \cite[Definition 8.4.3]{KS}. A sheaf being constructible implies that it is also cohomologically constructible \cite[Definition 3.4.1]{KS}.

\begin{definition}
    Let $Sh^b_c(M) \subset Sh(M)$ be the dg derived category of constructible sheaves on $M$ consisting of sheaves with subanalytic Lagrangian singular support and perfect stalks. For $\widehat\Lambda \subset T^*M$ a conic subanalytic Lagrangian (resp.~$\Lambda \subset T^{*,\infty}M$ a subanalytic Legendrian), we let $Sh^b_{\widehat\Lambda}(M) = Sh^b_c(M) \cap Sh_{\widehat\Lambda}(M)$ (resp.~$Sh^b_\Lambda(M) \subset Sh^b_c(M) \cap Sh_\Lambda(M)$).
\end{definition}

    We define the linear dual and Verdier dual of a sheaf. Recall that for $p: M \rightarrow \{*\}$, the dualizing sheaf of $M$ is $\omega_M = p^!\Bbbk$. When $M$ is orientable with dimension $n$, $\omega_M = \Bbbk_M[n]$. For the detailed discussion, see Kashiwara-Schapira \cite[Section 3.3]{KS}.

\begin{definition}
    Let $\mathscr{F} \in Sh(M)$. The linear dual $D'\mathscr{F}$ and Verdier duality $D\mathscr{F}$ of $\mathscr{F}$ are defined by
    $$D'\mathscr{F} = \mathscr{H}om(\mathscr{F}, \Bbbk_M), \,\,\, D\mathscr{F} = \mathscr{H}om(\mathscr{F}, \omega_M).$$
\end{definition}

\begin{proposition}[\cite{KS}*{Proposition 3.4.6}]\label{construct}
    Let $\mathscr{F, G} \in Sh^b_c(M)$ be constructible. Then
    $$\mathscr{H}om(\mathscr{F, G}) \simeq D(D\mathscr{G} \otimes \mathscr{F}).$$
\end{proposition}

    We introduce the notion of a convolution and state the microlocal cut-off lemma.

\begin{definition}
    Let $V$ be an $\mathbb{R}$-vector space. Let
    \begin{gather*}
    \pi_1\colon V \times V \rightarrow V, (v_1, v_2) \mapsto v_1, \,\,\, \pi_2\colon V \times V \rightarrow V, (v_1, v_2) \mapsto v_2, \\
     s\colon V \times V \rightarrow V, (v_1, v_2) \mapsto v_1+v_2.
    \end{gather*}
    For $\mathscr{F, G} \in Sh(V)$, define the convolution as
    \[\begin{split}
    \mathscr{F} \star \mathscr{G} &= s_*(\pi_1^{-1}\mathscr{F} \otimes \pi_2^{-1}\mathscr{G}), \\
    \mathscr{F} \star' \mathscr{G} &= s_!(\pi_1^{-1}\mathscr{F} \otimes \pi_2^{-1}\mathscr{G}).
    \end{split}\]
\end{definition}

    Let $V$ be an $\mathbb{R}$-vector space and $\gamma \subset V$ be a closed cone, meaning that $\gamma$ is invariant under $\mathbb{R}_{>0}$-dilation. Then the polar set of $\gamma$ is
    $$\gamma^\vee = \{u \in V^\vee \mid \left<u, v \right> \geq 0, \, \forall\, v\in \gamma\}.$$
    For a subset $A \subset M$, the interior of $A$ is denoted by $A^\circ$.

\begin{lemma}[microlocal cut-off lemma \cite{KS}*{Proposition 5.2.3}, \cite{Gui}*{Proposition 2.9}]\label{cutoff}
    Let $V$ be an $\mathbb{R}$-vector space, $\gamma \subset V$ be a closed cone and $\mathscr{F} \in Sh^b(V)$. Then $SS(\mathscr{F}) \subset V \times \gamma^\vee$ iff
    $$\Bbbk_\gamma \star \mathscr{F} \xrightarrow{\sim} \Bbbk_0 \star \mathscr{F}.$$
\end{lemma}
\begin{remark}
    In Kashiwara-Schapira they use $\gamma^\circ$ as the polar set and $\mathrm{Int}(\gamma^\circ)$ for its interior but here we use different notions.
\end{remark}

    Here are some singular support estimates we are going to use. Let $f: M \rightarrow N$ be a smooth map. Then we have the following maps between vector bundles
    $$T^*M \xleftarrow{f_d} M \times_N T^*N \xrightarrow{f_\pi} T^*N,$$
    where $f_\pi$ is the natural map determined by fiber product, and $f_d$ is the pullback map of covectors or differential forms. More explicitly, for $(x, \eta) \in M \times_N T^*N$ where $\eta \in T_{f(x)}^*N$,
    $$f_\pi(x, \eta) = (f(x), \eta), \,\,\, f_d(x, \eta) = (x, f^*\eta).$$

\begin{proposition}[\cite{KS}*{Proposition 5.4.5}]
    Let $\mathscr{F} \in Sh(N)$ and $f: M \rightarrow N$ be a submersion. Then
    $$SS(f^{-1}\mathscr{F}) = f_d f_\pi^{-1}(SS(\mathscr{F})).$$
\end{proposition}
\begin{proposition}[\cite{KS}*{Proposition 5.4.4}]\label{sspushforward}
    Let $\mathscr{F} \in Sh(M)$ and $f: M \rightarrow N$ be a proper smooth map. Then
    $$SS(f_*\mathscr{F}) \subset f_\pi f_d^{-1}(SS(\mathscr{F})).$$
\end{proposition}
\begin{remark}
    In Kashiwara-Schapira, they call a smooth/continuous map as a morphism between manifolds, and call a submersion as a smooth morphism beween manifolds. Here we instead use the terminologies that may be more familiar to geometric topologists.
\end{remark}

\begin{proposition}[\cite{KS}*{Proposition 5.4.14}]\label{ssforhom}
    Let $\mathscr{F, G} \in Sh(M)$. Suppose $(-SS(\mathscr{F})) \cap SS(\mathscr{G}) \subset M \subset T^*M$. Then
    $$SS(\mathscr{F} \otimes \mathscr{G}) \subset SS(\mathscr{F}) + SS(\mathscr{G}).$$
    Suppose $SS(\mathscr{F}) \cap SS(\mathscr{G}) \subset M \subset T^*M$. Then
    $$SS(\mathscr{H}om(\mathscr{F, G})) \subset (-SS(\mathscr{F})) + SS(\mathscr{G}).$$
    Under the assumption, when $\mathscr{F}$ is constructible, then $\mathscr{H}om(\mathscr{F, G}) \simeq D'\mathscr{F} \otimes \mathscr{G}$.
\end{proposition}

    One machinery that we will be frequently using is the microlocal Morse thoery. We state the results here.

\begin{proposition}[microlocal Morse lemma \cite{KS}*{Corollary 5.4.19}]\label{morselemma}
    Let $\mathscr{F} \in Sh(M)$ and $f: M \rightarrow \mathbb{R}$ be a smooth function that is proper on $\mathrm{supp}(\mathscr{F})$. Suppose for any $x\in f^{-1}([a, b))$, $df(x) \notin SS(\mathscr{F})$. Then
    $$\Gamma(f^{-1}((-\infty, b)), \mathscr{F}) \xrightarrow{\sim} \Gamma(f^{-1}((-\infty, a)), \mathscr{F}).$$
\end{proposition}
\begin{example}[\cite{STZ}*{Section 3.3}]\label{combin-model}
    Suppose $\Lambda = \nu^{*,\infty}_{\mathbb{R}^n \times \mathbb{R}_{>0},-}\mathbb{R}^{n+1} \subset T^{*,\infty}\mathbb{R}^{n+1}$ is the inward conormal bundle of $\mathbb{R}^n \times \mathbb{R}_{>0}$ at infinity, and $\mathscr{F} \in Sh^b_\Lambda(\mathbb{R}^{n+1})$. Then by microlocal Morse lemma, $\mathscr{F}|_{\mathbb{R}^n \times \{0\}}$, $\mathscr{F}|_{\mathbb{R}^n \times (0,+\infty)}$ and $\mathscr{F}|_{\mathbb{R}^n \times (-\infty,0)}$ are locally constant sheaves, and
    $$\Gamma(\mathbb{R}^n \times \{0\}, \mathscr{F}) \simeq \Gamma(\mathbb{R}^{n+1}, \mathscr{F}) \simeq \Gamma(\mathbb{R}^n \times [0, +\infty), \mathscr{F}).$$
    Suppose that the locally constant sheaves are
    $$\mathscr{F}|_{\mathbb{R}^n \times [0,+\infty)} = F_+|_{\mathbb{R}^n \times [0,+\infty)}, \,\,\, \mathscr{F}|_{\mathbb{R}^n \times (-\infty,0)} = F_-|_{\mathbb{R}^n \times (-\infty,0)}.$$
    Then $\mathscr{F}$ is determined by the diagram (Figure \ref{singularsupport})
    \[\xymatrix{
    F_- & F_+ \ar[l] \ar[r]^\sim & F_+
    }\]
\end{example}

\begin{figure}
  \centering
  \includegraphics[width=0.25\textwidth]{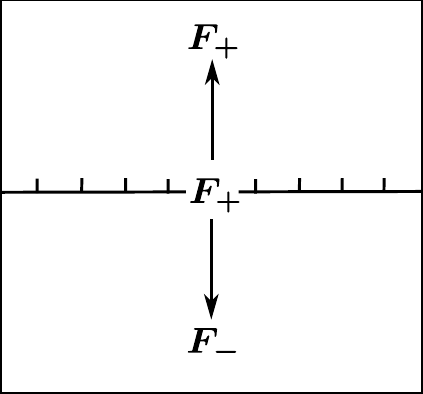}\\
  \caption{The singular support of a sheaf and the combinatoric description.}\label{singularsupport}
\end{figure}

\begin{proposition}[microlocal Morse inequality \cite{KS}*{Proposition 5.4.20}]\label{morseinequality}
    Let $\mathscr{F} \in Sh(M)$ and $f: M \rightarrow \mathbb{R}$ be a smooth function that is proper on $\mathrm{supp}(\mathscr{F})$. Let $\Lambda_\varphi = \{(x, d\varphi(x)) \mid x\in M\}$, and suppose that
    $$SS(\mathscr{F}) \cap \Lambda_\varphi = \{(x_1, \xi_1), \dots, (x_n, \xi_n)\}$$
    and $V_i = \Gamma_{\varphi \geq \varphi(x_i)}(\mathscr{F})_{x_i}$ is finite dimensional. Then $\Gamma(M, \mathscr{F})$ is also finite dimensional and for any $l\in \mathbb{Z}$
    $$\sum_{1\leq i\leq n}\sum_{j \leq l}(-1)^{l-j}\dim H^j(V_i) \geq \sum_{j\leq l}(-1)^{l-j}\dim H^j(M, \mathscr{F}).$$
    In particular for any $j\in \mathbb{Z}$, $\sum_{1\leq i\leq n}\dim H^j(V_i) \geq \dim H^j(M, \mathscr{F})$.
\end{proposition}

\subsection{Microlocalization and $\mu Sh$}
    We review the definition and properties of microlocalization and the sheaf of categories $\mu Sh$ and $\mu Sh_\Lambda$, considered in \cite{Gui,NadWrapped,Guisurvey,NadShen}. This will mainly used in the proof of the exact triangle (Theorem \ref{exactseq}).

\begin{definition}
    Let $\Lambda \subset T^{*,\infty}M$ be a subset. Define a presheaf of dg categories on $T^{*,\infty}M$ supported on $\Lambda$ to be
    $$\mu Sh^{\text{pre}}_\Lambda: \, \Omega \mapsto Sh_{\Lambda \cup T^{*,\infty}M \backslash \Omega}(M)/Sh_{T^{*,\infty}M \backslash \Omega}(M).$$
    The sheafification of $\mu Sh^{\text{pre}}_\Lambda$ is $\mu Sh_\Lambda$. In particular, write $\mu Sh = \mu Sh_{T^{*,\infty}M}$ for the sheaf of categories on $T^{*,\infty}M$.

    Let $Sh_{(\Lambda)}(M)$ be the category of sheaves $\mathscr{F}$ such that there exists an open set $\Omega$ containing $\Lambda$ satisfying $SS^\infty(\mathscr{F}) \cap \Omega \subset \Lambda$. For $\mathscr{F, G} \in Sh(M)$, let the sheaf of homomorphisms in $\mu Sh_\Lambda$ be
    $$\mu hom(\mathscr{F}, \mathscr{G})|_\Lambda: \, U \mapsto Hom_{\mu Sh_\Lambda}(\mathscr{F, G}).$$
    In particular, write $\mu hom(\mathscr{F}, \mathscr{G})|_{T^{*,\infty}M}$ to be the sheaf of homomorphisms in $\mu Sh$.
\end{definition}
\begin{remark}
    We briefly explain the relation between the above definition and Guillermou's definition \cite[Definition 10.1.1]{Guisurvey}. Since in our definition, the presheaf of categories $\mu Sh_\Lambda^{\text{pre}}$ is supported on $\Lambda$, we know that for an open set $\Lambda_0 \subset \Lambda$ in the relative topology, 
    \begin{align*}
    \mu Sh_\Lambda^{\text{pre}}(\Lambda_0) = \operatorname{colim}_{\Omega \colon \Omega \cap \Lambda = \Lambda_0}\mu Sh_\Lambda^{\text{pre}}(\Lambda_0) \cong Sh_{(\Lambda_0)}(M)/Sh_{T^{*,\infty}M \setminus \Lambda_0}(M).
    \end{align*}
    Then \cite[Theorem 10.1.5]{Guisurvey} shows that the hom in $\mu Sh$ indeed agrees with the sheaf $\mu hom$ in \cite[Section 4.4]{KS}.
    One reason to give the definition as above is that $Sh_{T^{*,\infty}M \backslash \Omega}(M)$ is closed under colimits and we can take localization of (unbounded but) complete dg categories.
\end{remark}

    Denote by $m_\Lambda$ the natural quotient functor on the sheaf of categories, which, on the level of global sections, induce
    $$m_\Lambda: \, Sh_\Lambda(M) \rightarrow \mu Sh_\Lambda(\Lambda).$$
    We call $m_\Lambda$ the microlocalization functor.

    Now we define the notion of microstalks, and thus define simple sheaves and pure sheaves, or microlocal rank $r$ sheaves.

\begin{definition}
    Let $\Lambda \subset T^{*,\infty}M$ be a Legendrian submanifold. Suppose $\mu(\Lambda) = 0$ and $\Lambda$ is relative spin. For $p = (x, \xi) \in \Lambda$, $\varphi \in C^1(M)$ such that $\varphi(x) = 0,d\varphi(x) = \xi$, the microstalk of $\mathscr{F} \in Sh(M)$ at $p$ is defined up to degree shifts as
    $$m_{\Lambda,p}(\mathscr{F}) = \Gamma_{\varphi\geq 0}(\mathscr{F})_x.$$
    $\mathscr{F} \in Sh_\Lambda(M)$ is called microlocal rank $r$ if $m_{\Lambda,p}(\mathscr{F})$ is concentrated at a single degree with rank $r$. In this case $\mathscr{F}$ is called pure, and when $r = 1$ it is also called simple.
\end{definition}

    The following proposition justifies the name microstalk, showing that microstalks are indeed the stalks of the sheaf of categories $\mu Sh_\Lambda$ on a smooth Legendrian $\Lambda$.

\begin{proposition}[\cite{Guisurvey}*{Equation (1.4.6), Remark 10.1.7 \& Lemma 10.2.2}, \cite{NadShen}*{Corollary 5.4}]\label{microstalk}
    For $p = (x, \xi) \in \Lambda \subset T^{*,\infty}M$ where $\Lambda \subset T^{*,\infty}M$ is a smooth Legendrian, the microstalk satisfies the following: for $\mathscr{F, G} \in Sh_{(\Lambda)}(M)$, $\varphi \in C^1(M)$ such that $\varphi(x) = 0,d\varphi(x) = \xi$,
    $$Hom_{\mu Sh_{p}}(\mathscr{F, G}) = \mu hom(\mathscr{F, G})_p = Hom(\Gamma_{\varphi \geq 0}(\mathscr{F})_x, \Gamma_{\varphi \geq 0}(\mathscr{G})_x).$$ 
    In particular, the stalk of $\mu Sh_\Lambda$ is $\mu Sh_{\Lambda,p} \simeq \mathrm{Mod}(\Bbbk)$.
\end{proposition}

\begin{proposition}[\cite{Guisurvey}*{Equation (1.4.4)}]
    Let $\Lambda \subset T^{*,\infty}M$ be a Legendrian submanifold. $\mathscr{F} \in Sh_\Lambda(M)$ is microlocal rank $r$ at $p \in \Lambda$ iff
    $$\mu hom(\mathscr{F, F})_p \simeq \Bbbk^{r^2}.$$
\end{proposition}

    Globally, the sheaf of categories $\mu Sh_\Lambda$ is not always the sheaf of local systems on $\Lambda$, but this is true when $\Lambda$ has zero Maslov class and relative second Stiefel-Whitney class.

\begin{theorem}[Guillermou \cite{Gui}*{Theorem 11.5}]\label{Gui}
    Let $\Lambda \subset T^{*,\infty}M$ be a Legendrian submanifold. Suppose the Maslov class $\mu(\Lambda) = 0$ and $\Lambda$ is relative spin, then as sheaves of categories
    $$\mu Sh_\Lambda \xrightarrow{\sim} Loc_\Lambda.$$
\end{theorem}

\begin{proposition}[Guillermou \cite{Gui}*{Theorem 7.6~(iv), 7.9, 8.10 \& Lemma 11.4}, \cite{Guisurvey}*{Proposition 10.5.3 \& 10.6.2}]
    Let $\Lambda \subset T^{*,\infty}M$ be a Legendrian submanifold. Suppose the Maslov class $\mu(\Lambda) = 0$ and $\Lambda$ is relative spin. When the front projection of $\Lambda$ is a smooth hypersurface near $p$ and $\varphi \in C^1(M)$ is a local defining function for $\Lambda$, then
    $$m_{\Lambda,p}(\mathscr{F}) = \Gamma_{\varphi \geq 0}(\mathscr{F})_x[-d(p)].$$
    For two different points $p$ and $p' \in \Lambda$, $d(p) - d(p')$ is equal to the difference of any Maslov potential at $p$ and $p'$.
\end{proposition}

\begin{example}\label{microstalk-cone}
    Suppose $\Lambda = \nu^{*,\infty}_{\mathbb{R}^n \times \mathbb{R}_{>0},-}\mathbb{R}^{n+1} \subset T^{*,\infty}\mathbb{R}^{n+1}$ is the inward conormal of $\mathbb{R}^n \times \mathbb{R}_{>0}$ and $\mathscr{F} \in Sh_\Lambda(\mathbb{R}^{n+1})$. Then $\mathscr{F}$ is determined by
    \[\xymatrix{
    F_- & F_+ \ar[l] \ar[r]^\sim & F_+
    }\]
    For $p = (0,\dots,0, 0; 0, \dots, 0, 1) \in \Lambda$ we can pick $\varphi(x) = x_{n+1}$ and get
    $$\Gamma_{\varphi \geq 0}(\mathscr{F})_{(0,\dots,0)} = \mathrm{Cone}(F_+ \rightarrow F_-)[-1] \simeq \mathrm{Tot}(F_+ \rightarrow F_-).$$
    Therefore one can see that the definition of the microstalk coincides with the definition of the microlocal monodromy defined by Shende-Treumann-Zaslow \cite[Section 5.1]{STZ}, and indeed
    $$m_{\Lambda,p}(\mathscr{F}) \simeq \mu mon(\mathscr{F})_p[-1].$$
\end{example}

    Finally we recall the famous Sato's exact triangle, which follows from \cite[Equation (4.3.1)]{KS}. This will be the essential ingredient for the proof of Sato-Sabloff exact triangle in Theorem \ref{exactseq}.

\begin{theorem}[Sato's exact triangle \cite{Guisurvey}*{Equation (1.3.5)}]\label{sato}
    Let $\mathscr{F} \in Sh^b_c(M)$ be a constructible sheaf. Then there is an exact triangle
    $$D'\mathscr{F} \otimes \mathscr{G} \rightarrow Hom(\mathscr{F, G}) \rightarrow \pi_*(\mu hom(\mathscr{F, G})|_{T^{*,\infty}M}) \xrightarrow{+1}.$$
\end{theorem}

\subsection{Functors for Hamiltonian Isotopies}
    In this section we review the equivalence functor from a Hamiltonian isotopy defined by Guillermou-Kashiwara-Schapira \cite{GKS}.

\begin{definition}
    Let $\widehat H_s\colon T^*M \times I \rightarrow T^*M$ be a homogeneous Hamiltonian on $T^*M$. Then the Lagrangian graph of the homogeneous Hamiltonian is
    $$\mathrm{Graph}_{\widehat H} = \{(x, x', \xi, \xi', s, \sigma) \mid (x', \xi') = \varphi_{\widehat H}^s(x, \xi), \sigma = -\widehat H_s \circ \varphi_{\widehat H}^s(x, \xi)\} \subset T^*(M \times M \times I).$$
    For a conical Lagrangian $\widehat\Lambda$, the Lagrangian movie of $\widehat\Lambda$ under the Hamiltonian is
    $$\widehat\Lambda_{\widehat H} = \{(x, \xi, s, \sigma) \mid (x, \xi) = \varphi_{\widehat H}^s(x_0, \xi_0), \sigma = -\widehat H_s \circ \varphi_{\widehat H}^s(x_0, \xi_0), (x_0, \xi_0) \in \widehat \Lambda\} \subset T^*(M \times I).$$
\end{definition}

    The main theorem of Guillermou-Kashiwara-Schapira is that Hamiltonian isotopies define equivalence functors via convolutions of sheaf kernels in the product, which is called the sheaf quantization of the Hamiltonian. We will call the induced equivalence the sheaf quantization functors.  

\begin{theorem}[Guillermou-Kashiwara-Schapira \cite{GKS}*{Proposition 3.12}]\label{GKS}
    Let $\widehat H_s \colon T^*M \times I \rightarrow T^*M$ be a homogeneous Hamiltonian on $T^*M$ and $\widehat\Lambda$ a conical Lagrangian in $T^*M$. Then there are functors that give equivalences
    $$Sh_{\widehat\Lambda}(M) \xleftarrow{\sim} Sh_{\hat\Lambda_{\hat H}}(M \times I) \xrightarrow{\sim} Sh_{\varphi_{\widehat H}^1(\widehat\Lambda)}(M)$$
    given by restriction functors $i_0^{-1}$ and $i_1^{-1}$ where $i_s: M \times \{s\} \hookrightarrow M \times I$ is the inclusion.
\end{theorem}
\begin{remark}\label{GKS-trivial}
    For $\Lambda \times I = \{(x, \xi, s, 0) \mid (x, \xi) \in \Lambda, s \in I\} \subset T^{*,\infty}(M \times I)$, we have equivalences of categories $Sh_{\Lambda\times I}(M \times I) \simeq Sh_\Lambda(M)$ given by the restriction $i_s^{-1}$ and its inverse $q^{-1}$ where $i_s: M \times \{s\} \hookrightarrow M \times I$ is the inclusion and $q: M \times I \to M$ is the projection. In particular, we know that $q_*q^{-1} = \mathrm{id}$ \cite[Corollary 1.6]{GKS}.
\end{remark}

\section{Duality and Exact Triangle}\label{dualityexact}

\subsection{Two Sheaf Categories}
    We recall the definitions we made in the introduction and prove some basic properties. As is explained in the introduction, we consider to add an extra $\mathbb{R}$ factor in order to see the Reeb chords. 

\begin{definition}[Definition \ref{translation}]
    Let $q\colon M \times \mathbb{R}^2 \rightarrow M \times \mathbb{R}$ be $q(x, t, u) = (x, t)$ and $r: M \times \mathbb{R}^2 \rightarrow M \times \mathbb{R}$ be $r(x, t, u) = (x, t-u)$. For a Legendrian $\Lambda \subset T^{*,\infty}_{\tau > 0}(M \times \mathbb{R})$, let
    \[\begin{split}
    \Lambda_q &= \{(x, \xi, t, \tau, u, 0) \mid (x, \xi, t, \tau) \in \Lambda\}, \\
    \Lambda_r &= \{(x, \xi, t + u, \tau, u, -\tau) \mid (x, \xi, t, \tau \in \Lambda)\}.
    \end{split}\]
    For a sheaf $\mathscr{F} \in Sh(M \times \mathbb{R})$, we write
    $\mathscr{F}_q = q^{-1}\mathscr{F}, \,\,\, \mathscr{F}_r = r^{-1}\mathscr{F}.$
\end{definition}
\begin{remark}
    For readers who are familiar with Tamarkin categories \cite{Tamarkin1,GS}, we can explain the relation to the geometric construction using the Reeb flow. Let $p\colon M \times \mathbb{R}^2 \rightarrow M \times \mathbb{R}$ be the projection $p(x, t, u) = (x, u)$. Then up to inversion of the $\mathbb{R}$ factor, the internal hom in Tamarkin category is \cite[Equation (45)]{GS} \cite[Definition 3.1]{AsanoIke}
    $$\mathscr{H}om^\star(\mathscr{F, G}) = p_*\mathscr{H}om(q^{-1}\mathscr{F}, r^!\mathscr{G}) = p_*\mathscr{H}om(\mathscr{F}_q, \mathscr{G}_r)[1].$$
\end{remark}

    It is not hard to observe that every intersection point for some $\Lambda$ and Reeb translation $T_{c}(\Lambda)$ where
    $$T_c: T^{*,\infty}_{\tau>0}(M \times \mathbb{R}) \rightarrow T^{*,\infty}_{\tau>0}(M \times \mathbb{R}); \,\,\, (x, \xi, t, \tau) \mapsto (x, \xi, t+c, \tau)$$
    comes from a Reeb chord of $\Lambda$. The following lemma shows that those are all covectors pointing toward $du$ direction (i.e.~in $M \times \mathbb{R}_t \times T^*\mathbb{R}_u$) that lie in the singular support of $\mathscr{H}om(\mathscr{F}_q, \mathscr{G}_r)$.

\begin{lemma}\label{reebchord-hom}
    For $\Lambda \subset T^{*,\infty}_{\tau > 0}(M \times \mathbb{R})$ and $\mathscr{F, G} \in Sh_\Lambda(M \times \mathbb{R})$,
    $$SS^\infty(\mathscr{H}om(\mathscr{F}_q, \mathscr{G}_r)) \cap \mathrm{Graph}(du) = \varnothing.$$
    On the other hand, there is an injection from
    $$SS^\infty(\mathscr{H}om(\mathscr{F}_q, \mathscr{G}_r)) \cap \mathrm{Graph}(-du)$$
    to the set of directed Reeb chords (trajectories of the Reeb flow for some positive or negative time $u$) $\mathcal{Q}_\pm(\Lambda) = \{\gamma\colon [0,u] \rightarrow T^{*,\infty}_{\tau>0}(M \times \mathbb{R}) \mid \gamma(s) = (x, \xi, t+s, \tau),\,\gamma(0), \gamma(u) \in \Lambda\}$.
\end{lemma}
\begin{proof}
    Since $SS^\infty(\mathscr{F}_q) \cap SS^\infty(\mathscr{F}_r) = \Lambda_q \cap \Lambda_r = \varnothing$, we can apply the singular support estimate Proposition \ref{ssforhom}
    $$SS^\infty(\mathscr{H}om(\mathscr{F}_q, \mathscr{G}_r)) \subset (-SS^\infty(\mathscr{F}_q)) + SS^\infty(\mathscr{G}_r) = (-\Lambda_q) + \Lambda_r.$$
    Hence $(x, 0, t, 0, u, \nu) \in (-\Lambda_q) + \Lambda_r$ iff there exists a pair $(x, \xi, t, \tau), (x, \xi, t+u, \tau) \in \Lambda$, or equivalently there is a (directed) Reeb chord from $(x, \xi, t, \tau)$ to $(x, \xi, t+u, \tau) \in \Lambda$ of length $u$. Moreover, we know that $\nu = -\tau < 0$ is determined by such a pair. Hence when $\nu > 0$, there will never be $(x, 0, t, 0, u, \nu) \in (-\Lambda_q) + \Lambda_r$. Therefore
    \begin{gather*}
    SS^\infty(\mathscr{H}om(\mathscr{F}_q, \mathscr{G}_r )) \cap \mathrm{Graph}(du) = \varnothing, \\
    SS^\infty(\mathscr{H}om(\mathscr{F}_q, \mathscr{G}_r)) \cap \mathrm{Graph}(-du) \hookrightarrow \mathcal{Q}_\pm(\Lambda),
    \end{gather*}
    where the injection sends $(x, 0, t, 0, u, -\tau)$ to the (directed) Reeb chord from $(x, \xi, t, \tau)$ to $(x, \xi, t+u, \tau) \in \Lambda$.
\end{proof}

\begin{figure}
  \centering
  \includegraphics[width=0.8\textwidth]{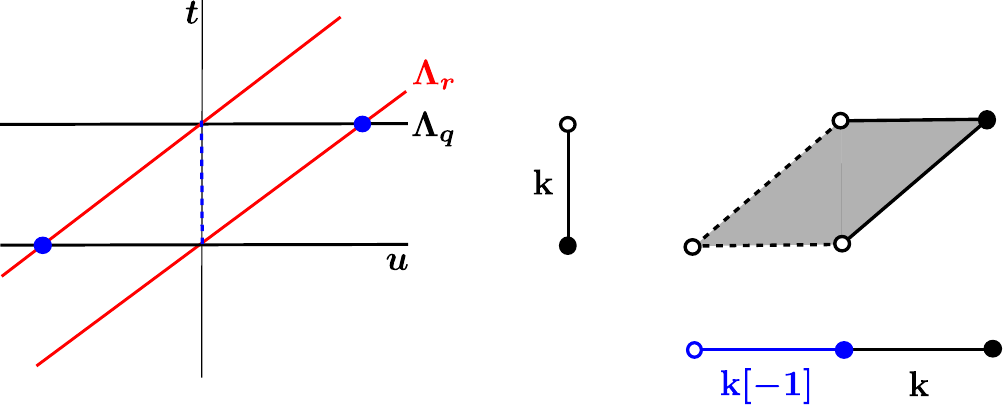}\\
  \caption{When $M$ is a point, $\Lambda \subset \mathbb{R}$ consists of two points $0$ and $1$, the front of the Legendrians $\Lambda_q$ and $\Lambda_r$ are shown on the left. For $\mathscr{F} = \Bbbk_{[0,1)}$, the sheaf $\mathscr{H}om(\mathscr{F}_q, \mathscr{F}_r)$ and its projection $u_*\mathscr{H}om(\mathscr{F}_q, \mathscr{F}_r)$ are shown on the right. The blue points are coming from the Reeb chord corresponding to the dashed blue line.}\label{reebexample}
\end{figure}

    The following corollary produces an acyclic complex, which will be used to deduce Sabloff duality. The reader may compare it to the acyclic complex produced in generating family (co)homology \cite[Section 3.1]{Genfamily}.

\begin{corollary}\label{acyclic}
    For $\Lambda \subset T^{*,\infty}_{\tau > 0}(M \times \mathbb{R})$ and $\mathscr{F, G} \in Sh_\Lambda^b(M \times \mathbb{R})$ with perfect stalks and compact supports,
    $$\Gamma(M \times \mathbb{R}^2, \mathscr{H}om(\mathscr{F}_q, \mathscr{G}_r)) \simeq 0.$$
\end{corollary}
\begin{proof}
    Since $SS^\infty(\mathscr{F}_q) \cap SS^\infty(\mathscr{G}_r) = \Lambda_q \cap \Lambda_r = \emptyset$, by Proposition \ref{ssforhom}
    $$\mathscr{H}om(\mathscr{F}_q, \mathscr{G}_r) \simeq D'\mathscr{F}_q \otimes \mathscr{G}_r.$$
    Since $\mathrm{supp}(\mathscr{F}), \mathrm{supp}(\mathscr{G})$ are compact, we know that when $c$ is sufficiently large, $\Lambda \cap T_{\pm c}(\Lambda) = \varnothing$. Hence for large $c > 0$,
    $$\mathrm{supp}(D'\mathscr{F}_q \otimes \mathscr{G}_r) \subset q^{-1}(\mathrm{supp}(\mathscr{F})) \cap r^{-1}(\mathrm{supp}(\mathscr{G})) \subset M \times [-c, c]^2.$$
    Therefore consider the function $\varphi_+(x,t,u) = u$, $\varphi_+|_{\mathrm{supp}(R\mathscr{H}om(\mathscr{F}_q, \mathscr{G}_r))}$ is proper and
    $$SS(\mathscr{H}om(\mathscr{F}_q, \mathscr{G}_r)) \cap \mathrm{Graph}(d\varphi_+) = \varnothing.$$
    One can apply microlocal Morse lemma \ref{morselemma} and see that
    $$\Gamma(M \times \mathbb{R}^2, \mathscr{H}om(\mathscr{F}_q, \mathscr{G}_r)) \simeq \Gamma(M \times \mathbb{R} \times (-\infty, -c), \mathscr{H}om(\mathscr{F}_q, \mathscr{G}_r)) = 0.$$
    This completes the proof.
\end{proof}

    Similar to the case in Legendrian contact homology, where people defines two $\mathcal{A}_\infty$-categories $\mathcal{A}ug_-$ and $\mathcal{A}ug_+$, here we also define two dg categories of sheaves. The idea comes from the definition of the generating family cohomology.

    From now on, the projection $M \times \mathbb{R}^2, \,(x, t, u) \mapsto u$ will be denoted by $u$.

\begin{definition}[Definition \ref{defhom}]
    For $\Lambda \subset T^{*,\infty}_{\tau > 0}(M \times \mathbb{R})$ and $\mathscr{F, G} \in Sh_\Lambda(M \times \mathbb{R})$, let
    \[\begin{split}
    Hom_-(\mathscr{F, G}) = \Gamma(u^{-1}([0,+\infty)), Hom(\mathscr{F}_q, \mathscr{G}_r)),\\
    Hom_+(\mathscr{F, G}) = \Gamma(u^{-1}((0,+\infty)), Hom(\mathscr{F}_q, \mathscr{G}_r)).
    \end{split}\]
\end{definition}

\begin{example}
    Let $M$ be a point, $\Lambda \subset \mathbb{R}$ consists of two points $0$ and $1$ (see Figure \ref{reebexample}). For $\mathscr{F} = \Bbbk_{[0,1)}$, the sheaf
    $$u_*\mathscr{H}om(\mathscr{F}_q, \mathscr{F}_r) \simeq \Bbbk_{(-1,0]}[-1] \oplus \Bbbk_{(0, 1]}.$$
    Therefore as the projection $u: \mathbb{R}^2 \rightarrow \mathbb{R}$ is proper on $\mathrm{supp}(\mathscr{H}om(\mathscr{F}_q, \mathscr{F}_r))$, we have
    \[\begin{split}
    Hom_-(\mathscr{F, F}) = \Gamma([0,+\infty), \Bbbk_0[-1]& \oplus \Bbbk_{(0,1]}) = \Bbbk[-1],\\
    Hom_+(\mathscr{F, F}) = \Gamma((0,+\infty), &\Bbbk_{(0,1]}) = \Bbbk.
    \end{split}\]
\end{example}

\begin{theorem}[Theorem \ref{computehom}]\label{computehom2}
    For a Legendrian $\Lambda \subset T^{*,\infty}_{\tau > 0}(M \times \mathbb{R})$ and sheaves $\mathscr{F, G} \in Sh^b_\Lambda(M \times \mathbb{R})$ with perfect stalks and compact supports,
    $$Hom_-(\mathscr{F, G}) \simeq \Gamma(D'\mathscr{F} \otimes \mathscr{G}), \,\,\, Hom_+(\mathscr{F, G}) \simeq Hom(\mathscr{F, G}).$$
\end{theorem}

    Now we prove the theorem. The first part
    $$\Gamma(u^{-1}((0,+\infty)), Hom(\mathscr{F}_q, \mathscr{G}_r)) \simeq Hom(\mathscr{F}, \mathscr{G})$$
    is essentially due to Guillermou \cite[Corollary 12.3.5]{Guisurvey}. A more general version can be found in \cite[Proposition 2.9]{Zhou}. Here we adapt the proof of Jin-Treumann \cite[Proposition 3.16]{JinTreu}.

\begin{proof}[Proof of Theorem \ref{computehom2} part 1: $Hom_+(\mathscr{F, G}) \simeq Hom(\mathscr{F}, \mathscr{G})$]
    As in the proof of Corollary \ref{acyclic}, we choose the function $\varphi_+(x,t,u) = u$. Then by microlocal Morse lemma Proposition \ref{morselemma},
    $$\Gamma(M \times \mathbb{R} \times (0, +\infty), \mathscr{H}om(\mathscr{F}_q, \mathscr{G}_r)) \simeq \varprojlim\nolimits_{c\to 0}\Gamma(M \times \mathbb{R} \times (0, c), \mathscr{H}om(\mathscr{F}_q, \mathscr{G}_r)).$$
    Now it suffices to show that
    $$\varprojlim\nolimits_{c\to 0}\Gamma(M \times \mathbb{R} \times (0, c), \mathscr{H}om(\mathscr{F}_q, \mathscr{G}_r)) \simeq Hom(\mathscr{F, G}).$$
    This follows from Guillermou's result which we now recall.  Let $c_0$ be the shortest length of Reeb chords in $\mathcal{Q}(\Lambda)$. When $0 < c < c_0$, there are no intersection points between $\Lambda$ and $T_c(\Lambda)$. Hence $(\Lambda_q \cup \Lambda_r) \cap T^{*,\infty}(M \times \mathbb{R} \times (0, c))$ is the movie of a Legendrian isotopy, which extends to a contact Hamiltonian isotopy. By Guillermou-Kashiwara-Schapira's Theorem \ref{GKS}, we know for any $0< c < c_0$
    \[
    \varprojlim\nolimits_{c>0}\Gamma(M \times \mathbb{R} \times (0, c), \mathscr{H}om(\mathscr{F}_q, \mathscr{G}_r)) 
    \simeq Hom\big(\mathscr{F}, \varprojlim\nolimits_{c>0}T_{c*}\mathscr{G}\big).
    \]
    Here $T_c: M \times \mathbb{R} \rightarrow M \times \mathbb{R}$ is the vertical translation (by abuse of notations). Note that since $SS^\infty(\mathscr{F}), SS^\infty(\mathscr{G}) \subset T^{*,\infty}_{\tau > 0}(M \times \mathbb{R})$, by microlocal cutoff lemma \ref{cutoff}
    $$\mathscr{G} \simeq \Bbbk_{[0,+\infty)} \star \mathscr{G} = s_*(\pi_{\mathbb{R}}^{-1}\Bbbk_{[0,+\infty)} \otimes \pi_{M \times \mathbb{R}}^{-1}\mathscr{G}),$$
    where $s\colon M \times \mathbb{R}^2 \rightarrow M \times \mathbb{R}, \, (x, t_1, t_2) \mapsto (x, t_1+t_2)$. By elementary computation we know
    $$T_{c*}\mathscr{G} \simeq \Bbbk_{[c,+\infty)} \star \mathscr{G},$$
    and the map $\mathscr{G} \rightarrow T_{c*}\mathscr{G}$ is induced by $\Bbbk_{[0,+\infty)} \rightarrow \Bbbk_{[c,+\infty)}$. Since $\Bbbk_{[0,+\infty)} \xrightarrow{\sim} \varprojlim_{c > 0}\Bbbk_{[c,+\infty)}$, and the push-forward functor commutes with limits, we can conclude that
    $$Hom\big(\mathscr{F}, \varprojlim\nolimits_{c>0}T_{c*}\mathscr{G}\big) \simeq Hom(\mathscr{F, G}).$$
    This proves the assertion.
\end{proof}

    For the second part of the theorem, we will need to use the fact that $\mathscr{H}om(\mathscr{F}_q, \mathscr{G}_r) \simeq D'\mathscr{F}_q \otimes \mathscr{G}_r$ in order to relate $\mathscr{H}om(\mathscr{F}_q, \mathscr{G}_r)$ with $D'\mathscr{F} \otimes \mathscr{G}$.

\begin{proof}[Proof of Theorem \ref{computehom2} part 2: $Hom_-(\mathscr{F, G}) \simeq \Gamma(D'\mathscr{F} \otimes \mathscr{G})$]
    Choose $\varphi_-(x, t, u) = -u$. By microlocal Morse lemma Proposition \ref{morselemma}, 
    \[\begin{split}
    \Gamma(u^{-1}([0,+\infty)), \mathscr{H}om(\mathscr{F}_q \otimes \mathscr{G}_r)) \simeq \Gamma(u^{-1}(0), \mathscr{H}om(\mathscr{F}_q, \mathscr{G}_r)).
    \end{split}\]
    Since $\Lambda_q \cap \Lambda_r = \varnothing$, by Proposition \ref{ssforhom}, we have $\mathscr{H}om(\mathscr{F}_q, \mathscr{G}_r) = D'\mathscr{F}_q \otimes \mathscr{G}_r$. Since the pull-back functor commutes with tensor product, we have
    \[\begin{split}
    \Gamma(u^{-1}(0), D'\mathscr{F}_q \otimes \mathscr{G}_r) &\simeq \Gamma(u^{-1}(0), i_{u=0}^{-1}(D'\mathscr{F}_q \otimes \mathscr{G}_r)) \\
    &\simeq \Gamma(u^{-1}(0), i_{u=0}^{-1}(D'\mathscr{F}_q) \otimes i_{u=0}^{-1}\mathscr{G}_r),
    \end{split}\]
    where $i_{u=0}: u^{-1}(0) \hookrightarrow M \times \mathbb{R}^2$ is the inclusion. In addition, since $q: M \times \mathbb{R}^2 \to M \times \mathbb{R}$ has contractible fiber,
    $$q_*D'\mathscr{F}_q \simeq q_*\mathscr{H}om(q^{-1}\mathscr{F}, \Bbbk_{M \times \mathbb{R}^2}) = \mathscr{H}om(\mathscr{F}, \Bbbk_{M \times \mathbb{R}}) = D'\mathscr{F}.$$ 
    Since $SS^\infty(D'\mathscr{F}_q) \subset -\Lambda_q$, we know by Guillermou-Kashiwara-Schapira Theorem \ref{GKS} and Remark \ref{GKS-trivial} that $q^{-1}q_*D'\mathscr{F}_q = D'\mathscr{F}_q$ and hence
    $i_{u=0}^{-1}(D'\mathscr{F}_q) \simeq D'\mathscr{F}.$
    Therefore, we can conclude that
    $$\Gamma(u^{-1}([0,+\infty)), \mathscr{H}om(\mathscr{F}_q, \mathscr{G}_r)) \simeq \Gamma(i_{u=0}^{-1}D'\mathscr{F}_q \otimes  i_{u=0}^{-1}\mathscr{G}_r) \simeq \Gamma(D'\mathscr{F} \otimes \mathscr{G}),
    $$
    The proof is thus completed.
\end{proof}

\begin{remark}
    The reason $Hom(\mathscr{F, G}) \not\simeq Hom_-(\mathscr{F, G})$ is that for the homomorphism
    $$i_{u=0}^{-1}\mathscr{H}om(\mathscr{F}_q, \mathscr{G}_r) \neq \mathscr{H}om(i_{u=0}^{-1}\mathscr{F}_q, i_{u=0}^{-1}\mathscr{G}_r).$$
    (Using the language in Nadler-Shende \cite[Section 2]{NadShen}, this is because the gapped condition fails for $\Lambda_r$ and $\Lambda_q$ as there exist Reeb chords whose lengths shrink to zero when $u \rightarrow 0$.) However, for tensor products we can easily get
    $$i_{u=0}^{-1}(D'\mathscr{F}_q \otimes \mathscr{G}_r) \simeq i_{u=0}^{-1}(D'\mathscr{F}_q) \otimes i_{u=0}^{-1}\mathscr{G}_r.$$
\end{remark}

\subsection{Duality and Exact Triangle}
    Now we are able to prove Theorem \ref{duality} and \ref{exactseq}.

\begin{theorem}[Sabloff Duality; Theorem \ref{duality}]\label{duality2}
    Let $M$ be orientable. For $\Lambda \subset T^{*,\infty}_{\tau > 0}(M \times \mathbb{R})$ and $\mathscr{F, G} \in Sh^b_\Lambda(M \times \mathbb{R})$ sheaves with perfect stalks and compact supports,
    $$Hom_+(\mathscr{F}_q, \mathscr{G}_r) \simeq D'Hom_-(\mathscr{G}_q, \mathscr{F}_r)[-n-1].$$
\end{theorem}

    Before proving the theorem we use the acyclic complex obtained in Corollary \ref{acyclic} to get a partial duality result. Again one may compare the result with the analogous ones in generating families.

\begin{proposition}\label{positive-negative}
    For $\Lambda \subset T^{*,\infty}_{\tau > 0}(M \times \mathbb{R})$ and $\mathscr{F, G} \in Sh^b_\Lambda(M \times \mathbb{R})$ sheaves with compact supports,
    $$Hom_+(\mathscr{F}_q, \mathscr{G}_r) \simeq Hom(\Bbbk_{u\leq 0}, \mathscr{H}om(\mathscr{F}_q, \mathscr{G}_r))[1].$$
\end{proposition}
\begin{proof}
    Consider the exact triangle
    $$\Gamma_{u\leq 0}(\mathscr{H}om(\mathscr{F}_r, \mathscr{G}_q)) \rightarrow \mathscr{H}om(\mathscr{F}_r, \mathscr{G}_q) \rightarrow i_{u>0 *}i_{u>0}^{-1} \mathscr{H}om(\mathscr{F}_r, \mathscr{G}_q) \xrightarrow{+1}.$$
    where $i_{u>0}: u^{-1}((0, +\infty)) \hookrightarrow M \times \mathbb{R}^2$ is the inclusion. We have $\Gamma(M \times \mathbb{R}^2, \mathscr{H}om(\mathscr{F}_r, \mathscr{G}_q)) \simeq 0$ by Corollary \ref{acyclic}. Therefore the assertion follows.
\end{proof}

\begin{proof}[Proof of Theorem \ref{duality2}]
    Since $\mathscr{F}_r, \mathscr{G}_q$ are constructible, by Propositions \ref{construct} we have
    \[\begin{split}
    Hom_+(\mathscr{F}_q, \mathscr{G}_r)  &\simeq Hom(\Bbbk_{u\leq 0}, \mathscr{H}om(\mathscr{F}_q, \mathscr{G}_r))[1] \\
    &\simeq Hom(\Bbbk_{u\leq 0}, D'(D'\mathscr{G}_r \otimes \mathscr{F}_q))[1] \\
    &\simeq Hom(\Bbbk_{u\leq 0} \otimes (D'\mathscr{G}_r \otimes \mathscr{F}_q), \Bbbk_{M \times \mathbb{R}^2})[1]
    \end{split}\]
    Note that $M$ is orientable with dimension $n$, we have $\omega_{M \times \mathbb{R}^2} = \Bbbk_{M \times \mathbb{R}^2}[-n-2]$. Since $SS^\infty(\mathscr{F}_q) \cap SS^\infty(\mathscr{G}_r) = \emptyset$, we know $D'\mathscr{G}_r \otimes \mathscr{F}_q \simeq \mathscr{H}om(\mathscr{G}_r, \mathscr{F}_q)$. Hence
    \[\begin{split}
    Hom_+(\mathscr{F}_q, \mathscr{G}_r)  &\simeq Hom(\Bbbk_{u\leq 0} \otimes D'\mathscr{G}_r \otimes \mathscr{F}_q, \omega_{M \times \mathbb{R}^2})[-n-1] \\
    &\simeq Hom(\Gamma(\Bbbk_{u\leq 0} \otimes D'\mathscr{G}_r \otimes \mathscr{F}_q), \Bbbk)[-n-1] \\
    &\simeq D'\Gamma(u^{-1}((-\infty,0]), D'\mathscr{G}_r \otimes \mathscr{F}_q)[-n-1].
    \end{split}\]
    The second last identity follows from the fact that the sheaf is compactly supported.

    Consider the diffeomorphism $\phi\colon M \times \mathbb{R}^2 \to M \times \mathbb{R}^2, (x, t, u) \mapsto (x, t-u, -u)$. Since $q = r \circ \phi$ and $r = q \circ \phi$, one can see that $\phi^{-1}\mathscr{F}_r = \mathscr{F}_q$, $\phi^{-1}\mathscr{G}_q = \mathscr{G}_r$, and
    $$\Gamma(u^{-1}((-\infty,0]), \mathscr{H}om(\mathscr{G}_r, \mathscr{F}_q)) \simeq \Gamma(u^{-1}([0,+\infty)), \mathscr{H}om(\mathscr{G}_q, \mathscr{F}_r)).$$
    This completes the proof.
\end{proof}

    Now we prove Thoerem \ref{exactseq}. The main ingredient is Sato's exact triangle Theorem \ref{sato}.

\begin{theorem}[Theorem \ref{exactseq}; Sabloff-Sato exact triangle]\label{exactseq2}
    Let $\Lambda \subset T^{*,\infty}_{\tau > 0}(M \times \mathbb{R})$ be a closed Legendrian, and $\mathscr{F, G} \in Sh^b_\Lambda(M \times \mathbb{R})$ be sheaves with perfect stalks and compact supports, we have an exact triangle
    $$Hom_-(\mathscr{F}, \mathscr{G}) \rightarrow Hom_+(\mathscr{F}, \mathscr{G}) \rightarrow \Gamma(\Lambda; \mu hom(\mathscr{F, G})) \xrightarrow{+1}.$$
    In particular, when the Maslov class and relative second Stiefel-Whitney class of $\Lambda$ vanish and $\mathscr{F} = \mathscr{G}$ are furthermore microlocal rank $r$ sheaves, we have
    $$Hom_-(\mathscr{F}, \mathscr{F}) \rightarrow Hom_+(\mathscr{F}, \mathscr{F}) \rightarrow C^*(\Lambda; \Bbbk^{r^2}) \xrightarrow{+1}.$$
\end{theorem}
\begin{proof}
    Consider Sato's exact triangle Theorem \ref{sato}
    $$D'\mathscr{F} \otimes \mathscr{G} \rightarrow \mathscr{H}om(\mathscr{F, G}) \rightarrow \pi_*(\mu hom(\mathscr{F, G})|_{T^{*,\infty}(M \times \mathbb{R})}) \xrightarrow{+1},$$
    where $\pi: T^{*,\infty}(M \times \mathbb{R}) \rightarrow M\times \mathbb{R}$ is the projection. Therefore, taking global sections gives the following exact triangle
    $$\Gamma(D'\mathscr{F} \otimes \mathscr{G}) \rightarrow Hom(\mathscr{F, G}) \rightarrow C^*(\Lambda; \mu hom(\mathscr{F, G})) \xrightarrow{+1}.$$

    To show the assertion of the theorem, we claim that there is a commutative diagram of exact triangles where vertical arrows are all quasi-isomorphisms given by Theorem \ref{computehom}
    \[\xymatrix{
    \Gamma(D'\mathscr{F} \otimes \mathscr{G}) \ar[r] & Hom(\mathscr{F, G}) \\
    \Gamma(u^{-1}([0,+\infty)), \mathscr{H}om(\mathscr{F}_q, \mathscr{G}_r)) \ar[u] \ar[r] & \Gamma(u^{-1}((0,+\infty)), \mathscr{H}om_+(\mathscr{F}_q, \mathscr{G}_r)) \ar[u].
    }\]
    Since $SS^\infty(\mathscr{F}_q) \cap SS^\infty(\mathscr{G}_r) = \varnothing$ and $\mathscr{F}_q, \mathscr{G}_r$ are constructible sheaves, we know that $\mathscr{H}om_+(\mathscr{F}_q, \mathscr{G}_r) \simeq D'\mathscr{F}_q \otimes \mathscr{G}_r$. Then we can rewrite the commutative diagram as
    \[\xymatrix{
    \Gamma(D'\mathscr{F} \otimes \mathscr{G}) \ar[r] & Hom\big(\mathscr{F}, \varprojlim\nolimits_{c>0}T_{c*}\mathscr{G}\big)\\
    \Gamma(u^{-1}([0,+\infty)), D'\mathscr{F}_q \otimes \mathscr{G}_r) \ar[u] \ar[r] & \Gamma(u^{-1}((0,+\infty)), \mathscr{H}om(\mathscr{F}_q, \mathscr{G}_r)) \ar[u],
    }\]
    where the left vertical map is induced by the restriction at $u^{-1}(0)$, and the right vertical map is induced by the restriction at $u^{-1}(c)$ as $c \to 0^+$. The above diagram can be decomposed into two commutative diagrams. First, we have a commutative diagram
    \[\xymatrix{
    \Gamma(D'\mathscr{F} \otimes \mathscr{G}) \ar[r] & Hom\big(\mathscr{F}, \varprojlim\nolimits_{c>0}T_{c*}\mathscr{G}\big)\\
    \Gamma(u^{-1}([0,+\infty)), D'\mathscr{F}_q \otimes \mathscr{G}_r) \ar[u] \ar[r] & \Gamma(u^{-1}([0,+\infty)), \mathscr{H}om(\mathscr{F}_q, \mathscr{G}_r)) \ar[u].
    }\]
    Since both vertical maps are induced by the restriction at $u_{-1}(0)$, and both horizontal maps are induced by the natural transformation $D'\mathscr{F} \otimes \mathscr{G} \to \mathscr{H}om(\mathscr{F}, \mathscr{G})$, this diagram commutes. Second, we have another commutative diagram
    \[\xymatrix{
    Hom(\mathscr{F}, \mathscr{G}) \ar[r] & Hom\big(\mathscr{F}, \varprojlim\nolimits_{c>0}T_{c*}\mathscr{G}\big)\\
    \Gamma(u^{-1}([0,+\infty)), \mathscr{H}om(\mathscr{F}_q, \mathscr{G}_r)) \ar[u] \ar[r] & \Gamma(u^{-1}((0,+\infty)), \mathscr{H}om(\mathscr{F}_q, \mathscr{G}_r)) \ar[u].
    }\]
    Here, the vertical maps are induced by the restrictions. Since the top horizontal map is induced the identification of sections on $u^{-1}(0)$ with $u^{-1}([0, c])$ using microlocal Morse lemma and then the restriction from $u^{-1}([0, c])$ to $u^{-1}(c)$, and the bottom horizontal map is induced by the restriction $u^{-1}([0,+\infty))$ to $u^{-1}([c, +\infty))$, this diagram also commutes.

    Finally, when $\mathscr{F} \simeq \mathscr{G}$ are microlocal rank $r$ sheaves and the Maslov class and relative second Stiefel-Whitney class vanish if $\Bbbk \neq \mathbb{Z}/2\mathbb{Z}$, and by Guillermou's Theorem \ref{Gui},
    $$\Gamma(\Lambda; \mu hom(\mathscr{F, F})) \simeq C^*(\Lambda; \Bbbk^{r^2}).$$
    This completes the proof.
\end{proof}

    The following corollary can be viewed as a version of degeneration to Morse flow trees in Legendrian contact homology (that certain pseudoholomorphic curves degenerate to Morse gradient flows) in for example \cite[Theorem 3.6, Part~(4)]{EESduality}. It says that certain sheaf homomorphism degenerates to Morse theory. This recovers Ike's result \cite[Lemma 4.9 \& Proposition 4.10 \& Theorem 4.13]{Ike} (where the isomorphism to Morse theory was obtained for $\mathbb{Z}/2\mathbb{Z}$).

\begin{corollary}\label{morseflow}
    Let $\Lambda \subset T^{*,\infty}_{\tau > 0}(M \times \mathbb{R})$ be a Legendrian and $\mathscr{F, G} \in Sh^b_\Lambda(M \times \mathbb{R})$ be constructible sheaves with perfect stalks and compact supports, then
    $$\Gamma(u^{-1}(0), \Gamma_{u\leq 0}(\mathscr{H}om(\mathscr{F}_q, \mathscr{F}_r)))[1] \simeq \Gamma(\Lambda; \mu hom(\mathscr{F, G})).$$
    In particular, when the Maslov class and relative second Stiefel-Whitney class of $\Lambda$ vanish and $\mathscr{F} = \mathscr{G}$ are microlocal rank $r$ sheaves, we have
    $$\Gamma(u^{-1}(0), \Gamma_{u\leq 0}(\mathscr{H}om(\mathscr{F}_q, \mathscr{F}_r)))[1] \simeq C^*(\Lambda; \Bbbk^{r^2}).$$
\end{corollary}
\begin{proof}
    We have an exact triangle as in Proposition \ref{positive-negative}
    \[\begin{split}
    \mathscr{H}om(\mathscr{F}_q, \mathscr{G}_r) \rightarrow i_{u>0*}i_{u>0}^{-1}\mathscr{H}om(\mathscr{F}_q, \mathscr{G}_r) \rightarrow \Gamma_{u\leq 0}(\mathscr{H}om(\mathscr{F}_q, \mathscr{G}_r))[1] \xrightarrow{+1},
    \end{split}\]
    where $i_{u>0}: u^{-1}((0,+\infty)) \hookrightarrow u^{-1}([0,+\infty))$ is the inclusion. By taking the sections on $u^{-1}([0, +\infty))$ and compare it with the exact triangle in Theorem \ref{exactseq}, we know that
    $$\Gamma_{u\leq0}(u^{-1}([0, +\infty)), \mathscr{H}om(\mathscr{F}_q, \mathscr{G}_r))[1] \simeq \Gamma(\Lambda; \mu hom(\mathscr{F, G})).$$
    However, note that the sheaf $\Gamma_{u\leq 0}(\mathscr{H}om(\mathscr{F}_q, \mathscr{G}_r))$ is supported in $u^{-1}((-\infty, 0])$, so
    $$\Gamma_{u\leq0}(u^{-1}([0, +\infty)), \mathscr{H}om(\mathscr{F}_q, \mathscr{G}_r))[1] \simeq \Gamma_{u\leq0}(u^{-1}(0), \mathscr{H}om(\mathscr{F}_q, \mathscr{G}_r))[1].$$
    This proves our assertion.
\end{proof}

\section{Persistence and Hamiltonian Isotopy}\label{persist}

\subsection{Persistence Modules and Sheaves}
    A persistent module is roughly speaking an $\mathbb{R}$-direct system of modules. It has been extensively studied \cite{Proximitypersist,Stablepersist} and has also been introduced in the context of sheaf theory by \cite{KSpersist}.

\begin{definition}
    Let $\Bbbk$ be a field. A persistence module $M_\mathbb{R}$ is a family $\{M_\alpha\}_{\alpha \in \mathbb{R}}$ of graded $\Bbbk$-modules, together with a family $\{f_{\alpha_0\alpha_1}: M_{\alpha_0} \rightarrow M_{\alpha_1}\}_{\alpha_0 \leq \alpha_1}$ such that $f_{\alpha_1\alpha_2} \circ f_{\alpha_0\alpha_1} = f_{\alpha_0\alpha_2}$ and $f_{\alpha\alpha} = \mathrm{id}_{M_\alpha}$. $M_\mathbb{R}$ is tame if for any $\alpha \in \mathbb{R}$, $\dim M_\alpha < \infty$.
\end{definition}

\begin{definition}
    Let $M_\mathbb{R}, N_\mathbb{R}$ be persistence modules. They are $(\epsilon, \epsilon')$-interleaved if there exists
    \begin{align*}
    \phi_\alpha \colon M_\alpha \rightarrow N_{\alpha+\epsilon}, \,\,\, \psi_\alpha \colon N_\alpha \rightarrow M_{\alpha+\epsilon'}
    \end{align*}
    such that $\phi_\beta \circ f^M_{\alpha,\beta} = f^N_{\alpha,\beta} \circ \phi_\alpha$, $\psi_\beta \circ f^N_{\alpha,\beta} = f^M_{\alpha,\beta} \circ \psi_\alpha$, and $f^M_{\alpha,\alpha+\epsilon+\epsilon'} = \psi_{\alpha+\epsilon} \circ \phi_{\alpha}, \,\, f^N_{\alpha,\alpha+\epsilon+\epsilon'} = \phi_{\alpha+\epsilon'} \circ \psi_\alpha.$
    The interleaving distance between $M_\mathbb{R}, N_\mathbb{R}$ is
    $$d(M_\mathbb{R}, N_\mathbb{R}) = \inf\{\epsilon + \epsilon' \mid M_\mathbb{R}, N_\mathbb{R} \text{ are }(\epsilon, \epsilon')\text{-interleaved}\}.$$
    Let $N_{\mathbb{R}+c}$ be the persistence module such that $(N_{\mathbb{R}+c})_\alpha = N_{\alpha+c}$. Then the translation invariant distance is
    $$\overline{d}(M_\mathbb{R}, N_\mathbb{R}) = \inf\{d(M_{\mathbb{R}}, N_{\mathbb{R}+c}) \mid c \in \mathbb{R} \}.$$
\end{definition}

    In this paper, we will use the language of constructible sheaves on $\mathbb{R}$ instead of persistence modules. Here is the classification result of these sheaves.

\begin{theorem}[Guillermou \cite{Guisurvey}*{Corollary 4.2.1}; Kashiwara-Schapira \cite{KSpersist}*{Theorem 1.17}]
    Let $\Bbbk$ be a field and $\mathscr{F} \in Sh_{\{\nu < 0\},c}^b(\mathbb{R})$ be a constructible sheaf. Then there exists an (index) set $A$ such that
    $$\mathscr{F} \simeq \bigoplus_{\alpha \in A}\Bbbk^{r_\alpha}_{(u_\alpha, v_\alpha]}[n_\alpha],$$
    and the collection of intervals $\{(u_\alpha, v_\alpha]\}_{\alpha \in A}$ is locally finite. Each interval $(u_\alpha, v_\alpha]$ is called a bar.
\end{theorem}

    Note that for any constructible sheaf $\mathscr{F} \in Sh_{\{\nu < 0\},c}^b(\mathbb{R})$, we can associate a tame persistence module by $M_\alpha = H^*\Gamma((-\infty, \alpha), \mathscr{F}).$ All definitions and results in persistence modules can be stated in 1-dimensional sheaf theory easily.

    Now we define the interleaving distance for sheaves in arbitrary dimensions.

\begin{definition}[\cite{AsanoIke,AsanoIkecomplete}]
    Let $\mathscr{F, G} \in Sh_{\{\tau > 0\}}(M \times \mathbb{R})$ be two sheaves. Let $T_c: \mathbb{R} \rightarrow \mathbb{R}$ be the translation $T_c(x, t) = (x, t+c)$. They are $(\epsilon, \epsilon')$-interleaved if there exists
    \[\begin{split}
    \phi: \mathscr{F} \rightarrow T_{\epsilon*}\mathscr{G}, \,\,\, \psi: \mathscr{G} \rightarrow T_{\epsilon'*}\mathscr{F},
    \end{split}\]
    such that the following diagrams commute
    $$t^\mathscr{F}_{0,\epsilon+\epsilon'} = T_{\epsilon*}\psi \circ \phi, \,\, t^\mathscr{G}_{0,\epsilon+\epsilon'} = T_{\epsilon'*}\phi \circ \psi$$
    where $t^\mathscr{H}_{a,b}: \mathscr{H}\rightarrow T_{a+b,*}\mathscr{H}$ is the natural map. The interleaving distance between $\mathscr{F, G}$ is
    $$d(\mathscr{F, G}) = \inf\{\epsilon + \epsilon' \mid \mathscr{F, G} \text{ are }(\epsilon, \epsilon')\text{-interleaved}\}.$$
    The translation invariant distance is $\overline{d}(\mathscr{F, G}) = \inf\{ d(\mathscr{F}, T_{c*}\mathscr{G}) \mid c \in \mathbb{R}\}$.
\end{definition}
\begin{remark}
    The original definition of Asano-Ike \cite{AsanoIke} uses four morphisms $\phi, \psi, \phi'$ and $\psi'$ to define the interleaving distance, but later they proved that their result also works for the above definition \cite[Remark 3.8]{AsanoIkecomplete}. Thus, we choose to use the simpler definition.
\end{remark}

\begin{example}\label{persistinterval}
    Consider the sheaves $\Bbbk_{(a_0,b_0]}$ and $\Bbbk_{(a_1,b_1]}$ in $Sh_{\{\nu<0\},c}^b(\mathbb{R})$. Since their singular supports satisfy $\nu<0$, we need to choose the translation in the negative direction $U_c: \mathbb{R} \rightarrow \mathbb{R}, \, x\mapsto x-c$. Then if $a, a', b, b'$ are distinct, by Proposition \ref{ssforhom}
    $$\mathscr{H}om(\Bbbk_{(a,b]}, \Bbbk_{(a',b']}) = \Bbbk_{[a,b) \cap (a',b']}.$$
    There exists a degree zero non-vanishing map iff $a' < a$ and $b' < b$.

\begin{figure}
  \centering
  \includegraphics[width=0.6\textwidth]{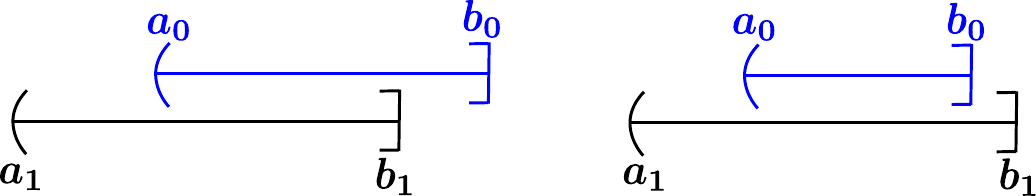}\\
  \caption{The sheaves $\Bbbk_{(a_0,b_0]}$ and $\Bbbk_{(a_1,b_1]}$ in two different cases.}\label{persistex}
\end{figure}

    Suppose $a_0>a_1, b_0>b_1$ and $a_0 < b_1$ (Figure \ref{persistex} left). One can show that the distance is
    $$d(\Bbbk_{(a_0,b_0]}, \Bbbk_{(a_1,b_1]}) = \inf\{\epsilon + \epsilon'\} = \max\{a_1-a_0, b_1-b_0\}.$$
    Suppose $a_0>a_1, b_0<b_1$ (Figure \ref{persistex} right). Then one can show that
    $$d(\Bbbk_{(a_0,b_0]}, \Bbbk_{(a_1,b_1]}) = \inf\{\epsilon + \epsilon'\} = (b_1-b_0)-(a_1-a_0).$$
    For the other two cases, one has similar results. In conclusion, one can see that the persistence distance is measuring how far the bars differ from each other.

    From the above computation, it is easy to show that we always have
    $$\overline{d}(\Bbbk_{(a_0,b_0]}, \Bbbk_{(a_1,b_1]}) = |(a_1-a_0) - (b_1-b_0)|.$$
    One can see that the translation invariant persistence distance is measuring how fast the lengths of bars change from one to the other.
\end{example}

    Here is a basic property we're going to use from time to time. It basically says that the persistence distance is a pseudo metric.

\begin{lemma}[\cite{AsanoIke}*{Proposition 4.10}]\label{triangleineq}
    Suppose $\mathscr{F, G}$ are $(a_0, b_0)$-interleaved, and $\mathscr{G, H}$ are $(a_1, b_1)$-interleaved. Then $\mathscr{F, H}$ are $(a_0+a_1, b_0+b_1)$-interleaved. In particular,
    $$d(\mathscr{F, H}) \leq d(\mathscr{F, G}) + d(\mathscr{G, H}).$$
    Moreover, $\overline{d}(\mathscr{F, H}) \leq \overline{d}(\mathscr{F, G}) + \overline{d}(\mathscr{G, H}).$
\end{lemma}

\subsection{Continuity under Hamiltonian Isotopy}
    Given a Hamiltonian isotopy $\varphi_H^s\,(s\in I)$ on $T^{*,\infty}_{\tau > 0}(M \times \mathbb{R})$, Guillermou-Kashiwara-Schapira defined an equivalence functor called sheaf quantization $\Phi_H^s: Sh_{\{\tau >0\}}(M \times \mathbb{R}) \rightarrow Sh^b_{\{\tau >0\}}(M \times \mathbb{R})$ (Theorem \ref{GKS}). Asano and Ike studied how the quantization of a Hamiltonian isotopy changes the interleaving distance. Recall that
    $$\|H\|_\text{osc}^\Lambda = \|H\|_\text{osc} = \int_0^1\left(\max_{(x, \xi, t) \in \varphi_H^s(\Lambda)} H_s(x, \xi, t) - \min_{(x, \xi, t) \in \varphi_H^s(\Lambda)} H_s(x, \xi, t)\right)\,ds.$$
    Given a Legendrian isotopy, there always exists a Hamiltonian that is constant away from a compact subset such that
    $$\|H\|_\text{osc}^\Lambda = \int_0^1\left(\max_{T^{*,\infty}_{\tau > 0}(M \times \mathbb{R})} H_s - \min_{T^{*,\infty}_{\tau > 0}(M \times \mathbb{R})} H_s\right)\,ds.$$
    See for example \cite[Theorem 2.6.2]{Geiges}. Then by the uniqueness theorem of Guillermou-Kashiwara-Schapira \cite{GKS}, we know that the equivalence functor of sheaves induced by the Hamiltonian isotopy only depends on the Legendrian isotopy. Therefore, we will always choose such a Hamiltonian and use $\|H\|_\text{osc}^\Lambda$ and $\|H\|_\text{osc}$ interchangeably.

\begin{theorem}[Asano-Ike \cite{AsanoIke}*{Proposition 4.10}, \cite{AsanoIkecomplete}*{Proposition 3.9}]\label{asanoike}
    Let $H$ be a compactly supported Hamiltonian on $T^{*,\infty}_{\tau > 0}(M \times \mathbb{R})$ and $\Phi_H^s\,(s \in I)$ be its sheaf quantization functor. Then for $\mathscr{H} \in Sh_{\{\tau > 0\}}^b(M \times \mathbb{R})$,
    $${d}(\mathscr{H}, \Phi_H^1(\mathscr{H})) \leq \|H\|_\text{osc}.$$
\end{theorem}

    Using this machinery, we now study our sheaf $\mathscr{H}om(\mathscr{F}_q, \mathscr{G}_r)$ for $\mathscr{F, G} \in Sh^b_c(M \times \mathbb{R})$. As we have seen in previous sections, the last $\mathbb{R}$ component encodes the length of all Reeb chords on $\Lambda$. Hence in order to get information on how the Reeb chords change under Hamiltonian isotopies, we project the sheaf to the last component $\mathbb{R}$ via $u\colon M \times \mathbb{R}^2 \rightarrow \mathbb{R},\,\,(x, t, u) \mapsto u$ and estimate the persistence structure on
    $$u_{*}\mathscr{H}om(\mathscr{F}_q, \mathscr{G}_r).$$
    By Lemma \ref{reebchord-hom}, this is a constructible sheaf in $Sh^b_{\nu < 0,c}(\mathbb{R})$. Here is our main result in this section.

\begin{definition}[Definition \ref{persistmod}]\label{persistmod2}
    For sheaves $\mathscr{F, G} \in Sh(M \times \mathbb{R})$, let
    $$\mathscr{H}om_{(-\infty,+\infty)}(\mathscr{F, G}) = u_*\mathscr{H}om(\mathscr{F}_q, \mathscr{G}_r).$$
\end{definition}
\begin{remark}
    For people who are familiar with Tamarkin categories \cite{Tamarkin1,GS}, we mention that by abusing notations to write $u\colon: M \times \mathbb{R}_u \to \mathbb{R}_u$, we have
    $$\mathscr{H}om_{(-\infty,+\infty)}(\mathscr{F, G}) = u_*\mathscr{H}om^\star(\mathscr{F, G}).$$
\end{remark}

\begin{theorem}[Theorem \ref{persistcontinue}]\label{persistcontinue2}
    Let $\Lambda \subset T^{*,\infty}_{\tau > 0}(M \times \mathbb{R})$ be a compact Legendrian, $H$ be a Hamiltonian on $T^{*,\infty}_{\tau > 0}(M \times \mathbb{R})$ and $\Phi_H^s\,(s \in I)$ be its sheaf quantization. Then for sheaves $\mathscr{F, G} \in Sh_\Lambda(M \times \mathbb{R})$ with compact support,
    $$\overline{d}(\mathscr{H}om_{(-\infty,+\infty)}(\mathscr{F, G}), \mathscr{H}om_{(-\infty,+\infty)}(\mathscr{F}, \Phi_H^1(\mathscr{G}))) \leq \|H\|_\text{osc}^\Lambda.$$
\end{theorem}
\begin{proof}
    We may assume that $H = c$ outside a compact subset and consider the contact Hamiltonian $H - c$. Since $\overline{d}(\mathscr{H}, \Phi_H^1(\mathscr{H})) \leq d(\mathscr{H}, \Phi_{H-c}^1(\mathscr{H}))$ and $\|H\|_\text{osc} = \|H - c\|_\text{osc}$, we may assume replace $H$ be $H - c$ and assume that the Hamiltonian is compactly supported. We will show that
    $$d(u_*\mathscr{H}om(\mathscr{F}_q, \mathscr{G}_r), u_*\mathscr{H}om(\mathscr{F}_q, (\Phi_H^1\mathscr{G})_r) \leq d(\mathscr{G}, (\Phi_H^1\mathscr{G})).$$
    For the above inequality, it is enough to show that if $\mathscr{G}, \mathscr{G}'$ are $(\epsilon,\epsilon')$-interleaved, then $u_*\mathscr{H}om(\mathscr{F}_q, \mathscr{G}_r), \,u_*\mathscr{H}om(\mathscr{F}_q, \mathscr{G}'_r)$ will also be $(\epsilon,\epsilon')$-interleaved. 

    First, let $T_c(x, t, u) = (x, t+c, u)$ and $T_c(x, t) = (x, t+c)$. Then $T_{c*}\mathscr{G}_r = (T_{c*}\mathscr{G})_r$. This shows that if $\mathscr{G}, \mathscr{G}'$ are $(\epsilon,\epsilon')$-interleaved, then $\mathscr{G}_r, \mathscr{G}'_r$ will also be $(\epsilon,\epsilon')$-interleaved. Then let $U_c(x, t, u) = (x, t, u-c)$. Since $r \circ T_c = r \circ U_c$ and $q = q \circ U_c$,
    $$\mathscr{H}om(\mathscr{F}_q, T_{c*}\mathscr{G}_r) = \mathscr{H}om(U_{c*}\mathscr{F}_q, U_{c*}\mathscr{G}_r) = U_{c*}\mathscr{H}om(\mathscr{F}_q, \mathscr{G}_r).$$
    For any morphism $\mathscr{G}_r \rightarrow T_{c*}\mathscr{G}'_r$ there is a canonical morphism
    $$\mathscr{H}om(\mathscr{F}_q, \mathscr{G}_r) \rightarrow \mathscr{H}om(\mathscr{F}_q, T_{c*}\mathscr{G}'_r).$$
    Therefore there is always a canonical morphism
    $$\mathscr{H}om(\mathscr{F}_q, \mathscr{G}_r) \rightarrow U_{c*}\mathscr{H}om(\mathscr{F}_q, \mathscr{G}'_r).$$
    By abuse of notations, we also write $U_c: \mathbb{R} \rightarrow \mathbb{R}, \,u \mapsto u-c$. Note that $u \circ U_c = U_c$, so one will have a canonical morphism
    $$u_*\mathscr{H}om(\mathscr{F}_q, \mathscr{G}_r) \rightarrow U_{c*}u_*\mathscr{H}om(\mathscr{F}_q, \mathscr{G}'_r).$$
    This shows that if $\mathscr{G}_r, \mathscr{G}'_r$ are $(\epsilon,\epsilon')$-interleaved, then $u_*\mathscr{H}om(\mathscr{F}_q, \mathscr{G}_r), \,u_*\mathscr{H}om(\mathscr{F}_q, \mathscr{G}'_r)$ will also be $(\epsilon,\epsilon')$-interleaved. Then the
result follows from Theorem \ref{asanoike}.
\end{proof}

    As an example, we will try to understand the persistence module $\mathscr{H}om_{(-\infty,+\infty)}(\Bbbk_{(x_0,t_0)}, \mathscr{F})$ where $\Bbbk_{(x_0,t_0)}$ is the skyscraper sheaf at $(x_0, t_0) \in M \times \mathbb{R}$. While $\Bbbk_{(x_0,t_0)} \notin Sh_{\tau > 0}^b(M \times \mathbb{R})$, we claim that all the previous results are still valid  as long as $\mathscr{F} \in Sh_{\tau > 0}^b(M \times \mathbb{R})$.

\begin{lemma}\label{reebchord-hom2}
    For $\Lambda \subset T^{*,\infty}_{\tau > 0}(M \times \mathbb{R})$ and $\mathscr{F} \in Sh_\Lambda(M \times \mathbb{R})$,
    $$SS^\infty(\mathscr{H}om((\Bbbk_{(x_0,t_0)})_q, \mathscr{F}_r)) \cap \mathrm{Graph}(du) = \varnothing.$$
    On the other hand, there is an injection from
    $$SS^\infty(\mathscr{H}om((\Bbbk_{(x_0,t_0)})_q, \mathscr{F}_r)) \cap \mathrm{Graph}(-du)$$
    to the set of directed Reeb chords (Reeb trajectories for some positive or negative time $u$) $\mathcal{Q}_\pm(T_{(x_0,t_0)}^{*,\infty}(M \times \mathbb{R}), \Lambda) = \{\gamma\colon[0,u] \rightarrow T^{*,\infty}_{\tau>0}(M \times \mathbb{R}) \mid \gamma(s) = (x_0, t_0+s, \xi, \tau),\,\gamma(u) \in \Lambda \}$.
\end{lemma}

    The proof is identical as Lemma \ref{reebchord-hom}. Since this Lemma still holds, one can easily see that all previous discussions in this section still hold for
    $$\mathscr{H}om_{(-\infty,+\infty)}(\Bbbk_{(x_0,t_0)}, \mathscr{F}) = u_*\mathscr{H}om((\Bbbk_{(x_0,t_0)})_q, \mathscr{F}_r).$$

\begin{figure}
  \centering
  \includegraphics[width=0.8\textwidth]{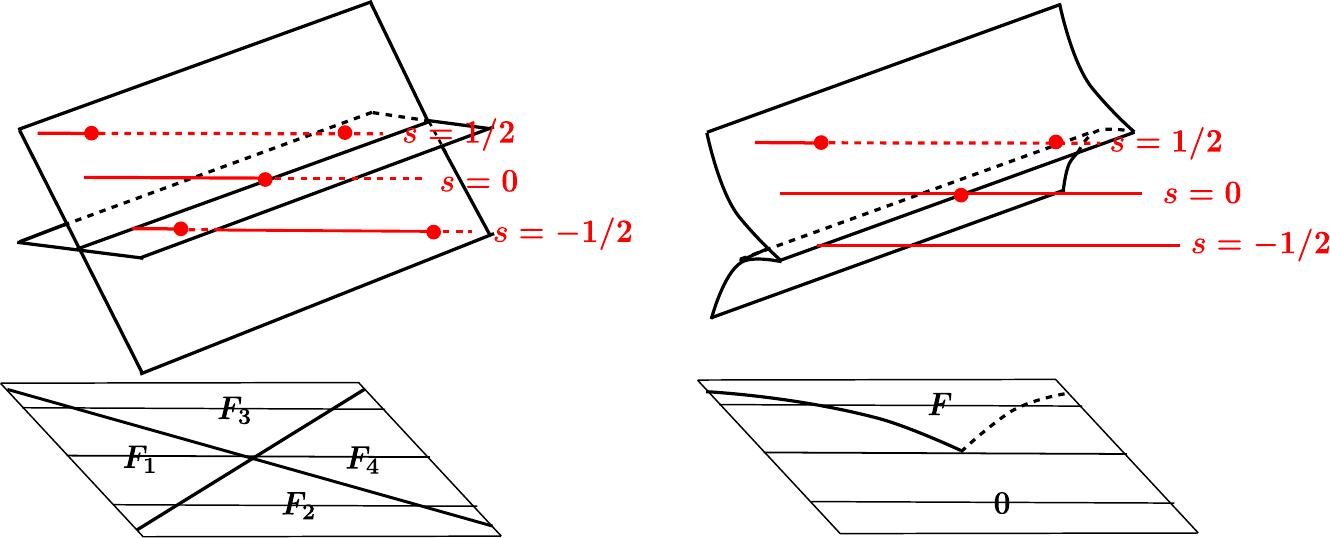}\\
  \caption{Birth-death of Reeb chords (on the right) and swapping of Reeb chords (on the left). On the top, the black Legendrians are $(\Lambda_s)_r$ while the red curves are $(T^{*,\infty}_{(0,1)}\mathbb{R}^2)_q$. The $u$-axis is horizontal, the $t$-axis is vertical, while the $s$-axis is pointing into the blackboard.}\label{isotopyex}
\end{figure}

\begin{example}\label{fiberlegendrian}
    The first example is about birth-death of Reeb chords (Figure \ref{isotopyex} right). We consider a family of Legendrians $\Lambda_s = \{(x, \pm 3(x+s)^{1/2}/2, (x+s)^{3/2}) \mid x+s \geq 0\} \subset J^1(\mathbb{R})$ whose front projections are standard cusps $\{(x, t) \mid t^2 = (x+s)^3\}$. Consider Reeb chords from $\Lambda_s$ to the fiber $T^{*,\infty}_{(0,1)}\mathbb{R}^2$. At $s = 0$, a pair of Reeb chords are created.

    For $F \in \mathrm{Mod}(\Bbbk)$, consider the sheaf
    $$\mathscr{F}_s = F_{\{(x, t) \mid 0 \leq t < (x+s)^{3/2} \text{ or } -(x+s)^{3/2} \leq t < 0\}}.$$
    Then consider $u_*\mathscr{H}om((\Bbbk_{(0,1)})_q, \mathscr{F}_r)$. One can see that
    $$u_*\mathscr{H}om((\Bbbk_{(0,1)})_q, (\mathscr{F}_s)_r)_{u=c} = \Gamma(\mathbb{R}, \mathscr{H}om(\Bbbk_{(0,1)}, T_{c*}\mathscr{F}_s)) = \mathscr{F}_s|_{(0,1-c)}.$$
    Therefore when $s\leq 0$, we have $\mathscr{H}om_{(-\infty,+\infty)}(\Bbbk_{(0,1)}, \mathscr{F}_s) = 0$. When $s> 0$,
    $$\mathscr{H}om_{(-\infty,+\infty)}(\Bbbk_{(0,1)}, \mathscr{F}_s) = F_{(1-s^{3/2}, 1+s^{3/2}]}.$$
    In other words, the birth of Reeb chords creates a new bar.
\end{example}

    When the Hamiltonian isotopy swaps the length of two Reeb chords, the behaviour of the sheaf $\mathscr{H}om_{(-\infty,+\infty)}(-, -)$ under the isotopy may be more complicated. However, there are still very specific cases where the behaviour is relatively clear.

\begin{example}\label{persistswap}
    The second example is a specific case of swapping of Reeb chords (Figure \ref{isotopyex} left). We consider a family of Legendrians $\Lambda_s = \{(x, \pm 1, \pm(x+s)) \mid x \in \mathbb{R}\} \subset J^1(\mathbb{R})$ whose front projections are standard crossings $\{(x, t) \mid t = \pm (x + s)\}$. Consider Reeb chords from $\Lambda_s$ to the fiber $T^{*,\infty}_{(0,1)}\mathbb{R}^2$. At $s = 0$, a pair of Reeb chords are swapped.

    For $F_1, F_2, F_3, F_4 \in \mathrm{Mod}(\Bbbk)$, consider the sheaf
    \begin{gather*}
    \mathscr{F}_s|_{\{(x, y)\mid t \geq |x|\}} = F_1|_{\{(x, y)\mid t \geq |x+s|\}}, \,\,\, \mathscr{F}_s|_{\{(x, y)\mid x < 0, -t < x+s \leq t\}} = F_2|_{\{(x, y)\mid x < 0, -t < x+s \leq t\}}, \\ 
    \mathscr{F}_s|_{\{(x, y)\mid x > 0, -t < x+s \leq t\}} = F_3|_{\{(x, y)\mid x > 0, -t < x+s \leq t\}}, \,\,\, \mathscr{F}_s|_{\{(x, y)\mid t < -|x+s|\}} = F_4|_{\{(x, y)\mid t < -|x+s|\}}.
    \end{gather*}
    The sheaf $\mathscr{F}_s$ is characterized by the diagram (see Example \ref{combin-model} or \cite[Section 3.3]{STZ})
    \[\xymatrix{
    F_1 \ar[r] \ar[d] & F_3 \ar[d] \\
    F_2 \ar[r]& F_4.
    }\]
    where $\mathrm{Tot}\,(F_1 \rightarrow F_2 \oplus F_3 \rightarrow F_4) \simeq 0$. Then $u_*\mathscr{H}om((\Bbbk_{(0,1)})_q, \mathscr{F}_r)_{u=c} = \mathscr{F}_s|_{(x,t)=(0,1-c)}$. When $s < 0$, $\mathscr{H}om_{(-\infty,+\infty)}(\Bbbk_{(0,1)}, \mathscr{F}_s)$ is determined by the diagram
    $$F_1 \longrightarrow F_2 \longrightarrow F_4.$$
    When $s> 0$, $\mathscr{H}om_{(-\infty,+\infty)}(\Bbbk_{(0,1)}, \mathscr{F}_s)$ is characterized by the diagram
    $$F_1 \longrightarrow F_3 \longrightarrow F_4.$$
    Decomposing the sheaf as $\bigoplus_{\alpha \in A}\Bbbk_{(a_\alpha, b_\alpha]}^{r_\alpha}[n_\alpha]$, we have for $s < 0$,
    \begin{align*}
    \mathscr{H}om_{(-\infty,+\infty)}(\Bbbk_{(0,1)}, \mathscr{F}_s) \simeq & \,V_{(-\infty,+\infty)} \oplus V_{(-\infty,-s]} \oplus V_{(-\infty,s]} \\
    & \, \oplus V_{(-s,s]} \oplus V_{(-s,+\infty)} \oplus V_{(s,+\infty)}.
    \end{align*}
    When $s> 0$,
    \begin{align*}
    \mathscr{H}om_{(-\infty,+\infty)}(\Bbbk_{(0,1)}, \mathscr{F}_s) \simeq & \,U_{(-\infty,+\infty)} \oplus U_{(-\infty,-s]} \oplus U_{(-\infty,s]} \\
    & \, \oplus U_{(-s,s]} \oplus U_{(-s,+\infty)} \oplus U_{(s,+\infty)}.
    \end{align*}
    Using the condition $\mathrm{Tot}\,(F_1 \rightarrow F_2 \oplus F_3 \rightarrow F_4) \simeq 0$, one can show that
    \begin{gather*}
    V_{(-s,s]} \simeq U_{(-s,s]} \simeq 0, \\
    V_{(-\infty,-s]} \simeq U_{(-\infty,s]}, \,\, V_{(-s,+\infty)} \simeq U_{(s, +\infty)}, \\
    V_{(-\infty,s]} \simeq U_{(-\infty,-s]}, \,\, V_{(s,+\infty)} \simeq U_{(-s, +\infty)}, \\
    V_{(-\infty,+\infty)} \simeq U_{(-\infty,+\infty)}.
    \end{gather*}
    Hence in this specific case, swapping of Reeb chords swaps starting/ending points of bars.
\end{example}

\section{Reeb Chord Estimation}\label{reebchordsection}

    Our goal in this section is to relate the number of Reeb chords with $Hom_+(\mathscr{F, G})$ and $\mathscr{H}om_{(-\infty,+\infty)}(\mathscr{F, G})$, and hence finish the proof of Theorem \ref{puresheaf}, \ref{mixedsheaf} and \ref{displaceable}.

\subsection{Local Calculation for Microstalks}
    By Lemma \ref{reebchord-hom}, we know that certain covectors in the singular support of $\mathscr{H}om(\mathscr{F}_q, \mathscr{G}_r)$ correspond to Reeb chords. The microlocal Morse inequality Proposition \ref{morseinequality} relates the global section of sheaves to its microstalks. Hence it suffices to determine if the ranks of the microstalks $\Gamma_{u\leq u_i}(\mathscr{H}om(\mathscr{F}_q, \mathscr{G}_r))_{(x_i, t_i, u_i)}$. 
    Using Corollary \ref{morseflow} (where the Legendrian is taken to be $\Lambda \cup T_{u_i}(\Lambda)$), we have
    $$\Gamma(M \times \mathbb{R}, \Gamma_{u\leq u_i}(\mathscr{H}om(\mathscr{F}_q, \mathscr{G}_r))) \simeq \Gamma(\Lambda \cup T_{u_i}(\Lambda), \mu hom(\mathscr{F}, T_{u_i*}\mathscr{G}))[-1].$$
    Here is the main result that obtains the microstalks by local computations.

\begin{proposition}\label{localcompute}
    For $\Lambda \subset T^{*,\infty}_{\tau > 0}(M \times \mathbb{R})$ a chord generic Legendrian and $\mathscr{F, G} \in Sh^b_\Lambda(M \times \mathbb{R})$ sheaves with perfect stalks whose microstalks are $F$ and $G$, let $\{(x_i, 0, t_i, 0, u_i, \nu_i)\}_{i\in I}$ be the set
    $$((-\Lambda_q) + \Lambda_r) \cap \{(x, 0, t, 0, u, \nu) \mid u>0, \nu < 0\}.$$
    Suppose $(x_i, t_i, u_i)$ corresponds to a degree $d_i$ Reeb chord from $(x_i, \xi_i, t_i, 1)$ to $(x_i, \xi_i, t_i+u_i, 1)$ in Lemma \ref{reebchord-hom}. Then
    $$\Gamma_{u\leq u_i}(\mathscr{H}om(\mathscr{F}_q, \mathscr{G}_r))_{(x_i, t_i, u_i)} \simeq \mu hom(\mathscr{F}, T_{u_i*}\mathscr{G})_{(x_i,\xi_i, t_i, 1)}[-1] \simeq Hom(F, G)[-d_i].$$
\end{proposition}
\begin{remark}
    When there is a Morse-Bott family of Reeb chords, we have a similar result by Ike \cite[Theorem 4.14]{Ike}. Here, as opposed to the result of Ike, we do not need to appeal to contact transformation \cite[Theorem 7.2.1]{KS}.
\end{remark}

    First, we note that $\mu hom(\mathscr{F}, T_{u_i*}\mathscr{G})_{(x_i,\xi_i, t_i, 1)}[-1]$ only depends on the microstalk of $\mathscr{F}$ and $\mathscr{G}$ by Proposition \ref{microstalk}.

\begin{lemma}\label{localcomputemicrostalk}
    For $\Lambda \subset T^{*,\infty}_{\tau > 0}(M \times \mathbb{R})$ a chord generic Legendrian. For any sheaves $\mathscr{F, G} \in Sh^b_\Lambda(M \times \mathbb{R})$ with perfect stalks whose microstalks are $F$ and $G$,
    $$\Gamma_{u\leq u_i}(\mathscr{H}om(\mathscr{F}_q, \mathscr{G}_r))_{(x_i, t_i, u_i)} \simeq \mu hom(\mathscr{F}, T_{u_i*}\mathscr{G})_{(x_i,\xi_i, t_i, 1)}[-1]$$
    are all isomorphic.
\end{lemma}

    Recall from Section \ref{prelimcontact} that the degree of a Reeb chord $\gamma \in \mathcal{Q}_+(\Lambda)$ is defined as follows (recall that $\gamma \in \mathcal{Q}_+(\Lambda)$ means $\gamma$ is a trajectory of the Reeb flow for some positive time so that $\gamma'(t) > 0$). Suppose at $a = (x, \xi, t, \tau)$ and $b = (x, \xi, t+u, \tau)$ ($u>0$),
    $$n - \deg(\gamma) = d(a) - d(b) + \mathrm{ind}(D^2h_{ab}) - 1,$$
    where $d(b), d(a)$ are Maslov potentials at $b, a$, and $h_{ab} = h_b-h_a$ for $h_b, h_a$ whose graphs at $b, a$ are $\pi_\text{front}(\Lambda)$. By Morse lemma, we assume that in local coordinates
    $$h_b(x) = u, \,\,\, h_a(x) = -\sum_{i\leq k}x_i^2 + \sum_{j\geq k+1}x_j^2.$$
    Let $U_{x_i,t_i}$ be a small neighbourhood of $(x_i, t_i)$ and $\epsilon > 0$ be a small positive number. We write
    $$U^- = U_{x_i,t_i} \times \{u_i-\epsilon\}, \,\,\, U^+ = U_{x_i,t_i} \times \{u_i+\epsilon\}.$$
    Consider the stratification of $U^\pm$ by the graphs of $h_a$ and $h_b$. We will write
    \begin{gather*}
    U^\pm \cap \{(x, t) \mid t > h_b(x)\} = U_{q,0}, \,\,\, U^\pm \cap \{(x, t) \mid t \leq h_b(x)\} = U_{q,1}, \\
    U^\pm \cap \{(x, t) \mid t < h_a(x) + u_i\pm \epsilon/2\} = U_{r,0}^\pm, \,\,\, U^\pm \cap \{(x, t) \mid t \geq h_a(x) + u_i\pm \epsilon/2\} = U_{r,1}^\pm.
    \end{gather*}

    By microlocal Morse lemma and Lemma \ref{reebchord-hom}, it suffices to calculate
   \begin{align*}
    \mathrm{Cone}\big( &\Gamma(U_{x_i,t_i} \times (u_i-2\epsilon, u_i), \mathscr{H}om(\mathscr{F}_q, \mathscr{F}_r)) \\
    &\rightarrow \Gamma(U_{x_i,t_i} \times (u_i, u_i+2\epsilon), \mathscr{H}om(\mathscr{F}_q, \mathscr{F}_r))\big)[-1].
    \end{align*}
    Note that $(\Lambda_q \cup \Lambda_r) \cap T^{*,\infty}(U_{x_i,t_i} \times (u_i-2\epsilon, u_i))$ and $(\Lambda_q \cup \Lambda_r) \cap T^{*,\infty}(U_{x_i,t_i} \times (u_i, u_i+2\epsilon))$ are movies of Legendrian isotopies. Then by Guillermou-Kashiwara-Schapira Theorem \ref{GKS} \cite{GKS}, it suffices to compute
    \begin{align*}
    \mathrm{Cone}\big( &\Gamma(U_{x_i,t_i} \times \{u_i-\epsilon\}, \mathscr{H}om(\mathscr{F}_q, \mathscr{F}_r))) \\
    &\rightarrow \Gamma(U_{x_i,t_i} \times \{u_i+\epsilon\}, \mathscr{H}om(\mathscr{F}_q, \mathscr{F}_r))\big)[-1].
    \end{align*}

    Since $\Lambda_q \cap \Lambda_r = \varnothing$, by Proposition \ref{ssforhom}
    $$\mathscr{H}om(\mathscr{F}_q, \mathscr{F}_r) \simeq D'\mathscr{F}_q \otimes \mathscr{F}_r.$$
    By Lemma \ref{localcomputemicrostalk}, as in Example \ref{combin-model} \cite[Section 3.3]{STZ} we assume that
    \begin{gather*}
    D'\mathscr{F}_q|_{U_{q,0}} \simeq (D'F)_{U_{q,0}}[-d(b)], \,\,\, \mathscr{G}_r|_{U_{r,0}^\pm} \simeq G_{U_{r,0}^\pm}[d(a)+1].
    \end{gather*}
    The following lemma ensures that when the microstalk of $\mathscr{F}$ is $F[d(b)]$, the microstalk of $D'\mathscr{F}$ is indeed $D'F[-d(b)-1]$, which justifies our assumption.

\begin{lemma}
    Let $\mathscr{F} \in Sh_{\nu^{*,\infty}_{\mathbb{R}^n \times \mathbb{R}_{>0},-}\mathbb{R}^{n+1}}^b(\mathbb{R}^{n+1})$ and $\varphi(x, t) = t$. Then
    $$\Gamma_{\varphi\leq 0}(D'\mathscr{F})_{(0,\dots,0)} = D'\Gamma_{\varphi\geq 0}(\mathscr{F})_{(0,\dots,0)}[-1].$$
\end{lemma}
\begin{proof}
    We assume that $\mathscr{F}|_{\mathbb{R}^n \times [0,+\infty)} = (F_1)_{\mathbb{R}^n \times [0,+\infty)}$ and $\mathscr{F}|_{\mathbb{R}^n \times (-\infty,0)} = (F_0)_{\mathbb{R}^n \times (-\infty,0)}$. Then we have an exact triangle
    $$\Gamma_{\varphi\geq 0}(\mathscr{F})_{(0,\dots,0)} \rightarrow F_1 \rightarrow F_0 \xrightarrow{+1}.$$
    Therefore by taking the dual we have
    $$D'F_0 \rightarrow D'F_1 \rightarrow D'\Gamma_{\varphi\geq 0}(\mathscr{F})_{(0,\dots,0)} \xrightarrow{+1}.$$
    We claim that $D'\mathscr{F}|_{\mathbb{R}^n \times (0,+\infty)} = (D'F_1)_{\mathbb{R}^n \times (0,+\infty)}$ and $D'\mathscr{F}|_{\mathbb{R}^n \times (-\infty,0]} = (D'F_0)_{\mathbb{R}^n \times (-\infty,0]}$. We will only check the stalk at $\mathbb{R}^n \times \{0\}$. In fact, since $SS^\infty(\mathscr{F}) = \nu^{*,\infty}_{\mathbb{R}^n \times \mathbb{R}_{>0},-}\mathbb{R}^{n+1}$, we have $SS^\infty(D'\mathscr{F}) = \nu^{*,\infty}_{\mathbb{R}^n \times \mathbb{R}_{>0},+}\mathbb{R}^{n+1}$. By microlocal Morse lemma Proposition \ref{morselemma}, 
    $$D'\mathscr{F}_{(0,\dots,0)} \simeq \Gamma(\mathbb{R}^{n+1}, D'\mathscr{F})
    \simeq \Gamma(\mathbb{R}^{n} \times (-\infty, 0), D'\mathscr{F}) \simeq D'F_0.$$
    Therefore we know that
    $$\Gamma_{\varphi\leq 0}(D'\mathscr{F})_{(0,\dots,0)} = \mathrm{Cone}(D'F_0 \rightarrow D'F_1)[-1] \simeq D'\Gamma_{\varphi\geq 0}(\mathscr{F})_{(0,\dots,0)}[-1].$$
    This proves the assertion.
\end{proof}

\begin{figure}
  \centering
  \includegraphics[width=0.9\textwidth]{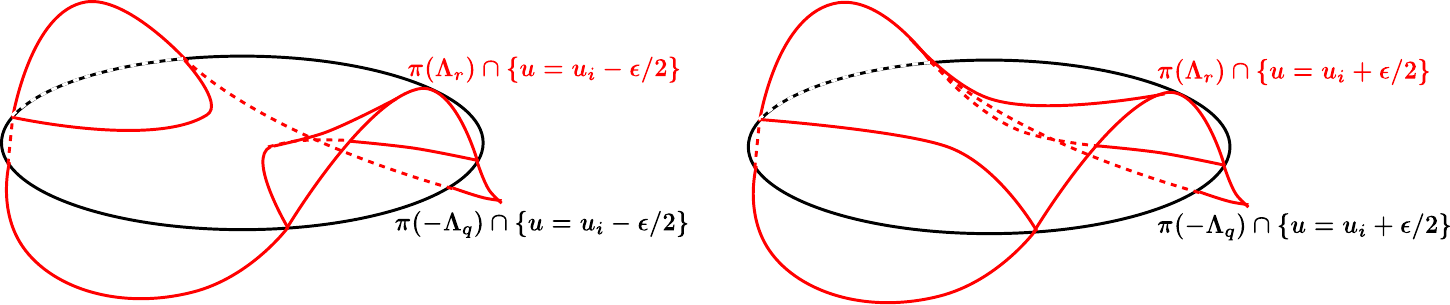}\\
  \caption{When $n=2$ and $k=1$, the open subsets $U^-$ (on the left) and $U^+$ (on the right).}\label{reebchord1}
\end{figure}

    With these preparations, we can now prove Proposition \ref{localcompute}.

\begin{proof}[Proof of Proposition \ref{localcompute}]
    First, consider the sections at $u=u_i-\epsilon$. Suppose $0 \leq k < n$. Since $U_{q,0} \cap U_{r,0}^- \cong D^{k+2} \times S^{n-k-1}$ and $\partial(U_{q,0} \cap U_{r,0}^-) \cong D^{k+1} \times S^{n-k-1}$, we know that
    \begin{align*}
    \Gamma (U^-, D'\mathscr{F}_q \otimes \mathscr{G}_r) &\simeq \Gamma(U^-, (D'F \otimes G)_{U_{q,0}\cap U_{r,0}^-})[d(a)-d(b)-1] \\
    &\simeq C^*(D^{k+2}, D^{k+1}; D'F \otimes G)[d(a)-d(b)-1] \simeq 0.
    \end{align*}
    Suppose $k = n$. Since $U_{q,0} \cap U_{r,0}^- \cong \varnothing$, the sections are
    $$\Gamma(U^-, D'\mathscr{F}_q \otimes \mathscr{G}_r) \simeq 0.$$
    Then consider sections at $u=u_i+\epsilon$. Since $U_{q,0} \cap U_{r,0}^+ \cong D^{k+1} \times D^{n-k}$ and $\partial(U_{q,0} \cap U_{r,0}^+) \cong S^k \times D^{n-k}$ for any $0 \leq k \leq n$, we have
    \begin{align*}
    \Gamma &(U^+, D'\mathscr{F}_q \otimes \mathscr{G}_r) \simeq \Gamma(U^+, (D'F \otimes G)_{U_{q,0}\cap U_{r,0}^+})[d(a)-d(b)-1] \\
    &\simeq C^*(D^{k+1}, S^k; D'F \otimes G)[d(a)-d(b)-1] = D'F \otimes G[d(a) - d(b) + k].
    \end{align*}
    Therefore, the microstalk is given by
    $$\mathrm{Cone}(\Gamma(U^-, D'\mathscr{F}_q \otimes \mathscr{G}_r) \to  \Gamma(U^+, D'\mathscr{F}_q \otimes \mathscr{G}_r) )[-1] \simeq D'F \otimes G[d(a)-d(b)+k-1].$$
    Hence the proof is completed.
\end{proof}

    When $u < 0$, we consider $\{(x_i, 0, t_i, 0, u_i, \nu_i)\}_{i\in I}$ be the set
    $$((-\Lambda_q) + \Lambda_r) \cap \{(x, 0, t, 0, u, \nu) \mid u<0, \nu < 0\}.$$
    The calculation in Proposition \ref{localcompute} still holds, except that we have to be careful about the gradings.

    We always assume that in our local model, when $u$ increases, the point $a$ is moving up in the horizontal $u$-direction passing through $b$. In the case of $u>0$, the point $(x_i, 0, t_i, 0, u_i, \nu_i)$ comes from a Reeb chord connecting $a$ to $b$ where $b$ is above $a$, and as $u > 0$ increases from $0$, $b$ is fixed and $a$ is moving up. $\mathrm{Graph}(h_b), \mathrm{Graph}(h_a)$ are local models of $\pi_\text{front}(\Lambda)$ at $b, a$, and in local coordinates
    $$h_b(x) = u_i > 0, \,\,\, h_a(x) = -\sum_{i\leq k}x_i^2 + \sum_{j\geq k+1}x_j^2.$$
    However in the case of $u<0$, the point $(x_i, 0, t_i, 0, u_i, \nu_i)$ will then come from a Reeb chord connecting $b$ to $a$ where $a$ is above $b$, and now as $u < 0$ increases to $0$, $a$ is moving up and yet $b$ is fixed. In local coordinates
    $$h_b(x) = u_i < 0, \,\,\, h_a(x) = -\sum_{i\leq k}x_i^2 + \sum_{j\geq k+1}x_j^2.$$
    Then that the Morse index $\mathrm{ind}(D^2h_{ba})$ where $h_{ba} = h_a - h_b$ will become $k$ instead of $n-k$ (the order of $a$ and $b$ are switched as their heights are switched). Thus if the degree of the original chord is $d_i$, the degree shifting will be
    $$-d(b)-1 + d(a) - k = -d(b)-1+d(a)-\mathrm{ind}(D^2h_{ba}) = -n+d_i-2.$$

\begin{proposition}\label{localcomputemix}
    For $\Lambda \subset T^{*,\infty}_{\tau > 0}(M \times \mathbb{R})$ a chord generic Legendrian and $\mathscr{F, G} \in Sh^b_\Lambda(M \times \mathbb{R})$ sheaves with perfect stalks and microstalks $F$ and $G$, let $\{(x_i, 0, t_i, 0, u_i, \nu_i)\}_{i\in I}$ be the set
    $$((-\Lambda_q) + \Lambda_r) \cap \{(x, 0, t, 0, u, \nu) \mid u<0, \nu < 0\}.$$
    Suppose $(x_i, t_i, u_i)$ corresponding to a degree $d_i$ Reeb chord  starting from $(x_i, \xi_i, t_i, 1)$ to $(x_i, \xi_i, t_i+u_i, 1)$ in Lemma \ref{reebchord-hom}. Then
    $$\Gamma_{u\leq u_i}(\mathscr{H}om(\mathscr{F}_q, \mathscr{G}_r))_{(x_i, t_i, u_i)} \simeq \mu hom(\mathscr{F}, T_{u_i*}\mathscr{G})_{(x_i,\xi_i,t_i,1)} \simeq Hom(F, G)[-n+d_i-2].$$
\end{proposition}

\subsection{Application to the Morse Inequality}
    Combining the previous propositions, we are able to prove the main theorems \ref{puresheaf} and \ref{mixedsheaf} using duality exact sequence. The main ingredient for these theorems will be the following Morse inequalities.

\begin{theorem}[Theorem \ref{reebinequality}]\label{reebinequality2}
    For $\Lambda \subset T^{*,\infty}_{\tau > 0}(M \times \mathbb{R})$ a closed chord generic Legendrian and $\mathscr{F} \in Sh^b_\Lambda(M \times \mathbb{R})$ a microlocal rank $r$ sheaf, let $\mathcal{Q}_j(\Lambda)$ be the set of degree $j$ Reeb chords on $\Lambda$. Suppose $\mathrm{supp}(\mathscr{F})$ is compact. Then for any $k \in \mathbb{Z}$
    $$r^2\sum_{j\leq k}(-1)^{k-j}|\mathcal{Q}_j(\Lambda)| \geq \sum_{j\leq k}(-1)^{k-j}\dim H^jHom_+(\mathscr{F, F}).$$
    In particular, for any $j\in \mathbb{Z}$, $r^2|\mathcal{Q}_j(\Lambda)| \geq \dim H^jHom_+(\mathscr{F, F})$.
\end{theorem}

\begin{theorem}\label{mixedinequality}
    For $\Lambda \subset T^{*,\infty}_{\tau > 0}(M \times \mathbb{R})$ a closed chord generic Legendrian and $\mathscr{F} \in Sh^b_\Lambda(M \times \mathbb{R})$ a sheaf with prefect microstalk $F$, let $\mathcal{Q}_j(\Lambda)$ be the set of degree $j$ Reeb chords on $\Lambda$. Suppose $\mathrm{supp}(\mathscr{F})$ is compact. Then for any $k \in \mathbb{Z}$
    $$\sum_{j\leq k}(-1)^{k-j}\sum_{i\in \mathbb{Z}}\dim H^iHom(F, F)|\mathcal{Q}_{j-i}(\Lambda)| \geq \sum_{j\leq k}(-1)^{k-j}\dim H^jHom_+(\mathscr{F, F}).$$
    In particular, for any $j\in \mathbb{Z}$,
    $$\sum_{i\in \mathbb{Z}}\dim H^iHom(F, F)|\mathcal{Q}_{j-i}(\Lambda)| \geq \dim H^jHom_+(\mathscr{F, F}).$$
\end{theorem}

\begin{proof}[Proof of Theorem \ref{reebinequality2} and \ref{mixedinequality}]
    By Proposition \ref{localcompute}, it suffices to prove a Morse-type inequality on the rank of microstalks
    $$\Gamma_{u\leq u_i}(\mathscr{H}om(\mathscr{F}_q, \mathscr{F}_r))_{(x_i, t_i, u_i)}.$$
    By Lemma \ref{reebchord-hom}, we know that $SS^\infty(\mathscr{H}om(\mathscr{F}_q, \mathscr{F}_r)) \subset (-\Lambda_q) + \Lambda_r.$ As in the proof of Corollary \ref{acyclic}, we know that $\mathrm{supp}(\mathscr{H}om(\mathscr{F}_q, \mathscr{F}_r))$ is compact. Consider $\varphi_-(x, t, u) = -u$. Then
    $$SS(\mathscr{H}om(\mathscr{F}_q, \mathscr{F}_r)) \cap \mathrm{Graph}(d\varphi_-) \cap u^{-1}((0,+\infty)) \subset \{(x_i, 0, t_i, 0, u_i, \nu_i)\}_{i\in I}.$$
    Now the result follows from the microlocal Morse inequality Proposition \ref{morseinequality}.
\end{proof}

    Now the main theorems \ref{puresheaf} and \ref{mixedsheaf} follow immediately from previous results.

\begin{proof}[Proof of Theorem \ref{puresheaf} and \ref{mixedsheaf}]
    Theorem \ref{puresheaf} immediately follows from Theorems \ref{duality2}, \ref{exactseq2} and \ref{reebinequality2}. For Theorem \ref{mixedsheaf}, by Theorems \ref{duality2} and \ref{mixedinequality} we know that
    \[\begin{split}
    \sum_{j\in \mathbb{Z}}\sum_{i\in \mathbb{Z}}&\dim H^iHom(F, F)|\mathcal{Q}_{j-i}(\Lambda)| \geq \sum_{j\in \mathbb{Z}}\dim H^jHom_+(\mathscr{F, F}) \\
    &\geq \frac{1}{2} \sum_{i\in \mathbb{Z}}\dim H^iHom(F, F) \sum_{j=0}^n\dim H^j(\Lambda).
    \end{split}\]
    Now the theorem follows.
\end{proof}

\subsection{Application to the Persistence Module}\label{thmpersist}
    We now apply the results to relate persistence structure to Reeb chords. We first reprove Theorem \ref{puresheaf}, \ref{mixedsheaf} using persistence of $\mathscr{H}om_{(-\infty,+\infty)}(\mathscr{F}, \mathscr{F})$, and then prove Theorem \ref{displaceable} using the continuity of persistence of $\mathscr{H}om_{(-\infty,+\infty)}(\mathscr{F}, \Phi_H^1(\mathscr{F}))$ under Hamiltonian isotopies.

\begin{proof}[Proof of Theorem \ref{puresheaf} and \ref{mixedsheaf}]
    Consider the sheaf $\mathscr{H}om_{(-\infty,+\infty)}(\mathscr{F}, \mathscr{F})$. We know
    $$\mathscr{H}om_{(-\infty,+\infty)}(\mathscr{F}, \mathscr{F}) = u_*\mathscr{H}om(\mathscr{F}_q, \mathscr{G}_r) \simeq \bigoplus_{\alpha \in I}\Bbbk^{r_\alpha}_{(c_\alpha, c_\alpha']}[n_\alpha].$$
    Since $u: M \times \mathbb{R}^2 \rightarrow \mathbb{R}$ is proper on $\mathrm{supp}(\mathscr{H}om(\mathscr{F}_q, \mathscr{G}_r))$, we know that
    $$\Gamma_{u\leq c}(u_*\mathscr{H}om(\mathscr{F}_q, \mathscr{G}_r))_{c} \simeq u_*\Gamma_{u\leq c}(\mathscr{H}om(\mathscr{F}_q, \mathscr{G}_r))_{u^{-1}(c)}.$$
    On the other hand, given a bar $\Bbbk_{(c, c']}$, we know that
    $$\Gamma_{u\leq c}(\Bbbk_{(c, c']})_c \simeq \Bbbk[-1], \,\,\, \Gamma_{u\leq c'}(\Bbbk_{(c, c']})_{c'} \simeq \Bbbk.$$
    Hence by Proposition \ref{localcompute} we will determine the number of starting point/ending point of bars from the rank of the microstalk.

    By Corollary \ref{morseflow}, we know that in degree $j + 1$, there are at least $\dim H^j(\Lambda; \Bbbk^{r^2})$ starting points or ending points of bars at $u = 0$. The starting points of such bars should come from bars of the form $\Bbbk_{(0,c_+]}[-j]$ while the ending points of bars should come from bars of the form $\Bbbk_{(c_-,0]}[-j-1]$. By Lemma \ref{reebchord-hom}, the other ending point/starting point of these bars will correspond to signed lengths of Reeb chords in $\mathcal{Q}_\pm(\Lambda)$. By Proposition \ref{localcompute}, we know that for $c_+ > 0$ that corresponds to a degree $d_+$ Reeb chord, the microstalk
    $$\Gamma_{u\leq c_+}(u_*\mathscr{H}om(\mathscr{F}_q, \mathscr{G}_r))_{c_+} \simeq \Bbbk^{r^2}[-d_+].$$
    Hence the corresponding ending point of a bar $\Bbbk_{(0,c_+]}[-j]$ should be a degree $j$ Reeb chord. Similarly for $c_- < 0$ that corresponds to a degree $d_-$ Reeb chord, by Proposition \ref{localcomputemix} the microstalk
    $$\Gamma_{u\leq c_-}(u_*\mathscr{H}om(\mathscr{F}_q, \mathscr{G}_r))_{c_-} \simeq \Bbbk^{r^2}[- n-2 + d_-].$$
    Hence the corresponding starting point of a bar $\Bbbk_{(c_-,0]}[-j-1]$ should be a degree $n - j$ Reeb chord. Therefore
    $$r^2|\mathcal{Q}_j(\Lambda)| + r^2|\mathcal{Q}_{n-j}(\Lambda)| \geq r^2 \dim H^j(\Lambda; \Bbbk).$$
    This proves Theorem \ref{puresheaf}. The proof of Theorem \ref{mixedsheaf} is similar.
\end{proof}

    Finally we prove Theorem \ref{displaceable}, which gives estimates on the Reeb chords between $\Lambda$ and its Hamiltonian pushoff $\varphi_H^1(\Lambda)$ for a contact Hamiltonian flow $\varphi_H^s\,(s\in I)$.

\begin{proof}[Proof of Theorem \ref{displaceable}]
    Consider the sheaf $\mathscr{H}om_{(-\infty,+\infty)}(\mathscr{F}, \mathscr{F})$. We know from the previous proof that starting points and ending points of bars at $u = 0$ in degree $j + 1$ correspond to a basis of $H^j(\Lambda; \Bbbk^{r^2})$. In addition, the corresponding ending point of a bar $\Bbbk_{(0,c_+]}[-j]$ should be a degree $j$ Reeb chord, and the corresponding starting point of a bar $\Bbbk_{(c_-,0]}[-j-1]$ should be a degree $n - j$ Reeb chord. The lengths of these bars at time $s = 0$ will be at least
    $$c_j(\Lambda) = c_{n-j}(\Lambda) = \min\{l(\gamma) \mid \gamma \in \mathcal{Q}_j(\Lambda) \cup \mathcal{Q}_{n-j}(\Lambda)\}.$$
    Consider the Hamiltonian $\varphi_H^s\,(s\in I)$. Since
    $$\|H\|_\text{osc} < c_{j_0}(\Lambda), \dots, c_{j_k}(\Lambda),$$
    we know by Theorem \ref{persistcontinue2} that these bars will survive in $\mathscr{H}om_{(-\infty,+\infty)}(\mathscr{F}, \Phi_H^1(\mathscr{F}))$.

    We claim that each bar in $\mathscr{H}om_{(-\infty,+\infty)}(\mathscr{F}, \Phi_H^1(\mathscr{F}))$ corresponds to a Reeb chord between $\Lambda$ and $\varphi_H^1(\Lambda)$. Namely the proof is similar to Lemma \ref{reebchord-hom}. Note that $\Lambda_q \cap (\varphi_H^1(\Lambda))_r = \emptyset$, so $(u, \nu) \in SS^\infty(\mathscr{H}om_{(-\infty,+\infty)}(\mathscr{F}, \Phi_H^1(\mathscr{F})))$ iff
    $$(x, 0, t, 0, u, \nu) \in (-\Lambda_q) + (\varphi_H^1(\Lambda))_r,$$
    iff there exists $(x, \xi, t, \tau) \in \Lambda, (x, \xi, t+u, \tau) \in \varphi_H^1(\Lambda)$ (and $\nu = -\tau$). In addition, the computation of microstalks in Proposition \ref{localcompute} still holds. Hence the endpoints of bars count Reeb chords both from $\Lambda$ to $\varphi_H^1(\Lambda)$ and from $\varphi_H^1(\Lambda)$ back to $\Lambda$, i.e.~the chords between $\Lambda$ and $\varphi_H^1(\Lambda)$. Thus
    $$r^2|\mathcal{Q}(\Lambda, \varphi_H^1(\Lambda))| \geq r^2 \sum_{0\leq i\leq k} \dim H^{j_i}(\Lambda; \Bbbk).$$
    This completes the proof of the theorem.
\end{proof}

\subsection{Horizontal displaceability}\label{horizontal}
    As is mentioned in Remark \ref{hori}, we show that for all horizontally displaceable closed Legendrians $\Lambda \subset T^{*,\infty}_{\tau > 0}(M \times \mathbb{R})$, $\mathscr{F} \in Sh^b_\Lambda(M \times \mathbb{R})$ with zero stalk near $M \times \{-\infty\}$ necessarily has compact support. Note that under the assumption that $M$ is noncompact, such $\mathscr{F} \in Sh^b_\Lambda(M \times \mathbb{R})$ will always have compact support as the front projection $\pi(\Lambda)$ is compact in $M \times \mathbb{R}$, so we only need to consider the case where $M$ is compact.

    Recall that $\Lambda \subset T^{*,\infty}_{\tau > 0}(M \times \mathbb{R})$ is horizontally displaceable if there is a Hamiltonian flow $\varphi_H^s\,(s \in I)$ such that there are no Reeb chords between $\Lambda$ and $\varphi_H^1(\Lambda)$.

\begin{lemma}\label{hori-hom}
    Let $\Lambda, \Lambda' \subset T^{*,\infty}_{\tau > 0}(M \times \mathbb{R})$ be closed Legendrians, and $\mathscr{F} \in Sh^b_{\Lambda}(M \times \mathbb{R}), \mathscr{F}' \in Sh^b_{\Lambda'}(M \times \mathbb{R})$ such that the stalks near $M \times \{-\infty\}$ are zero. Suppose there are no Reeb chords between $\Lambda$ and $\Lambda'$. Then for any $c \in \mathbb{R}$,
    $$Hom(\mathscr{F}, T_{c*}\mathscr{F}') \simeq 0.$$
\end{lemma}
\begin{proof}
    We know that
    $$\Gamma_{u\leq c}(u_*\mathscr{H}om(\mathscr{F}_q, \mathscr{F}'_r))_{c} \simeq u_*\Gamma_{u\leq c}(\mathscr{H}om(\mathscr{F}_q, \mathscr{F}'_r))_{u^{-1}(c)}.$$
    Therefore since there are no Reeb chords between $\Lambda$ and $\Lambda'$, by Lemma \ref{reebchord-hom}, we know that $\mathscr{H}om_{(-\infty,+\infty)}(\mathscr{F, F}') = u_*\mathscr{H}om(\mathscr{F}_q, \mathscr{F}'_r)$ is a constant sheaf on $\mathbb{R}$.

    First, consider $C \in \mathbb{R}$ such that the front projection $\pi_{M \times \mathbb{R}}(\Lambda')$ is contained in $M \times (-C, C)$. Then, consider $u = -c$ is sufficiently small so that the front projection $\pi_{M \times \mathbb{R}}(T_{-c}(\Lambda'))$ is contained in $M \times (-\infty, -C)$. Let $i_{u=-c}$ be the inclusion $M \times \mathbb{R} \times \{-c\} \hookrightarrow M \times \mathbb{R}^2$. Then as Proposition \ref{ssforhom} implies that
    $$i_{u=-c}^{-1}\mathscr{H}om(\mathscr{F}_q, \mathscr{F}'_r) = \mathscr{H}om(\mathscr{F}, T_{-c*}\mathscr{F}') \simeq D'\mathscr{F} \otimes T_{-c*}\mathscr{F}',$$
    and the stalk of $\mathscr{F}$ is zero near $\pi_{M \times \mathbb{R}}(\Lambda')$, it is implied that
    $$SS^\infty(i_{u=-c}^{-1}\mathscr{H}om(\mathscr{F}_q, \mathscr{F}'_r)) \subset (-\Lambda) \subset T^{*,\infty}_{\tau < 0}(M \times \mathbb{R}).$$
    By microlocal Morse lemma we can conclude that
    $$\Gamma(M \times \mathbb{R}, i_{u=-c}^{-1}\mathscr{H}om(\mathscr{F}_q, \mathscr{F}'_r)) \simeq \Gamma(M \times (-\infty, -C), i_{u=-c}^{-1}\mathscr{H}om(\mathscr{F}_q, \mathscr{F}'_r)) \simeq 0.$$
    Since $\mathscr{H}om_{(-\infty,+\infty)}(\mathscr{F, F}')$ is constant this shows the assertion.
\end{proof}

\begin{proposition}\label{horizontal-displace}
    Let $M$ be compact. If $\Lambda \subset T^{*,\infty}_{\tau > 0}(M \times \mathbb{R})$ is horizontally displaceable, then any $\mathscr{F} \in Sh^b_\Lambda(M \times \mathbb{R})$ with perfect stalks that has zero stalk near $M \times \{-\infty\}$ will have compact support.
\end{proposition}
\begin{proof}
    Suppose $\mathrm{supp}(\mathscr{F})$ is noncompact. Then the fact that $M$ is compact and that $\mathscr{F}$ has zero stalk near $M \times \{-\infty\}$ necessarily mean that for any $C > 0$ sufficiently large, there exists $x \in M$ and $t > C$ such that $\mathscr{F}_{(x,t)} \neq 0$. Let
    $$C > \sup\{t \in \mathbb{R} \mid \exists\,(x, \xi) \in T^*M, \, (x, \xi, t, 1) \in \Lambda\}.$$
    Then $\mathscr{F}$ is locally constant on $M \times [C, +\infty)$ with nonzero stalk.

    Since $\Lambda$ is horizontally displaceable, there is a Hamiltonian flow $\varphi_H^s\,(s \in I)$ such that no Reeb chords are between $\Lambda$ and $\varphi_H^1(\Lambda)$. Let $\Lambda' = \varphi_H^1(\Lambda)$ and (following Theorem \ref{GKS}) $\mathscr{F}' = \Phi_H^1(\mathscr{F})$. $\mathscr{F}'$ is also locally constant on $M \times [C, +\infty)$ for sufficiently large $C > 0$ with nonzero stalk. By Lemma \ref{hori-hom},
    $$Hom(\mathscr{F}, T_{c*}\mathscr{F}') \simeq 0.$$
    Let $c > 0$ be sufficiently large such that the front projection $\pi_{M \times \mathbb{R}}(T_c(\Lambda'))$ is contained in $M \times (C, +\infty)$. Then using the formula
    $$\mathscr{H}om(\mathscr{F}, T_{c*}\mathscr{F}') = D'\mathscr{F} \otimes T_{c*}\mathscr{F}',$$
    near $\pi_{M \times \mathbb{R}}(\Lambda)$ the stalk of $\mathscr{H}om(\mathscr{F}, T_{c*}\mathscr{F}')$ is zero. Hence
    $$SS^\infty(\mathscr{H}om(\mathscr{F}, T_{c*}\mathscr{F}')) \subset \Lambda' \subset T^{*,\infty}_{\tau > 0}(M \times \mathbb{R}).$$
    By microlocal Morse lemma we can conclude that
    $$Hom(\mathscr{F}, T_{c*}\mathscr{F}') \simeq \Gamma(M \times (C, +\infty), \mathscr{H}om(\mathscr{F}, T_{c*}\mathscr{F}')) \not\simeq 0,$$
    which leads to a contradiction.
\end{proof}
\begin{remark}
    Using the framework of Tamarkin categories, the above proposition can be viewed as a version of the Tamarkin separation theorem \cite[Theorem 3.2]{Tamarkin1}.
\end{remark}

    Using the above criterion, we will be able to prove certain Legendrian submanifolds are not horizontally displaceability. Indeed, for any clsoed manifold $M$ with dimension $\dim M \geq 2$, we will construct closed Legendrians in $J^1(M)$ that are not horizontally displaceable. See Appendix \ref{appendix} Theorem \ref{genfamily-example}.

\subsection{Non-squeezing into Loose Legendrians}\label{loosesection}

    In this section we show Theorem \ref{squeeze} that the $C^0$-limit of a smooth family of Legendrian submanifolds is not going to be stablized or loose when there exists some non-trivial sheaf theoretic invariant. Here is the definition and the theorem.

\begin{definition}[Dimitroglou Rizell-Sullivan; Definition \ref{squeezedef}]
    Let $n = \dim M$ and $U \subset T^{*,\infty}_{\tau > 0}(M \times \mathbb{R})$ be an open subset with $H_n(U; \mathbb{Z}/2\mathbb{Z}) \not\cong 0$. A Legendrian submanifold $\Lambda \subset T^{*,\infty}_{\tau > 0}(M \times \mathbb{R})$ can be squeezed into $U$ if there is a Legendrian isotopy $\Lambda_t$ with $\Lambda_0 = \Lambda$ and
    $$\Lambda_1 \subset U, \,\,\, [\Lambda_1] \neq 0 \in H_n(U; \mathbb{Z}/2\mathbb{Z}).$$
\end{definition}

\begin{theorem}[Theorem \ref{squeeze}]\label{squeeze2}
    Let $\Lambda_\text{loose} \subset T^{*,\infty}_{\tau > 0}(\mathbb{R}^{n+1})$ be a stablized/loose Legendrian, and $\Lambda \subset T^{*,\infty}_{\tau > 0}(\mathbb{R}^{n+1})$ be a Legendrian so that there exists $\mathscr{F} \in Sh^b_\Lambda(\mathbb{R}^{n+1})$ whose microstalk has odd Euler characteristic. Then $\Lambda$ cannot be squeezed into a tubular contact neighbourhood of $\Lambda_\text{loose}$.
\end{theorem}

    The idea is to detect the Legendrian $\Lambda$ by a fiber $T^{*,\infty}_{(x_0,t_0)}\mathbb{R}^{n+1}$ as in Example \ref{fiberlegendrian}. First we state a geometric lemma that is needed. This is proved by Dimitroglou Rizell-Sullivan \cite{RizSpersist}. For the concepts including formal Legendrian isotopy, loose Legendrian submanifolds and $h$-principles, the reader may refer to \cite{Loose}.

\begin{lemma}[Dimitroglou Rizell-Sullivan \cite{RizSpersist}*{Lemma 4.3}]\label{looseleg}
    For $n \geq 2$, let $\Lambda_\text{loose} \subset T^{*,\infty}_{\tau > 0}(\mathbb{R}^{n+1})$ be any loose Legendrian submanifold. Then for any small $A > \epsilon > 0$, $\Lambda_\text{loose}$ is isotopic to $\Lambda_\text{loose}'$ that satisfies the following properties:
\begin{enumerate}
    \item there exists $(x_0,t_0) \in \mathbb{R}^{n+1}$ such that there are precisely 2 (transverse) Reeb chords $\gamma_0, \gamma_1$ from $\Lambda'_\text{loose}$ to $T_{(x_0,t_0)}^{*,\infty}\mathbb{R}^{n+1}$ and
    $$l(\gamma_0) - l(\gamma_1) \geq A;$$
    \item there exists a Hamiltonian $H_s\,(s\in I)$ with $\|H\|_\text{osc} \leq \epsilon$ that horizontally displaces $\Lambda'_\text{loose}$ from the cotangent fiber $T_{(x_0,t_0)}^{*,\infty}\mathbb{R}^{n+1}$.
\end{enumerate}
\end{lemma}

\begin{figure}
  \centering
  \includegraphics[width=0.45\textwidth]{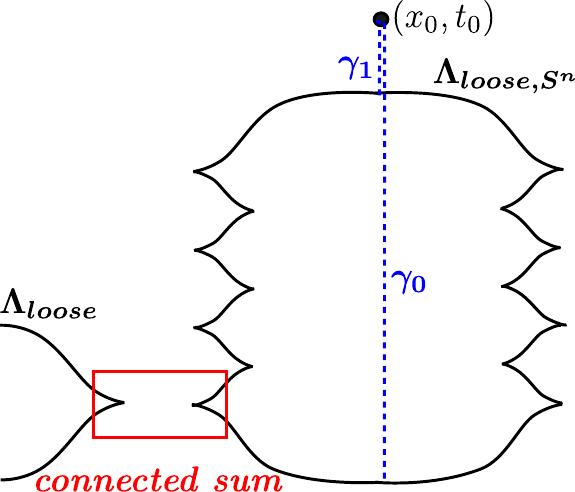}\\
  \caption{On the left there is the loose Legendrian $\Lambda_\text{loose}$ and on the right there is a loose Legendrian $\Lambda_{S^n,\text{loose}}$ formally isotopic to the unknotted sphere (the front projection should be spinning around its symmetry axis). In the red region we perform the connected sum construction.}\label{looselegendrian}
\end{figure}

\begin{proof}[Proof of Theorem \ref{squeeze2}]
    First assume that $n \geq 2$. Suppose $\Lambda$ can be squeezed into a contact tubular  neighbourhood $U_\text{loose}$ of $\Lambda_\text{loose}$. By Lemma \ref{looseleg}, we can apply a contact isotopy so that the contact tubular neighbourhood $U_\text{loose}$ is mapped to a contact tubular neighbourhood $U_\text{loose}'$ of $\Lambda_\text{loose}'$. Denote by $\Lambda'$ the image of the original Legendrian submanifold in $U_\text{loose}'$. By shrinking the contact tubular neighbourhood $U_\text{loose}'$ we may assume that for the projection $\pi_{\mathbb{R}^n} \circ \pi_\text{front}: U_\text{loose}' \rightarrow \mathbb{R}^n$, the height of each connected component of $U_\text{loose}'$ in the fiber of $\pi_{\mathbb{R}^n} \circ \pi_\text{front}$ is less than $\epsilon'$ where $4\epsilon' < A - \epsilon$.

    Lemma \ref{looseleg} ensures that there exists $(x_0,t_0) \in \mathbb{R}^{n+1}$ such that there are 2 (transverse) Reeb chords from $\Lambda_\text{loose}'$ to $T_{(x_0,t_0)}^{*,\infty}\mathbb{R}^{n+1}$, starting from $(x_0, t_1)$ and $(x_0, t_2)$. For $\Lambda' \subset U_\text{loose}'$, since the degree $[\Lambda'] \neq 0 \in H_n(\Lambda_\text{loose}'; \mathbb{Z}/2\mathbb{Z})$, the preimage of $(x_0, t_1)$ and $(x_0, t_2)$ under the projection $U_\text{loose}' \rightarrow \Lambda_\text{loose}'$ contains an odd number of points $p_{1,1}, \dots, p_{1,2k+1}$ and $p_{2,1}, \dots, p_{2,2k+1}$. We can assume that the fiber of the contact tubular neighbourhood $U_\text{loose}' \rightarrow \Lambda_\text{loose}'$ are contained in the fibers of the standard projection $J^1(\mathbb{R}^n) \to \mathbb{R}^n$, so $p_{1,1}, \dots, p_{1,2k+1}$ and $p_{2,1}, \dots, p_{2,2k+1}$ are also the preimage of $x_0$ under the projection $U_\text{loose}' \rightarrow \mathbb{R}^n$. We can also assume that
    $$\min_{1\leq i,j\leq 2k+1}|u(p_{1,i}) - u(p_{2,j})| \geq A - 2\epsilon'.$$

    For $\mathscr{F}' \in Sh_{\Lambda'}(\mathbb{R}^{n+1})$ that is the image of $\mathscr{F} \in Sh_{\Lambda}(\mathbb{R}^{n+1})$ under the contact isotopy, we now calculate
    $$\mathscr{H}om_{(-\infty,+\infty)}(\Bbbk_{(x_0,t_0)}, \mathscr{F}').$$
    By Lemma \ref{reebchord-hom2}, $u(p_{1,1}), \dots, u(p_{1,2k+1})$ and $u(p_{2,1}), \dots, u(p_{2,2k+1})$ correspond to all the starting points and ending points of the bars. In addition, for each point the number of bars $\Bbbk_{(a,b]}$ in the sheaf is at least the rank of the cohomology of the microstalk of $\mathscr{F}'$, which is an odd number since the Euler characteristic of the microstalk is odd. We argue that there must be a bar starting from $u(p_{1,i})$ and ending at $u(p_{2,j})$. Otherwise all bars start at some $u(p_{1,i})$ will end at some $u(p_{1,j})$ for $i \neq j$. However, there are odd number of points $u(p_{1,1}), \dots, u(p_{1,2k+1})$ and there are odd numbers of bars starting from or ending at each point since the total rank of the microstalk is odd. Thus there are odd numbers of endpoints of bars in $u(p_{1,1}), \dots, u(p_{1,2k+1})$, and it is impossible for all bars start at some $u(p_{1,i})$ to end at some $u(p_{1,j})$ for $i \neq j$. Now that we know there is a bar starting from $u(p_{1,i})$ and ending at $u(p_{2,j})$, it will have length at least $A - 2\epsilon'$. 

    Consider the Hamiltonian $H_s\,(s\in I)$ with $\|H\|_\text{osc} \leq \epsilon+\epsilon' < A - 2\epsilon'$ and horizontally displaces $\Lambda_\text{loose}'$ from the cotangent fiber $T^{*,\infty}_{(x_0,t_0)}\mathbb{R}^{n+1}$ as in Lemma \ref{looseleg}. For a sufficiently small neighbourhood $U_\text{loose}'$ of $\Lambda_\text{loose}'$, the Hamiltonian $H_s\,(s\in I)$ will also horizontally displace $U_\text{loose}'$. By Theorem \ref{persistcontinue2}, we have
   $$\bar{d}(\mathscr{H}om_{(-\infty,+\infty)}(\Bbbk_{(x_0,t_0)}, \mathscr{F}'), \mathscr{H}om_{(-\infty,+\infty)}(\Bbbk_{(x_0,t_0)}, \Phi_H(\mathscr{F}'))) \leq \epsilon + \epsilon' < A - 2\epsilon'.$$
    Then, consider the bar in $\mathscr{H}om_{(-\infty,+\infty)}(\Bbbk_{(x_0,t_0)}, \mathscr{F}')$ with length at least $A - 2\epsilon'$. This bar will persist in $\mathscr{H}om_{(-\infty,+\infty)}(\Bbbk_{(x_0,t_0)}, \Phi_H(\mathscr{F}')))$. This contradicts with the assumption that $H_s\,(s\in I)$ horizontally displaces $U_\text{loose}'$.

    Finally when $n = 1$, suppose $\Lambda$ is contained in a contact tubular neighbourhood of $\Lambda_\text{loose}$. We apply the spinning construction \cite[Section 4.4]{EESNoniso} (Figure \ref{spinning}) to a stablized Legendrian knot, as explained in \cite[Section 4]{RizSpersist}. Namely, consider a real line $t = t_0$ that is disjoint from the front projection $\Lambda_\text{loose}$ and $\Lambda$ and spin around the front along the line $x = x_0$. The standard zigzag thus gives a loose chart for the new Legendrian $\Lambda_\text{loose,spin}$ and $\Lambda_\text{spin}$ in $T^{*,\infty}_{\tau > 0}\mathbb{R}^3$. It is clear from the front projection that, if there is a sheaf with singular support on a knot, then there is also a sheaf with singular support on its spinning. In fact, we consider $\mathbb{R}^3 \setminus \{(x, y, t) \mid x = x_0, y = 0\} \cong \{(x, t) \mid x < x_0\} \times S^1$ and the projection
    $$\pi: \mathbb{R}^3 \setminus \{(x, y, t) \mid x = x_0, y = 0\} \cong \{(x, t) \mid x < x_0\} \times S^1 \rightarrow \{(x, t) \mid x < x_0\} \cong \mathbb{R}^2.$$
    Now for $\mathscr{F}$ on $\{(x, t) \mid x < x_0\}$ with singular support $\Lambda$, take the sheaf $\pi^{-1}\mathscr{F}$ on $\mathbb{R}^3 \setminus \{(x, y, t) \mid x = x_0, y = 0\}$, whose singular support will be in $\Lambda_\text{spin}$. Note that $\mathrm{supp}(\mathscr{F})$ is compact, so $\pi^{-1}\mathscr{F}$ has zero stalk near the line $\{(x, y, t) \mid x = x_0, y = 0\}$ and we can easily extend it to a sheaf on $\mathbb{R}^3$. Then applying the argument above will complete the proof.
\end{proof}

\begin{remark}
    We explain why we need the assumption that the Euler characteristic of the microstalk is odd. In fact, this result should be false if we drop the odd Euler characteristic assumption. The issue is that, for each real number that contains $r$ endpoints of bars, these $r$ bars may go to different endpoints. Thus, there exist persistence modules with odd number of endpoints such that the rank changes by 2 at each endpoint, which makes it impossible to exclude the case that all bars starting at $u(p_{1,i})$ end at some nearby points $u(p_{1,j})$. For example, the sheaf $\mathscr{F} = \Bbbk_{(-1,0]} \oplus \Bbbk_{(0,1]} \oplus \Bbbk_{(-1,1]}$ satisfies this condition, and one can similarly find sheaves with this property on a three-copy Reeb push-off of a loose Legendrian submanifold by taking direct sums of three copies of the doubling construction, following \cite[Part 11]{Guisurvey}. This gives an example of a Legendrian which admits sheaves whose microlocal stalks have even Euler characteristics that can be squeezed into a tubular neighbourhood of a loose Legendrian.

    When the Legendrian submanifold is connected, it is possible that the odd dimensional stalk condition is no longer necessary. However, the proof presented here will not work because of the reason we just explained above (that we cannot exclude the case where different bars from the same starting point end at different ednpoints). It is an interesting question to try to generalize the result in that setting.
\end{remark}

\begin{figure}
  \centering
  \includegraphics[width=0.45\textwidth]{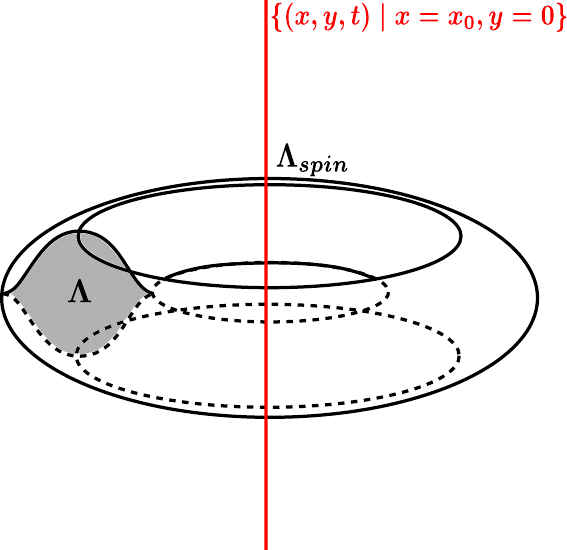}\\
  \caption{The front spinning of a standard unknot along the red vertical line $x = x_0$. We consider the Legendrian front projection in the plane $\{(x, t) \mid x < x_0\} \subset \mathbb{R}^2$ and obtain a Legendrian front projection in $\mathbb{R}^3 \setminus \{(x, y, t) \mid x = x_0, y = 0\} \subset \mathbb{R}^3$.}\label{spinning}
\end{figure}

\appendix
\section{Sheaves and generating families}\label{appendix}

    In this appendix, we explain the relationship between microlocal rank 1 sheaves with compact support and generating families linear at infinity. For the definition of generating families linear at infinity, we follow \cite{TrayGen,Genfamily}. The main results are as follows:

\begin{proposition}\label{genfamily-sheaf}
    A generating family $f: M \times \mathbb{R}^N \to \mathbb{R}$ for a closed Legendrian $\Lambda \subset J^1(M)$ with generic front projection induces a microlocal rank 1 sheaf on $M \times \mathbb{R}$ with singular support in $\Lambda$ with $0$ stalk at $M \times \{-\infty\}$. Moreover, if $f: M \times \mathbb{R}^N \to \mathbb{R}$ is linear at infinity, then the sheaf has compact support.
\end{proposition}

\begin{proposition}\label{genfamily-example}
    For any closed manifold $M$ with $\dim M \geq 2$, there exists a Legendrian $\Lambda \subset J^1(M)$ that is not horizontally displaceable, does not admit a generating family linear at infinity, but admits a microlocal rank 1 sheaf on $M \times \mathbb{R}$ with compact support for some coefficient ring $\Bbbk$.
\end{proposition}

    Here, we say that the front projection of $\Lambda$ is generic if for a generic point $p \in \Lambda$, there exists a neighboruhood $\Omega \subset \Lambda$ such that the front projection $\pi|_\Omega: \Omega \to M \times \mathbb{R}$ is a smooth embedding. The first proposition is standard, which goes back to Viterbo \cite[Section 9.1.2]{ViterboIntroSheaf}, who proved a similar proposition for generating families quadratic at infinity.

\begin{proof}[Proof of Proposition \ref{genfamily-sheaf}]
    We write $j_{f \geq c} : \{(x, u, t) \mid f(x, u) \geq c\} \hookrightarrow M \times \mathbb{R}^N \times \mathbb{R}$ and $\pi_{M \times \mathbb{R}}: M \times \mathbb{R}^N \times \mathbb{R} \to M \times \mathbb{R}$. Then we define the sheaf to be
    $$\mathscr{F} = \pi_{M \times \mathbb{R}!} j_{f \geq c !} \Bbbk_{f \geq c} \in Sh^b(M \times \mathbb{R}).$$
    Since $SS^\infty(j_{f \geq c !} \Bbbk_{f \geq c}) = \Lambda_{j^1f} = \{(x, u, f(x, u); \xi, \nu, 1) \mid (\xi, \nu) = df(x, u) \}$, it follows from Proposition \ref{sspushforward} that 
    $$SS^\infty( \pi_{M \times \mathbb{R}!} j_{f \geq c !} \Bbbk_{f \geq c}) \subset \Lambda.$$
    Moreover, consider a generic point in the front projection where $f(x_0, u_0) = t_0$. Since $\Lambda_{j^1f} \pitchfork J^1(M) \times\mathbb{R}^N \times  0 \times \mathbb{R}$, by implicit function theorem, we may find $u = u(x)$ such that $Df(x, u(x)) = 0$ on a small ball $U_{(x_0, u_0, t_0)}$ around $(x_0, u_0, t_0)$. Let $\varphi(x, t) = t - f(x, u(x))$. By Morse lemma,
    $$\Gamma_{\varphi \geq 0}(\pi_{M \times \mathbb{R} !} j_{f \geq c !} \Bbbk_{f \geq c} )_{(x_0, t_0)} = \Bbbk[-\mathrm{ind}(D^2f_{(x_0, u_0)})].$$

    Finally, it suffices to show that when $f$ is linear at infinity, then $\mathscr{F}$ has compact support. Since $f$ is linear at infinity, we know that for any $x_0 \in M$, there exists $t_0$ sufficiently large, such that $f$ is linear on $f^{-1}([t_0, +\infty))$. Thus when $t_0' > t_0$,
    $$(\pi_{M \times \mathbb{R}!} j_{f \geq c !} \Bbbk_{f \geq c})_{(x_0, t_0')} = \Gamma_c(\mathbb{R}^N, \Bbbk_{\{f(x_0, u) \geq t_0'\}}) = 0.$$
    Therefore, since $SS^\infty(\mathscr{F}) \subset \Lambda$ and $\pi(\Lambda) \subset M \times \mathbb{R}$ is compact, we know that $\mathrm{supp}(\mathscr{F})$ is compact.
\end{proof}

\begin{proof}[Proof of Proposition \ref{genfamily-example}]
    We first construct the Legendrian in $J^1(S^n) \,(n \geq 2)$. Let $S^{n-2} \subset S^n$ be the standard dividing sphere of $S^{n-1} \subset S^n$. Then $\pi_1(S^n \backslash S^{n-2}) = \mathbb{Z}$. Consider a tubular neighbourhood $U(S^{n-2}) = S^{n-2} \times D^2 \subset S^n$ with coordinate $(u, v) \in S^{n-2} \times D^2$. Let $\rho: [0, 1] \to [0, 1]$ be a smooth function such that 
    $$\rho|_{[0, \epsilon)}(x) \equiv x, \;\; \rho|_{(1-\epsilon, 1]}(x) \equiv 1.$$
    We define a Legendrian in $J^1(U(S^{n-2}))$ as
    $$\Lambda_0 = \{(u, v, 0, \eta, t) \mid t^2 = \rho(v^2), t \eta = v \rho'(v^2) \} \subset J^1(U(S^{n-2})),$$
    whose front projection is an $S^{n-2}$-family of cone singularities. We then glue it with 
    $$\Lambda_1 = \{(x, 0, \pm 1) \mid x \in S^n \backslash U(S^{n-2})\} \subset J^1(S^n \backslash U(S^{n-2}))$$
    and get a Legendrian $\Lambda \subset J^1(S^n)$. See Figure \ref{Nondisplace-Example}.

\begin{figure}
  \centering
  \includegraphics[width=0.35\textwidth]{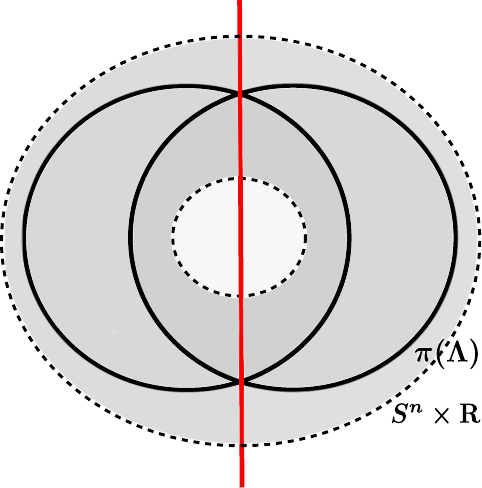}\\
  \caption{The front projection of the Legendrian $\pi(\Lambda) \subset S^n \times \mathbb{R}$ along the slice $S^{n-1} \times \mathbb{R}$. The front projection  $\pi(\Lambda)$ is obtained by spinning the front projection in the picture along the red vertical line.}\label{Nondisplace-Example}
\end{figure}

    Consider microlocal rank 1 sheaves $\mathscr{F} \in Sh^b_\Lambda(S^n \times \mathbb{R}; \Bbbk)$. We will prove that 
    \begin{enumerate}
    \item for any $\Bbbk$, there exists $\mathscr{F} \in Sh^b_\Lambda(S^n \times \mathbb{R}; \Bbbk)$ that is not compactly supported;
    \item for $\Bbbk \neq \mathbb{Z}/2\mathbb{Z}$, there exists $\mathscr{F} \in Sh^b_\Lambda(S^n \times \mathbb{R}; \Bbbk)$ that is compactly supported;
    \item for $\Bbbk = \mathbb{Z}/2\mathbb{Z}$, all $\mathscr{F} \in Sh^b_\Lambda(S^n \times \mathbb{R}; \Bbbk)$ are non-compactly supported. 
    \end{enumerate}
    Statement (1) implies that $\Lambda$ is not horizontally displaceable by Proposition \ref{horizontal-displace}, and statement (3) implies that $\Lambda$ does not admit a generating family linear at infinity by Proposition \ref{genfamily-sheaf}. This will completes the proof of the proposition.

    Decompose $S^2 \times \mathbb{R} \backslash \pi(\Lambda)$ as the union of the stratum $S_-$ at $-\infty$, the stratum $S_+$ at $\infty$ and the stratum $S_0$ in the middle which is diffeomorphic to $S^n \backslash S^{n-2} \times (0, 1)$. Assume that the stalk of $\mathscr{F}$ on $S_-$ is $0$ and the stalk on $S_+$ is $F_+$. Then $\mathscr{F}|_{S_0}$ is a rank 1 local system. 

    First, we prove statement (1). For any ring $\Bbbk$, suppose the local system has monodromy $1$. Then at a cone singularity $(x, 0) \in S^n \times \mathbb{R}$, we know that
    $$\Gamma_{t\geq 0}(\mathscr{F})_{(x, 0)} = \mathrm{Cone}(F_+ \to \Gamma(S_0, \mathscr{F})) \simeq 0.$$
    Since $\Gamma(S_0, \mathscr{F}) = H^*(S^1) \neq 0$, we know that $F_+ \neq 0$. Therefore, by Proposition \ref{horizontal-displace}, $\Lambda$ is not horizontally displaceable.

    Next, we prove statement (2). For $\Bbbk \neq \mathbb{Z}/2\mathbb{Z}$, suppose the local system has monodromy $c \neq 1$. Then at a cone singularity $(x, 0) \in S^n \times \mathbb{R}$, we know that
    $$\Gamma_{t\geq 0}(\mathscr{F})_{(x, 0)} = \mathrm{Cone}(F_+ \to \Gamma(S_0, \mathscr{F})) \simeq 0.$$
    Since $\Gamma(S_0, \mathscr{F}) = 0$, we know that $F_+ = 0$. Hence there exists a sheaf $\mathscr{F} \in Sh^b_\Lambda(S^n \times \mathbb{R})$ with compact support when $\Bbbk \neq \mathbb{Z}/2\mathbb{Z}$.

    Finally, we prove statement (3). For $\Bbbk = \mathbb{Z}/2\mathbb{Z}$, since any rank 1 local system on $S_0$ over $\Bbbk$ has monodromy 1, we can conclude that $F_+ \neq 0$. Hence there are no microlocal rank 1 sheaves $\mathscr{F} \in Sh^b_\Lambda(S^n \times \mathbb{R})$ with compact support when $\Bbbk = \mathbb{Z}/2\mathbb{Z}$.

    Finally, for a general closed manifold $M$. We can take connected sum of $J^1(M)$ and $J^1(S^n)$, and take the connected sum of $\{(x, 0, \pm 1) \mid x \in M)\} \subset J^1(M)$ with $\Lambda \subset J^1(S^n)$. The sheaf computation is local, and hence the proof goes through.
\end{proof}

\bibliography{legendriansheaf}

\begin{bibdiv}
\begin{biblist}

\bib{Akaho}{article}{
      author={Akaho, Manabu},
       title={Symplectic displacement energy for exact {L}agrangian
  immersions},
        date={2015},
     journal={arXiv preprint arXiv:1505.06560},
}

\bib{Arnold}{article}{
      author={Arnol'd, V~I},
       title={First steps in symplectic topology},
        date={1986},
     journal={Russian Mathematical Surveys},
      volume={41},
      number={6},
       pages={1\ndash 21},
}

\bib{AsanoIkeimmersion}{inproceedings}{
      author={Asano, Tomohiro},
      author={Ike, Yuichi},
       title={Sheaf quantization and intersection of rational {L}agrangian
  immersions},
   booktitle={Annales de l'{I}nstitut {F}ourier},
       pages={1\ndash 55},
}

\bib{AsanoIke}{article}{
      author={Asano, Tomohiro},
      author={Ike, Yuichi},
       title={Persistence-like distance on {T}amarkin's category and symplectic
  displacement energy},
        date={2020},
     journal={Journal of Symplectic Geometry},
      volume={18},
      number={3},
       pages={613\ndash 649},
}

\bib{AsanoIkecomplete}{article}{
      author={Asano, Tomohiro},
      author={Ike, Yuichi},
       title={Completeness of derived interleaving distances and sheaf
  quantization of non-smooth objects},
        date={2024},
        ISSN={0025-5831},
     journal={Math. Ann.},
      volume={390},
      number={2},
       pages={2991\ndash 3037},
         url={https://doi.org/10.1007/s00208-024-02815-x},
}

\bib{relativeCY}{article}{
      author={Brav, Christopher},
      author={Dyckerhoff, Tobias},
       title={Relative {C}alabi-{Y}au structures},
        date={2019},
        ISSN={0010-437X},
     journal={Compos. Math.},
      volume={155},
      number={2},
       pages={372\ndash 412},
         url={https://doi.org/10.1112/s0010437x19007024},
}

\bib{CasalsMurphydga}{article}{
      author={Casals, Roger},
      author={Murphy, Emmy},
       title={Differential algebra of cubic planar graphs},
        date={2018},
     journal={Advances in Mathematics},
      volume={338},
       pages={401 \ndash  446},
}

\bib{CasalsZas}{article}{
      author={Casals, Roger},
      author={Zaslow, Eric},
       title={Legendrian weaves: {N}--graph calculus, flag moduli and
  applications},
        date={2023},
     journal={Geometry \& Topology},
      volume={26},
      number={8},
       pages={3589\ndash 3745},
}

\bib{RepSheaf}{article}{
      author={Chantraine, Baptiste},
      author={Ng, Lenhard},
      author={Sivek, Steven},
       title={Representations, sheaves and {L}egendrian {$(2, m)$} torus
  links},
        date={2019},
     journal={Journal of the London Mathematical Society},
      volume={100},
      number={1},
       pages={41\ndash 82},
}

\bib{Proximitypersist}{inproceedings}{
      author={Chazal, Fr{\'e}d{\'e}ric},
      author={Cohen-Steiner, David},
      author={Glisse, Marc},
      author={Guibas, Leonidas~J},
      author={Oudot, Steve~Y},
       title={Proximity of persistence modules and their diagrams},
        date={2009},
   booktitle={Proceedings of the twenty-fifth annual symposium on computational
  geometry},
       pages={237\ndash 246},
}

\bib{Stablepersist}{article}{
      author={Chazal, Fr\'{e}d\'{e}ric},
      author={de~Silva, Vin},
      author={Glisse, Marc},
      author={Oudot, Steve},
       title={The structure and stability of persistence modules},
        date={2016},
       pages={x+120},
         url={https://doi.org/10.1007/978-3-319-42545-0},
}

\bib{Chekanov}{article}{
      author={Chekanov, Yu.~V.},
       title={Lagrangian intersections, symplectic energy, and areas of
  holomorphic curves},
        date={1998},
     journal={Duke Mathematical Journal},
      volume={95},
      number={1},
       pages={213\ndash 226},
}

\bib{Chiu}{article}{
      author={Chiu, Sheng-Fu},
       title={Non-squeezing property of contact balls},
        date={2017},
     journal={Duke Mathematical Journal},
      volume={166},
       pages={605\ndash 655},
}

\bib{Cieliechord}{article}{
      author={Cieliebak, Kai},
       title={Handle attaching in symplectic homology and the chord
  conjecture},
        date={2002},
     journal={Journal of the European Mathematical Society},
      volume={004},
      number={2},
       pages={115\ndash 142},
}

\bib{CE}{book}{
      author={Cieliebak, Kai},
      author={Eliashberg, Yakov},
       title={From {S}tein to {W}einstein and back: symplectic geometry of
  affine complex manifolds},
   publisher={American Mathematical Soc.},
        date={2012},
      volume={59},
}

\bib{RizGolestimating}{article}{
      author={Dimitroglou~Rizell, Georgios},
      author={Golovko, Roman},
       title={Estimating the number of {R}eeb chords using a linear
  representation of the characteristic algebra},
        date={2015},
     journal={Algebraic \& Geometric Topology},
      volume={15},
      number={5},
       pages={2885\ndash 2918},
}

\bib{RizSenergy}{article}{
      author={Dimitroglou~Rizell, Georgios},
      author={Sullivan, Michael~G.},
       title={An energy-capacity inequality for {L}egendrian submanifolds},
        date={2020},
        ISSN={1793-5253},
     journal={J. Topol. Anal.},
      volume={12},
      number={3},
       pages={547\ndash 623},
         url={https://doi.org/10.1142/S1793525319500572},
}

\bib{RizSpersist}{article}{
      author={Dimitroglou~Rizell, Georgios},
      author={Sullivan, Michael~G},
       title={The persistence of the {C}hekanov--{E}liashberg algebra},
        date={2020},
     journal={Selecta Mathematica},
      volume={26},
      number={5},
       pages={1\ndash 32},
}

\bib{Fewdoublepoints}{article}{
      author={Ekholm, Tobias},
      author={Eliashberg, Yakov},
      author={Murphy, Emmy},
      author={Smith, Ivan},
       title={Constructing exact {L}agrangian immersions with few double
  points},
        date={2013},
     journal={Geometric and Functional Analysis},
      volume={23},
       pages={1772\ndash 1803},
}

\bib{EESR2n+1}{article}{
      author={Ekholm, Tobias},
      author={Etnyre, John},
      author={Sullivan, Michael},
       title={The contact homology of {L}egendrian submanifolds in
  {$\mathbb{R}^{2n+1}$}},
        date={2005},
     journal={Journal of Differential Geometry},
      volume={71},
      number={2},
       pages={177\ndash 305},
}

\bib{EESNoniso}{article}{
      author={Ekholm, Tobias},
      author={Etnyre, John},
      author={Sullivan, Michael},
       title={Non-isotopic {L}egendrian submanifolds in {$\mathbb{R}^{2n+1}$}},
        date={2005},
     journal={Journal of Differential Geometry},
      volume={71},
      number={1},
       pages={85\ndash 128},
}

\bib{EESorientation}{article}{
      author={Ekholm, Tobias},
      author={Etnyre, John},
      author={Sullivan, Michael},
       title={Orientations in {L}egendrian contact homology and exact
  {L}agrangian immersions},
        date={2005},
     journal={International Journal of Mathematics},
      volume={16},
      number={05},
       pages={453\ndash 532},
}

\bib{EESduality}{article}{
      author={Ekholm, Tobias},
      author={Etnyre, John~B},
      author={Sabloff, Joshua~M},
       title={A duality exact sequence for {L}egendrian contact homology},
        date={2009},
     journal={Duke mathematical journal},
      volume={150},
      number={1},
       pages={1\ndash 75},
}

\bib{AugSheafknot}{article}{
      author={Gao, Honghao},
       title={Simple sheaves for knot conormals},
        date={2020},
     journal={Journal of Symplectic Geometry},
      volume={18},
      number={4},
       pages={1027\ndash 1070},
}

\bib{Geiges}{book}{
      author={Geiges, Hansj{\"o}rg},
       title={An introduction to contact topology},
   publisher={Cambridge University Press},
        date={2008},
      volume={109},
}

\bib{Gui}{article}{
      author={Guillermou, St{\'e}phane},
       title={Quantization of conic {L}agrangian submanifolds of cotangent
  bundles},
        date={2012},
     journal={arXiv preprint arXiv:1212.5818},
}

\bib{GuiGEthm}{article}{
      author={Guillermou, St{\'e}phane},
       title={The {G}romov-{E}liashberg theorem by microlocal sheaf theory},
        date={2013},
     journal={arXiv preprint arXiv:1311.0187},
}

\bib{Guithree}{article}{
      author={Guillermou, St{\'e}phane},
       title={The three cusps conjecture},
        date={2016},
     journal={arXiv preprint arXiv:1603.07876},
}

\bib{Guisurvey}{article}{
      author={Guillermou, St\'{e}phane},
       title={Sheaves and symplectic geometry of cotangent bundles},
        date={2023},
        ISSN={0303-1179},
     journal={Ast\'{e}risque},
      number={440},
       pages={x+274},
         url={https://doi.org/10.24033/ast.1199},
}

\bib{GKS}{article}{
      author={Guillermou, St{\'e}phane},
      author={Kashiwara, Masaki},
      author={Schapira, Pierre},
       title={Sheaf quantization of {H}amiltonian isotopies and applications to
  nondisplaceability problems},
        date={2012},
     journal={Duke Mathematical Journal},
      volume={161},
      number={2},
       pages={201\ndash 245},
}

\bib{GS}{incollection}{
      author={Guillermou, St{\'e}phane},
      author={Schapira, Pierre},
       title={Microlocal theory of sheaves and {T}amarkin's non displaceability
  theorem},
        date={2014},
   booktitle={Homological mirror symmetry and tropical geometry},
   publisher={Springer},
       pages={43\ndash 85},
}

\bib{Ike}{article}{
      author={Ike, Yuichi},
       title={Compact exact {L}agrangian intersections in cotangent bundles via
  sheaf quantization},
        date={2019},
     journal={Publications of the Research Institute for Mathematical
  Sciences},
      volume={55},
      number={4},
       pages={737\ndash 778},
}

\bib{JinTreu}{article}{
      author={Jin, Xin},
      author={Treumann, David},
       title={Brane structures in microlocal sheaf theory},
        date={2024},
        ISSN={1753-8416},
     journal={J. Topol.},
      volume={17},
      number={1},
       pages={Paper No. e12325, 68},
         url={https://doi.org/10.1112/topo.12325},
}

\bib{KS}{book}{
      author={Kashiwara, Masaki},
      author={Schapira, Pierre},
       title={Sheaves on manifolds},
   publisher={Springer Science \& Business Media},
        date={2013},
      volume={292},
}

\bib{KSpersist}{article}{
      author={Kashiwara, Masaki},
      author={Schapira, Pierre},
       title={Persistent homology and microlocal sheaf theory},
        date={2018},
     journal={Journal of Applied and Computational Topology},
      volume={2},
       pages={83\ndash 113},
}

\bib{Orbitcat}{article}{
      author={Keller, Bernhard},
       title={On triangulated orbit categories},
        date={2005},
        ISSN={1431-0635},
     journal={Doc. Math.},
      volume={10},
       pages={551\ndash 581},
}

\bib{Kuo}{article}{
      author={Kuo, Christopher},
       title={Wrapped sheaves},
        date={2023},
     journal={Advances in Mathematics},
      volume={415},
       pages={108882},
}

\bib{KuoLiSpherical}{article}{
      author={Kuo, Christopher},
      author={Li, Wenyuan},
       title={Spherical adjunction and {S}erre functor from microlocalization},
        date={2022},
     journal={arXiv preprint arXiv:2210.06643},
}

\bib{KuoLiCY}{article}{
      author={Kuo, Christopher},
      author={Li, Wenyuan},
       title={Relative {C}alabi-{Y}au structure on microlocalization},
        date={2024},
     journal={arXiv preprint arXiv:2408.04085},
}

\bib{LiuCuplength}{article}{
      author={Liu, Chun-Gen},
       title={Cup-length estimate for {L}agrangian intersections},
        date={2005},
     journal={Journal of Differential Equations},
      volume={209},
      number={1},
       pages={57 \ndash  76},
}

\bib{Mohnkechord}{article}{
      author={Mohnke, Klaus},
       title={Holomorphic disks and the chord conjecture},
        date={2001},
     journal={Annals of Mathematics},
       pages={219\ndash 222},
}

\bib{Loose}{article}{
      author={Murphy, Emmy},
       title={Loose {L}egendrian embeddings in high dimensional contact
  manifolds},
        date={2012},
     journal={arXiv preprint arxiv:1201.2245},
}

\bib{Arborealloose}{article}{
      author={Murphy, Emmy},
       title={Arboreal singularities and loose {L}egendrians {I}},
        date={2019},
     journal={arXiv preprint arXiv:1902.05056},
}

\bib{Nad}{article}{
      author={Nadler, David},
       title={Microlocal branes are constructible sheaves},
        date={2009},
     journal={Selecta Mathematica},
      volume={15},
      number={4},
       pages={563\ndash 619},
}

\bib{NadWrapped}{article}{
      author={Nadler, David},
       title={Wrapped microlocal sheaves on pairs of pants},
        date={2016},
     journal={arXiv preprint arXiv:1604.00114},
}

\bib{NadShen}{article}{
      author={Nadler, David},
      author={Shende, Vivek},
       title={Sheaf quantization in {W}einstein symplectic manifolds},
        date={2020},
     journal={arXiv preprint arXiv:2007.10154},
}

\bib{NadZas}{article}{
      author={Nadler, David},
      author={Zaslow, Eric},
       title={Constructible sheaves and the {F}ukaya category},
        date={2009},
     journal={Journal of the American Mathematical Society},
      volume={22},
      number={1},
       pages={233\ndash 286},
}

\bib{AugSheaf}{article}{
      author={Ng, Lenhard},
      author={Rutherford, Dan},
      author={Shende, Vivek},
      author={Sivek, Steven},
      author={Zaslow, Eric},
       title={Augmentations are sheaves},
        date={2020},
     journal={Geometry \& Topology},
      volume={24},
      number={5},
       pages={2149\ndash 2286},
}

\bib{PolShe}{article}{
      author={Polterovich, Leonid},
      author={Shelukhin, Egor},
       title={Autonomous {H}amiltonian flows, {H}ofer's geometry and
  persistence modules},
        date={2016},
     journal={Selecta Mathematica},
      volume={22},
      number={1},
       pages={227\ndash 296},
}

\bib{Rittertqft}{article}{
      author={Ritter, Alexander},
       title={Topological quantum field theory structure on symplectic
  cohomology},
        date={2013},
     journal={Journal of Topology},
      volume={6},
       pages={391\ndash 489},
}

\bib{MicrolocalInfty}{article}{
      author={Robalo, Marco},
      author={Schapira, Pierre},
       title={A lemma for microlocal sheaf theory in the $\infty $-categorical
  setting},
        date={2018},
     journal={Publications of the Research Institute for Mathematical
  Sciences},
      volume={54},
      number={2},
       pages={379\ndash 391},
}

\bib{AugSheafsurface}{article}{
      author={Rutherford, Dan},
      author={Sullivan, Michael},
       title={Sheaves via augmentations of {L}egendrian surfaces},
        date={2021},
     journal={Journal of Homotopy and Related Structures},
      volume={16},
       pages={703\ndash 752},
}

\bib{Sabduality}{article}{
      author={Sabloff, Joshua~M},
       title={Duality for {L}egendrian contact homology},
        date={2006},
     journal={Geometry \& Topology},
      volume={10},
      number={4},
       pages={2351\ndash 2381},
}

\bib{Genfamily}{article}{
      author={Sabloff, Joshua~M},
      author={Traynor, Lisa},
       title={Obstructions to {L}agrangian cobordisms between {L}egendrians via
  generating families},
        date={2013},
     journal={Algebraic \& Geometric Topology},
      volume={13},
      number={5},
       pages={2733\ndash 2797},
}

\bib{Shendeconormal}{article}{
      author={Shende, Vivek},
       title={The conormal torus is a complete knot invariant},
        date={2019},
     journal={Forum Math. Pi},
      volume={7},
       pages={e6, 16},
         url={https://doi.org/10.1017/fmp.2019.1},
}

\bib{STWZ}{article}{
      author={Shende, Vivek},
      author={Treumann, David},
      author={Williams, Harold},
      author={Zaslow, Eric},
       title={Cluster varieties from {L}egendrian knots},
        date={2019},
     journal={Duke Mathematical Journal},
      volume={168},
      number={15},
       pages={2801\ndash 2871},
}

\bib{STZ}{article}{
      author={Shende, Vivek},
      author={Treumann, David},
      author={Zaslow, Eric},
       title={Legendrian knots and constructible sheaves},
        date={2017},
     journal={Inventiones mathematicae},
      volume={207},
       pages={1031\ndash 1133},
}

\bib{Unbound}{article}{
      author={Spaltenstein, Nicolas},
       title={Resolutions of unbounded complexes},
        date={1988},
     journal={Compositio Mathematica},
      volume={65},
      number={2},
       pages={121\ndash 154},
}

\bib{Tamarkin1}{inproceedings}{
      author={Tamarkin, Dmitry},
       title={Microlocal condition for non-displaceability},
organization={Springer},
        date={2013},
   booktitle={Algebraic and analytic microlocal analysis},
       pages={99\ndash 223},
}

\bib{TrayGen}{article}{
      author={Traynor, Lisa},
       title={Generating function polynomials for {L}egendrian links},
        date={2001},
     journal={Geometry \& Topology},
      volume={5},
      number={2},
       pages={719\ndash 760},
}

\bib{UsherZpersist}{article}{
      author={Usher, Michael},
      author={Zhang, Jun},
       title={Persistent homology and {F}loer-{N}ovikov theory},
        date={2016},
     journal={Geometry \& Topology},
      volume={20},
      number={6},
       pages={3333\ndash 3430},
}

\bib{ViterboGen}{article}{
      author={Viterbo, Claude},
       title={Symplectic topology as the geometry of generating functions},
        date={1992},
     journal={Mathematische Annalen},
      volume={292},
      number={1},
       pages={685\ndash 710},
}

\bib{ViterboIntroSheaf}{article}{
      author={Viterbo, Claude},
       title={An introduction to symplectic topology through sheaf theory},
        date={2011},
     journal={Preprint, available at
  \href{https://www.imo.universite-paris-saclay.fr/~claude.viterbo/Eilenberg/Eilenberg.pdf}{the
  author's personal website}},
}

\bib{Zhang}{book}{
      author={Zhang, Jun},
       title={Quantitative {T}amarkin theory},
      series={CRM Short Courses},
   publisher={Springer, Cham},
        date={[2020] \copyright 2020},
        ISBN={978-3-030-37887-5; 978-3-030-37888-2},
         url={https://doi.org/10.1007/978-3-030-37888-2},
}

\bib{Zhou}{article}{
      author={Zhou, Peng},
       title={Sheaf quantization of {L}egendrian isotopy},
        date={2023},
     journal={Compositio Mathematica},
      volume={159},
      number={2},
       pages={419\ndash 435},
}

\end{biblist}
\end{bibdiv}
\bibliographystyle{amsplain}

\end{document}